\documentclass{amsart}
\usepackage{amssymb}
\usepackage{amsthm}
\usepackage{mathrsfs}
\usepackage{stmaryrd}
\usepackage{enumitem}
\usepackage[french,english]{babel}
\usepackage[utf8]{inputenc}
\usepackage[T1]{fontenc}
\usepackage{lmodern}
\usepackage[all]{xy}
\usepackage[colorlinks=true]{hyperref}
\hypersetup{colorlinks, citecolor=blue, filecolor=black, linkcolor=black, urlcolor=black}
\usepackage{upref}
\usepackage{array}
\usepackage{color}
\usepackage{tikz-cd}
\usepackage{verbatim}
\usepackage{framed}
\usepackage[capitalize]{cleveref}

\def\Qp{\mathbb{Q}_p}
\renewcommand{\P}{\mathbb P}
\def\N{\mathbb{N}}

\def\RG{\rR\Gamma}

\def\FHom{\mathscr{H}om}
\def\FEnd{\mathscr{E}nd}
\def\Ab{\mathbf{Ab}}

\def\Gm{\mathbb{G}_{m}}

\def\cG{\mathcal{G}}

\def\cH{\mathbf{H}}

\def\fc{\mathfrak{c}}
\def\fg{\mathfrak{g}}
\def\ft{\mathfrak{t}}
\def\fh{\mathfrak{h}}

\def\dia{\lozenge}

\def\bX{\mathbb{X}}
\def\BdR{\textnormal{B}_{\textnormal{dR}}}

\def\TT{\mathbb{T}}
\newcommand{\whJ}{\widehat{J}}
\def\sl{\mathfrak{sl}}

\newcommand{\A}{{\mathbf A}}

\DeclareMathOperator{\Res}{Res}
\DeclareMathOperator{\Sect}{Sect}

\DeclareMathOperator{\Et}{\textnormal{\'{E}t}}
\newcommand{\et}{{\textnormal{\'{e}t}}}
\newcommand{\profet}{{\textnormal{prof\'{e}t}}}
\DeclareMathOperator{\Eq}{Eq}
\DeclareMathOperator{\rk}{rank}

\DeclareMathOperator{\can}{can}
\DeclareMathOperator{\ad}{ad}
\DeclareMathOperator{\Gal}{Gal}

\DeclareMathOperator{\Hom}{Hom}
\DeclareMathOperator{\Sym}{Sym}

\DeclareMathOperator{\Ext}{Ext}

\DeclareMathOperator{\Spec}{Spec}
\DeclareMathOperator{\Spa}{Spa}

\DeclareMathOperator{\Spf}{Spf}

\DeclareMathOperator{\cont}{cont}

\newcommand{\rH}{{\mathrm H}}

\DeclareMathOperator{\rW}{W}

\newcommand{\rR}{\mathrm{R}}

\DeclareMathOperator{\dR}{dR}

\DeclareMathOperator{\id}{id}

\newcommand{\alg}{{\mathrm{alg}}}
\newcommand{\an}{{\mathrm{an}}}
\newcommand{\aeq}{{\stackrel{a}{=}}}

\DeclareMathOperator{\GL}{GL}
\DeclareMathOperator{\SL}{SL}
\DeclareMathOperator{\Sp}{Sp}
\DeclareMathOperator{\SO}{SO}

\DeclareMathOperator{\Bun}{Bun}

\DeclareMathOperator{\pr}{pr}

\DeclareMathOperator{\Aut}{Aut}
\DeclareMathOperator{\uAut}{\underline{Aut}}
\DeclareMathOperator{\uEnd}{\underline{End}}
\newcommand{\uHig}{\underline{\mathrm{Hi}}\mathrm{g}}

\DeclareMathOperator{\im}{Im}

\DeclareMathOperator{\Lie}{Lie}

\DeclareMathOperator{\aff}{aff}

\DeclareMathOperator{\Pic}{Pic}

\DeclareMathOperator{\LS}{LS}

\newcommand{\reg}{{\mathrm{reg}}}
\DeclareMathOperator{\Perf}{Perf}
\DeclareMathOperator{\kos}{kos}
\DeclareMathOperator{\HTlog}{HTlog}
\DeclareMathOperator{\HT}{HT}
\DeclareMathOperator{\Exp}{Exp}

\renewcommand{\O}{\mathcal O}
\newcommand{\Z}{\mathbb Z}
\newcommand{\wt}{\widetilde}

\newcommand{\wtOm}{\widetilde{\Omega}}
\newcommand{\Om}{\Omega}
\newcommand{\B}{{\mathcal B}}
\newcommand{\G}{{\mathbb G}}
\newcommand{\cts}{{\mathrm{cts}}}
\newcommand{\uZp}{{\underline{\Z}_p}}
\newcommand{\proet}{\mathrm{pro\acute{e}t}}
\newcommand{\wh}{\widehat}

\newcommand{\bfBun}{\mathbf{Bun}}
\newcommand{\bfHig}{\mathbf{Hig}}
\newcommand{\bfPic}{\mathbf{Pic}}
\newcommand{\bfP}{\mathbf{P}}
\newcommand{\CHig}{\mathscr{H}ig}
\newcommand{\CBun}{\mathscr{B}un}
\newcommand{\CRep}{\mathscr{R}ep}
\newcommand{\CPic}{\mathscr{P}ic}
\newcommand{\CP}{{\mathscr{P}}}
\newcommand{\CH}{{\mathscr{H}}}

\newcommand{\CS}{{\mathscr{S}}}
\newcommand{\cP}{\mathbf P}
\newcommand{\gI}{\mathcal{I}}
\newcommand{\gJ}{\mathcal{J}}

\newcommand*\isomarrow{%
	\xrightarrow{\raisebox{-0.35em}{\smash{\ensuremath{\sim}}}}
}

\numberwithin{equation}{subsection}
\newtheorem{theorem}[equation]{Theorem}
\newtheorem{prop}[equation]{Proposition}
\newtheorem{lemma}[equation]{Lemma}

\newtheorem{coro}[equation]{Corollary}

\theoremstyle{definition}
\newtheorem{example}[equation]{Example}
\newtheorem{rem}[equation]{Remark}

\newtheorem{definition}[equation]{Definition}
\newtheorem{secnumber}[equation]{}

\newcolumntype{L}{>{$}l<{$}} 

\makeatletter
\newcommand*{\relrelbarsep}{.386ex}
\newcommand*{\relrelbar}{%
	\mathrel{%
		\mathpalette\@relrelbar\relrelbarsep
	}%
}
\newcommand*{\@relrelbar}[2]{%
	\raise#2\hbox to 0pt{$\m@th#1\relbar$\hss}%
	\lower#2\hbox{$\m@th#1\relbar$}%
}
\providecommand*{\rightrightarrowsfill@}{%
	\arrowfill@\relrelbar\relrelbar\rightrightarrows
}
\providecommand*{\leftleftarrowsfill@}{%
	\arrowfill@\leftleftarrows\relrelbar\relrelbar
}
\providecommand*{\xrightrightarrows}[2][]{%
	\ext@arrow 0359\rightrightarrowsfill@{#1}{#2}%
}
\providecommand*{\xleftleftarrows}[2][]{%
	\ext@arrow 3095\leftleftarrowsfill@{#1}{#2}%
}
\makeatother

\newcommand{\quash}[1]{}

\title{$p$-adic non-abelian Hodge theory for curves via moduli stacks}
\author{Ben Heuer, Daxin Xu}

\AtEndDocument{\bigskip{\footnotesize
		
		\textsc{Ben Heuer, Institut für Mathematik, Johann Wolfgang Goethe-Universität,	Robert-Mayer-Str.\ 6-8, 60325 Frankfurt am Main, Germany;} \texttt{heuer@math.uni-frankfurt.de} \par

		\textsc{Daxin Xu, Morningside Center of Mathematics and Hua Loo-Keng Key Laboratory of Mathematics, Academy of Mathematics and Systems Science, Chinese Academy of Sciences, Beijing 100190, China;} \texttt{daxin.xu@amss.ac.cn}
}}

\begin{document}
	\selectlanguage{english}
	\maketitle
	
	\begin{abstract}
		For a smooth projective curve $X$ over $\mathbb C_p$ and any reductive group $G$, we show that the moduli stack of $G$-Higgs bundles on $X$ is a twist of the moduli stack of v-topological $G$-bundles on $X_v$ in a canonical way. We explain how a choice  of an exponential trivialises this twist on points. This yields a geometrisation of Faltings' $p$-adic Simpson correspondence for $X$, which we recover as a homeomorphism between the points of moduli spaces. We also show that our twisted isomorphism sends the stack of $p$-adic representations of $\pi_1(X)$ to an open substack of the stack of semi-stable Higgs bundles of degree $0$.
	\end{abstract}
	\section{Introduction}
	Let $K$ be any algebraically closed non-archimedean field over $\mathbb Q_p$. Let $X$ be a connected smooth projective curve over $K$. The starting point of this article is Faltings' $p$-adic Simpson correspondence \cite[Theorem~6]{Fal05}. Following \cite[Theorem~1.1]{Heu23}, this is an equivalence of categories 
	\begin{equation}\label{eq:paS}
		S:\{\text{vector bundles on $X_v$}\}\isomarrow \{\text{Higgs bundles on $X$}\}
	\end{equation}
	depending on the choice of a $B_{\dR}^+/\xi^2$-lift $\bX$ of $X$ and the datum of an exponential for $K$. Here we regard $X$ as an adic space over $\mathbb Q_p$ and $X_v$ is the v-site of the diamond associated to $X$ as defined by Scholze \cite{Sch18}.
	
	The choice of the lift $\bX$ of $X$ can be interpreted as a splitting of a Hodge--Tate sequence, and there is a canonical such choice in arithmetic settings. 
	In contrast, the exponential is a more mysterious datum: It is defined as a continuous splitting of the $p$-adic logarithm $1+\mathfrak m_K\to K$, and there is no canonical such choice.
	
	\subsection{The $p$-adic Simpson correspondence as a twisted isomorphism of moduli spaces}
	
	The goal of this article is to upgrade the $p$-adic Simpson correspondence \eqref{eq:paS} to a ``twisted isomorphism'' between analytic moduli spaces, in a way that explains all choices in a geometric fashion. In order to explain what we mean by this, we first state our main results in the special case of $\GL_n$ in terms of coarse moduli spaces: Let $n\in \N$ and consider the sheaves on the site $\Perf_{K,v}$ of perfectoid spaces over $K$ with the v-topology
	\begin{alignat*}{2}
		\bfBun_{n,v}:&\text{ sheafification of }&& \big(T\mapsto \{\text{v-vector bundles on $X\times T$ of rank $n$}\}/\sim\big),\\
		\bfHig_n:&\text{ sheafification of }&& \big(T\mapsto \{\text{Higgs bundles on $X\times T$ of rank $n$}\}/\sim\big).
	\end{alignat*}
	We can use these to endow the sets of isomorphism classes $\bfBun_{n,v}(K)$ and $\bfHig_n(K)$ of either side in \eqref{eq:paS} with natural topologies, by testing on perfectoid spaces associated to profinite sets. Moreover, \cite[\S1.5]{Heu} has constructed natural ``Hitchin maps'' to the classical Hitchin base $ \A:=\oplus_{i=1}^n\rH^0(X,\Omega_X^{\otimes i}(-i))\otimes_K \G_a$, considered as an adic space:
	\[ 	\wt H: \bfBun_{n,v}\rightarrow  \A\leftarrow \bfHig_n :H\]
	We can now state the first version of our main result: 	Let $\bX$ be a flat $B_{\dR}^+/\xi^2$-lift of $X$.
	\begin{theorem}[\Cref{t:homeom-moduli}]\label{t:intro-1}
		There is a natural Zariski-constructible v-sheaf $\cH_{\bX}\to  \A$ that induces a canonical isomorphism
			\[ \mathbf{ S}:\bfBun_{n,v}\times_{ \A}\cH_{\bX}\isomarrow  \bfHig_{n}\times_{ \A}\cH_{\bX}.\]
			Any choice of exponential for $K$ induces a section $\Exp: \A(K)\to \cH_{\bX}(K)$ that induces a homeomorphism
			\[ \bfBun_{n,v}(K)\isomarrow \bfHig_{n}(K).\]
	\end{theorem}
	This homeomorphism is a close $p$-adic analogue of a Theorem of Simpson in complex geometry, see \S\ref{s:comp-cpx}.
	\newpage
	
	In fact, we can be more precise:
	Let $\pi:Z\to  \A$ be the spectral curve, considered as an adic space. Then $\cH_{\bX}$ is a torsor under $\cP[p^\infty]:=\nu^\ast \rR^1\pi_{\et\ast}\mu_{p^\infty}$ on $\A_v$. Both $\bfBun_{n,v}$ and $\bfHig_{n}$ receive natural $\cP[p^\infty]$-actions.
	\begin{theorem}\label{t:intro-2}
		There is a natural isomorphism of v-sheaves over the Hitchin base $\A$
		\[\cH_{\bX}\times^{\cP[p^\infty]}\bfBun_{n,v}\isomarrow  \bfHig_{n}.\]
	\end{theorem}
	This exhibits $\bfBun_{n,v}$ as a twist of $\bfHig_{n}$, giving a precise technical meaning to the ``twisted isomorphism''. To give a third incarnation, assume that the genus of $X$ is $\geq 2$ and let $ \A^\circ\subseteq  \A$ be the regular locus, i.e.\ the dense Zariski-open subspace where the fibre $\pi:Z^\circ\to  \A^\circ$ of $\pi$ is smooth. Let $\cP:=\mathbf{Pic}_{Z^\circ| \A^\circ}^\dia$ be the v-sheaf associated to the relative Picard functor of $\pi$. Let $\bfBun_{n,v}^\circ$ and $\bfHig^\circ_{n}$ be the fibres of $ \A^\circ$ under $\wt H$ and $H$.
	\begin{theorem}[\Cref{t:moduli-spaces-circ}]\label{t:intro-regular-locus}
		\begin{enumerate}
			\item 
			The morphisms  $H:\bfHig^\circ_{n}\to  \A^\circ$ and $\wt H:\bfBun_{n,v}^\circ\to  \A^\circ$ are $\cP$-torsors. As such, the former is a split $\cP$-torsor, whereas the latter is non-split.
			\item $\bfHig^\circ_{n}$ and $\bfBun_{n,v}^\circ$ are represented by smooth rigid spaces. The restriction $\cH_{\bX| \A^\circ_{\et}}$ is locally constant. 
			\item The fibre over $\A^\circ$ of the isomorphism $\mathbf{S}$  in \Cref{t:intro-1} is an isomorphism of smooth rigid spaces.
		\end{enumerate}
	\end{theorem}
	This explains the first instance of our main results, formulated in terms of coarse moduli spaces.
	But in fact, our setup is more general: Instead of  working with $\GL_n$, i.e.\ with vector bundles, we allow general reductive groups $G$. Moreover, we can work with moduli stacks instead of coarse moduli spaces.
	\subsection{Twisting for the  moduli stack of v-topological $G$-torsors} \label{ss:twisting-action}
	We now explain the main results of this article in full generality and in more detail. 
	Let $G$ be a reductive group over $K$. Generalising from the case of $G=\GL_n$ of vector bundles of rank $n$, we work with the stack of v-$G$-torsors
	\[ \CBun_{G,v}:T\mapsto \{\text{$G$-torsors on $(X\times T)_v$}\}\]
	over $\Spa(K)_v$, defined by
	sending any perfectoid space $T\to \Spa(K)$ to the groupoid of $G$-torsors for the v-topology on the adic space $X\times T$.  Here we interpret $G$ as an adic group over $\Spa(K)$, hence as a v-sheaf.
	
	Based on the preparations from \cite{Heu}, there is also a  good notion of $G$-Higgs bundles on $X\times T$: Let $\wtOm$ be the $(-1)$-Tate twist of the pullback $\pr_X^{\ast}\Omega^1_{X|K}$ along the projection $\pr_X:X\times T\to X$. Then a $G$-Higgs bundle on $X\times T$ is a pair $(E,\theta)$ consisting of a $G$-torsor $E$ on $(X\times T)_\et$ and a section $\theta\in \rH^0(X\times T,\ad(E)\otimes \wtOm)$, where $\ad(E)$
	is the adjoint bundle of $E$. By \cite[Theorem~1.4]{Heu}, the functor fibred in groupoids
	\[ \CHig_{G}:T\mapsto \{\text{$G$-Higgs bundles on $(X\times T)_{\et}$}\}\]
	is then a small v-stack on $\Spa(K)_v$, as is $\CBun_{G,v}$. By \cite[\S1.5]{Heu}, both of these admit a Hitchin morphism 
	\[\CBun_{G,v} \to  \A\leftarrow  \CHig_{G}\]
	of v-stacks, where $ \A$ is the Hitchin base for $G$. The goal of this article is to compare these two morphisms.
	
	Following Ng\^o \cite{Ngo}, the regular centralizer of $G$ induces a commutative relative group scheme
	\[  J\to X\times  \A\]
	which again we may view as a rigid space. It has the fundamental property that for any perfectoid space $T$ and any Higgs bundle $(E,\theta)$ on $X\times T$ with associated morphism $b:=H(E,\theta):T\to  \A$, the base-change $J_b\to X\times T$ admits a natural action on $(E,\theta)$ via a homomorphism of adic groups over $X\times T$
	\begin{equation}\label{eq:intro-J-action-on-Higgs}
		J_b\to \underline{\Aut}(E,\theta).
	\end{equation}
	
	Our first result is that a similar construction is possible on the ``Betti side'', i.e.\ for $\CBun_{G,v}$:
	Let $V$ be a v-topological $G$-torsor on $X\times T$. Generalising a construction of Rodr\'iguez Camargo \cite{camargo2022geometric}, we show that $V$ can be endowed with a canonical Higgs field $\theta_V\in \rH^0(X\times T,\ad(V)\otimes \wtOm)$. We show that $\underline{\Aut}(V)\to X\times T$ is a smooth relative adic group, and that $\theta_V$ induces a canonical homomorphism of adic groups over $X\times T$
	\begin{equation}\label{eq:intro-J-action-on-vG-torsor}
		J_b\to \underline{\Aut}(V).
	\end{equation}
	It follows from this that we can twist both $G$-Higgs bundles and v-topological $G$-torsors on $X\times T$ in the fibre of $b\in  \A(T)$ with $J_b$-torsors. This is what we will use to set up the comparison of moduli stacks.
	
	\subsection{The twisted isomorphism of moduli stacks} \label{ss:twisted-isom}
	The fundamental conceptual idea behind our construction is to use the phenomenon of ``abelianisation'': Roughly, we will reduce the non-abelian Hodge theory of $G$ on $X$ to the relative Hodge theory of the commutative relative group $J\to X\times \A$. 
	
	In the case of $G=\GL_n$, this is related to the BNR-correspondence, which relates Higgs bundles on $X$ to line bundles on the spectral curve $Z$. Indeed, in this case, $J$ is given by $\mathrm{Res}_{Z|X\times \A}\G_m$. If $\pi_T:Z_T\to X\times T$ denotes the base-change along $b:T\to \A$, then $J_b$-torsors are in this case equivalent to line bundles on $Z_T$.
	
	\medskip
	
	To obtain such a relative Hodge theory for $J$, we develop a general theory of smooth relative adic groups over adic spaces like $X\times T$. We then show that the topologically $p$-torsion subsheaf
	$\whJ:=\FHom(\underline{\Z}_p,J)\to J$
	is represented by an open subgroup of $J\to X\times \A$. 
		For any morphism $b:T\to  \A$, we denote the base-change by $\wh J_b\to X\times T$. For $\tau\in \{\et,v\}$, let $\CP_{\tau}\to  \A$
		be the Picard v-stack sending any perfectoid space $b:T\to  \A$ to the groupoid of $\wh J_b$-bundles on $(X\times T)_{\tau}$. We set $\CP:=\CP_{\et}$.
	The canonical $J$-actions in \eqref{eq:intro-J-action-on-Higgs} and \eqref{eq:intro-J-action-on-vG-torsor} define natural $\CP$-actions on both $\CHig_{G}$ and $\CBun_{G,v}$.
	Our fundamental technical result on $J$ is now:
	\begin{theorem}[\Cref{p:leray-seq-for-UJ}]\label{t:fund-ses}
		There is a short exact sequence of Picard stacks on $\Perf_{K,v}$:
		\[1\to \CP\to  \CP_v\to \A_{J,\Omega}\to 0\]
		where $\A_{J,\Omega}$ is the abelian v-sheaf defined by $\pr_{\A,\ast}(\Lie J\otimes \wtOm)$ for the projection $\pr_\A:X\times \A\to \A$.
	\end{theorem}
	Here we can think of $\A_{J,\Omega}$ as being the Hitchin base for the commutative relative adic group $J$, and we can think of the morphism $\CP_v\to \A_{J,\Omega}$ in \Cref{t:fund-ses} as an analogue of the Hitchin morphism for $\whJ$. 
	
	Following Chen--Zhu \cite{CZ15}, there is a canonical section $\tau: \A\to \A_{J,\Omega}$. The key definition is now:
	\begin{definition}
		Let $\CH\to  \A$ be the v-stack defined as the fibre product
		\[\begin{tikzcd}
			\CH \arrow[r] \arrow[d] & \A \arrow[d,"\tau"] \\
			{\CP_v} \arrow[r] & \A_{J,\Omega}.
		\end{tikzcd} \]
	\end{definition}
	As a consequence of \Cref{t:fund-ses},
	there is a natural action of $\CP$ on 	$\CH$ making it a $\CP$-torsor.
	We can now state our main theorem, the $p$-adic Simpson correspondence as a twisted isomorphism of moduli stacks:
	\begin{framed}
		\vspace{-7pt}
	\begin{theorem}[\Cref{t:fCiso}]\label{t:intro-main-thm}
		There is a canonical and functorial equivalence of v-stacks
		\[\CS:\CH\times^{\CP}\CHig_G  \to \CBun_{G,v}.\]
	\end{theorem}
	\vspace{-10pt}
	\end{framed}
	In other words, this exhibits $ \CBun_{G,v}$ as a \textit{twist} of $\CHig_G$ in a natural way.
	We emphasize that this isomorphism does not depend on any choices, in contrast to the more classical formulation of the $p$-adic Simpson correspondence \eqref{eq:paS}. Instead, we can explain these choices in a geometric fashion, as follows.
	\subsection{The constructible sheaf associated to a lift $\bX$, and the exponential}
	To explain the role of the choices in \Cref{t:intro-1}, let us for simplicity switch back to coarse moduli spaces.
	Suppose we are given a $B_{\dR}^+/\xi^2$-lift $\bX$ of $X$. This defines a Faltings extension for $X$ that we can use to define a canonical splitting 
	$s_{\bX}: \A\to \bfBun_{\Lie J,v}$.
	The v-sheaf $\cH_{\bX}$ used in Theorems
	\ref{t:intro-1} and \ref{t:intro-2} is then defined as the fibre product
	\[\begin{tikzcd}
		\cH_{\bX} \arrow[d] \arrow[r] &  \A \arrow[d, "s_{\bX}"] \\
		{\bfBun_{\whJ,v}} \arrow[r, "\log"] & \bfBun_{\Lie J,v}.
	\end{tikzcd}\]
	At the small expense of requiring the additional datum of $\bX$, this gives a ``finer'' comparison between the two moduli spaces: Indeed, we show that $\cH_{\bX}\to  \A$ is a constructible sheaf.

	To explain the role of the exponential, let now $G=\GL_n$ and let $f:Z\to  \A$ be the spectral curve for $\GL_n$. Let 
	$\widehat{\G}_m$ be the open unit disc inside $\G_m$. The key result about the exponential is the following:
	\begin{theorem}[\Cref{c:exp-splits-L_X}] \label{t:intro-exp-splits-L_X}
		Let $\Lambda=\rR^1f_{v\ast}\uZp$ be the \'etale cohomology of the spectral curve.  There is a canonical isomorphism
		$\psi: \rR^1f_{v\ast}\Lie J\isomarrow \Lambda\otimes_{\uZp}\G_a$
	 that induces a Cartesian diagram of v-sheaves on $ \A$
		\[\begin{tikzcd}
			\cH_{\bX} \arrow[r] \arrow[d]     &  \A\arrow[d,"\psi\circ s_{\bX}"] \\
			\Lambda\otimes_{\uZp}\widehat{\G}_m \arrow[r,"\log"] & \Lambda\otimes_{\uZp}\G_a.
		\end{tikzcd}\]
	\end{theorem}
	\begin{coro}
		Let $S=\Spa(R,R^+)$ be any strictly totally disconnected space. By an exponential for $S$, we mean a continuous splitting $\Exp$ of the logarithm $\log:1+R^{\circ\circ}\to R$. Then any exponential for $S$ induces a section of $\cH_\bX(S)\to \A(S)$. In particular, it induces a bijection, natural in $\Exp$,
		\[ \{\text{$G$-torsors on $(X\times S)_v$}\}/\!\sim\; \isomarrow \;  \{\text{$G$-Higgs bundles  on $X\times S$}\}/\!\sim. \]
	\end{coro}
	We deduce \Cref{t:intro-1}.(2) by applying this  to profinite sets $S$ and taking a condensed perspective.
	
	\subsection{Representation variety} \label{ss:representation}
	While v-vector bundles on rigid spaces are now objects of independent interest, Faltings' original motivation for \eqref{eq:paS} was that it gives rise to a natural fully faithful functor 
	\begin{equation}\label{eq:Faltings-functor}
		{\Big\{\begin{array}{@{}c@{}l}\text{ continuous representations}\\\pi_1(X,x)\to \GL_n(K)\end{array}\Big\}}\hookrightarrow  {\Big\{\begin{array}{@{}c@{}l}\text{vector bundles}\\\text{of rank $n$ on $X_v$}\end{array}\Big\}}\xrightarrow[\sim]{S}  {\Big\{\begin{array}{@{}c@{}l}\text{Higgs bundles}\\\text{of rank $n$ on $X$}\end{array}\Big\}}
	\end{equation}
	where $\pi_1(X,x)$ is the \'etale fundamental group of $X$ for some fixed base point $x\in X(K)$.
	Describing the essential image of \eqref{eq:Faltings-functor} is a major open problem in $p$-adic non-abelian Hodge theory. In \cite{Heu22b}, it was solved for $n=1$, and the proof hinged on the geometric study of moduli spaces on either side.
	We therefore believe that the results of this article will help to study the problem for $n>1$. As a step in this direction, we construct a moduli stack $\CRep_{G}$ of continuous $G$-representations of $\pi_1(X,x)$ and an embedding $\CRep_{G}\to \CBun_{G,v}$ which geometrises the first functor in \eqref{eq:Faltings-functor} for $G=\GL_n$. We then prove:
	\begin{theorem}[\Cref{t:profet-locus-open}]\label{t:introCRep-open-substack}
		The natural map $\CRep_{G}\to \CBun_{G,v}$ is an open immersion.
	\end{theorem}
	We deduce that the sought-for condition on Higgs bundles describing the essential image of Faltings' functor cuts out an open sub-v-stack in the moduli space $\CHig_G^{\mathrm{sst},0}$ of semi-stable Higgs bundles of degree $0$ (\Cref{p:image-sst-deg-0}).
	That said, we caution that it is known in the easier special case of $G=\G_m$ that the essential image of the twist of $\CRep_{G}$ is in general strictly smaller than $\CHig_G^{\mathrm{sst},0}$,  see \cite[Theorem~1.1]{Heu22b}.
	\subsection{Comparison to the small $p$-adic Simpson correspondence}\label{s:comp-small}
	In contrast to Faltings' approach to \eqref{eq:paS}, our construction of the twisted isomorphism \Cref{t:intro-main-thm} is ``global'', in the sense that it is not obtained from gluing local isomorphisms. Indeed, the existence of the canonical Higgs field, or equivalently the abelianisation map \eqref{eq:intro-J-action-on-vG-torsor}, is essentially the only place in the proof of \Cref{t:intro-main-thm} where we use input from local non-abelian Hodge theory. In particular, we do not use any analogue of Faltings' ``global $p$-adic Simpson correspondence of small objects'' depending on the choice of an integral lift $\mathbb X$ of $X$ but not on an exponential, as described in \cite[Theorem 5]{Fal05} and studied in detail in \cite{AGT-p-adic-Simpson}. In this sense, we think of this small global correspondence and \Cref{t:intro-main-thm} as two independent results of independent interest. 
	
	That being said, the two kinds of correspondences can be compared to each other when $X$ is proper: Indeed, the small $p$-adic Simpson correspondence can be interpreted as an \textit{isomorphism} of moduli stacks of small objects  \cite[Theorem~1.1]{AHLB-small}.  From our perspective, when  $X$ is a proper curve, this is obtained by restricting the isomorphism \Cref{t:intro-main-thm} to the largest open subdisc $\A^+\subseteq \A$ of the Hitchin base that maps via \Cref{t:intro-exp-splits-L_X} to the open disc of $\Lambda \otimes \G_a$ where the $p$-adic exponential converges. This condition yields a canonical splitting of $\cH_\bX$ over $\A^+$, explaining why no choice of exponential is required in this context.
	
	One interesting aspect of the small $p$-adic Simpson correspondence is that it has an arithmetic counterpart for smooth rigid spaces over discretely valued $p$-adic fields, classifying small v-vector bundles in terms of small ``Higgs--Sen modules'', see \cite[\S2]{LiuZhu_RiemannHilbert}\cite[\S15]{tsuji2018notes}\cite{he2022sen}\cite[\S3]{MinWang22}\cite[\S5]{AHLB-small}. It would be interesting to see if this also has a geometric incarnation in terms of moduli spaces, with the arithmetic Hitchin morphism defined by spectral data of the Sen operator (see  \cite[\S5.5]{AHLB-small}) playing the role of the Hitchin morphism.
	\subsection{Comparison to  non-abelian Hodge theory over $\mathbb C$ and $\mathbb F_p$}\label{s:comp-cpx}
	Following the pioneering work of Faltings \cite{Fal05}, due to its relation to representations of $\pi_1(X,x)$ described in \eqref{eq:Faltings-functor}, the $p$-adic Simpson correspondence \eqref{eq:paS} is regarded as a $p$-adic analogue of the complex non-abelian Hodge correspondence of Corlette and Simpson: We recall that for a smooth projective variety $Y$ over $\mathbb C$, this is an equivalence of categories between the finite dimensional $\mathbb C$-linear representations of $\pi_1(Y)$ and the category of semi-stable Higgs bundles on $Y$ with vanishing rational Chern classes \cite{SimpsonCorrespondence}. Simpson has shown in \cite{SimpsonModuliII} that this induces a homeomorphism between the natural complex analytic moduli spaces on either side, which is however not complex analytic. Moreover, he also gives a generalisation from $\GL_n$ to any reductive $G$.
	
	Not unlike the step from \cite{SimpsonCorrespondence} and \cite{SimpsonModuliII}, the premise of the present article is to study to what extent the $p$-adic Simpson correspondence \eqref{eq:paS} can be understood in terms of moduli spaces. From this perspective, \Cref{t:intro-1}.(2) yields a very close analogue of the complex situation: There are natural $p$-adic analytic moduli spaces for either side, and there is a homeomorphism between $K$-points, but this homeomorphism is not $p$-adic analytic (by \Cref{t:intro-regular-locus}). Surprisingly, however, the situation seems to be better behaved than over $\mathbb C$, as there is additionally a twisted isomorphism  between the two moduli spaces as in \Cref{t:intro-1}.(1) and \Cref{t:intro-main-thm}. These appear to have no analogue in the complex setting.
	
	Instead, \Cref{t:intro-1} is reminiscent of a result in mod $p$ non-abelian Hodge theory due to Groechenig \cite[Theorem 3.29]{Groechenig_modp-Simpson}, while \Cref{t:intro-main-thm} is inspired by the twisted isomorphism between the de Rham stack and the Higgs stack in the mod $p$ theory, due to Chen--Zhu \cite[Theorem 1.2]{CZ15}. 
	This seems to hint at a relation between the moduli theoretic aspects of the $p$-adic and the mod $p$ theory, yet to be discovered.
	
	\medskip
	
	In complex and mod $p$ algebraic geometry, the moduli space of Higgs bundles and the Hitchin fibration are much-studied objects that are part of a very active area of research, for example in the context of the geometric Langlands program, especially the fundamental lemma, or in the context of the P=W conjecture.  It is therefore important to understand the relation of our v-stack $\CHig_{G}$ to its algebraic counterpart, the algebraic stack $\CHig^{\alg}_{G}$ of algebraic Higgs bundles on $X$. 
	We prove that the relation is as close as one could hope for: 
	Let $(-)^{\dia}$ be the diamondification functor from algebraic stacks over $K$ to v-stacks \eqref{eq:diamondification-stacks}. 
	\begin{theorem}[\Cref{t:comparison-Higgs-stacks}]\label{t:intro-comparison-alg}
		There is a canonical isomorphism of v-stacks $(\CHig^{\alg}_{G})^\dia=\CHig_{G}$.
	\end{theorem}
	
	Hence $\CHig_{G}$ and its Hitchin fibration are essentially algebraic objects, and one can use classical results like the BNR-correspondence to study them. This provides the final ingredient for the proof of \Cref{t:intro-regular-locus}.
	
	 On the other hand, this opens up new ways to study the complex moduli space of Higgs bundles from a $p$-adic perspective, where new phenomena like \Cref{t:intro-2} occur.
	Our proof of  \Cref{t:intro-comparison-alg} relies on a  perfectoid GAGA result for vector bundles on curves, \Cref{t:perf-GAGA}, that is of independent interest.
	\subsection{Structure of the article}
	In
	\S\ref{HT-rel-grp}, we introduce the logarithm of commutative relative adic groups and use it to construct  a relative Hodge--Tate sequence. In \S\ref{s:can-Higgs}, we define the canonical Higgs field on v-$G$-bundles via the local $p$-adic Simpson correspondence. 
	In \S\ref{s:abelianisation}, we construct the twisting action on the stacks of $G$-Higgs bundles and v-$G$-torsors (\S\ref{ss:twisting-action}). 
	\S\ref{s:naHC-vstacks} proves \Cref{t:intro-main-thm}, the twisted isomorphism between moduli v-stacks.
	In \S\ref{s:exponentials}, we prove \Cref{t:intro-exp-splits-L_X} and establish the homeomorphism between topological moduli spaces, \Cref{t:intro-1}.
	In \S\ref{s:v-stack-Higgs}, we prove \Cref{t:intro-comparison-alg} and deduce the structure of the Hitchin fibration over the regular locus. Finally,  
	\S\ref{s:representations} is devoted to the proof of \Cref{t:introCRep-open-substack} about the v-stack of representations. 

	\subsection*{Acknowledgments}
	This project started with an exchange of ideas when the first author was visiting Peking University, and we wish to thank Ruochuan Liu for the kind invitation which made this possible.
	
	 The idea that the stack of v-vector bundles could be a twist of the stack of Higgs bundles via the Hitchin map is due to Peter Scholze, and we  thank him heartily for related discussions and comments on an earlier version. We moreover thank Ahmed Abbes, Johannes Ansch\"utz, Bhargav Bhatt, Jingren Chi, Arthur-C\'esar Le Bras, Annette Werner, Mingjia Zhang and Xinwen Zhu for helpful conversations. 
	
	The first author is funded by Deutsche Forschungsgemeinschaft (DFG) -- Project-ID 444845124 -- TRR 326. 
	The second author is supported by National Natural Science Foundation of China Grant (Nos. 12222118, 12288201), CAS Project for Young Scientists in Basic Research, Grant No. YSBR-033. 
	The second author is also supported by Charles Simonyi Endowment for his stay at the Institute for Advanced Study. 

	\section{Notation, conventions and recollections} \label{s:notation}
	\subsection{Setup}\label{s:setup}
		Let $K$ be a non-archimedean field extension of $\mathbb Q_p$. Let $\O_K$ be its ring of integers and let $\mathfrak m_K$ be its maximal ideal. Fix any ring of integral elements $K^+\subseteq \O_K$.
		In this article, by a rigid space over $K$ we mean an adic space locally of finite type over $\Spa(K,K^+)$ in the sense of Huber.
		For any locally of finite type scheme $S$ over $K$, we denote by $S^{\an}$ its rigid analytification, considered as an adic space over $\Spa(K,K^+)$. We will often drop $K^+$ from notation when this is clear from the context.

		Let $G$ be a rigid group over $K$ and $\Lie G$ the Lie algebra of $G$, this is a finite dimensional $K$-vector space. Let $\fg:=(\Lie G\otimes_K \mathbb A^1)^\an$ be the affine space over $K$ associated to $\Lie G$, considered as an adic space.
		
		 From \S\ref{s:abelianisation} onwards, we will assume that $K$ is algebraically closed and that $G$ is connected reductive. 
			
		We use perfectoid spaces in the sense of \cite{perfectoid-spaces}.
		We denote by $\Perf_K$ the category of affinoid perfectoid spaces $T$ over $K$. When we endow it with the v-topology, we obtain a site $\Perf_{K,v}$. For any adic space $Y$ over $K$, Scholze defines an associated locally spatial diamond $Y^\dia$ in \cite[\S15]{Sch18}, which we may regard as a sheaf on $\Perf_{K,v}$. On all categories of adic spaces that we consider, this ``diamondification'' functor $Y\mapsto Y^\dia$ will be fully faithful. We will therefore often switch back and forth freely between adic spaces and their associated diamonds. For any adic space $Y$, let us denote by $Y_v$ the site of perfectoid spaces over $Y$ endowed with the v-topology. We sometimes also consider the ``big \'etale site'' $Y_{\Et}$ consisting of all perfectoid spaces over $Y$ endowed with the \'etale topology.  There is then a natural morphism of sites $\mu:Y_v\to Y_{\Et}$.
		
		Let $Y$ be any locally spatially diamond, then we denote by $Y_\et$ the \'etale site of \cite[Definition 14.1]{Sch18}.
		There exists a canonical morphism of sites $\nu_Y:Y_v\to Y_{\et}$ that we often simply  denote by $\nu$. 
		Given a v-sheaf $F$ on $Y_v$, we sometimes abusively denote $F$ by the \'etale sheaf $\nu_*F$ when this is clear from the context.
		
			\subsection{Twists by torsors}\label{sss:contracted}
		We briefly review the notion of twists and contracted products, and refer to \cite[\S2]{Breen} for some more background. Let $\mathcal C$ be any site and let  $\mathcal{G}$, $\mathcal H$, $\mathcal K$ be any sheaves of groups on $\mathcal C$.
		\begin{definition}
			A (right) $G$-torsor on $\mathcal C$ is a sheaf $E$ on $\mathcal C$ with a right-action $E\times \mathcal{G}\to E$ such that locally on $\mathcal C$, there is a $\mathcal{G}$-equivariant isomorphism $\mathcal{G}\isomarrow E$, where  $\mathcal{G}$ acts on itself by translation on the right.
		\end{definition}
		\begin{definition}\label{d:twist}
			Let $E$ be any $\mathcal{G}$-torsor on $\mathcal C$ and let $V$ be any sheaf on $\mathcal{C}$ with a left-action by $\mathcal G$. The \textit{twist} of $V$ by $E$ is the quotient sheaf $E\times^GV:=(E\times V)/G$ on $\mathcal C$ for the left-action $g\cdot (e,v):=(eg^{-1},gv)$.
		\end{definition}
		In general, the twist does not itself have a $G$-action, unless $E$ is equipped with further structure:
		\begin{definition}
			A \textit{$(\mathcal{G},\mathcal{H})$-bitorsor} is a sheaf $\mathcal{P}$ in $\mathcal{T}$ with a left action of $\mathcal{G}$ and a right action of $\mathcal{H}$, commuting with each other, such that
			$\mathcal{P}$ is an  $\mathcal{H}$-torsor and the $\mathcal{G}$-action $\mathcal{G}\to \uAut_{\mathcal{H}}(\mathcal{P})$ is an isomorphism.
		\end{definition}
		\begin{definition}[{\cite[\S2.3]{Breen}}]\label{d:contracted-prod}
			Let $\mathcal{Q}$ be an $(\mathcal{H},\mathcal{K})$-bitorsor. 
			The twist $\mathcal{P}\times^{\mathcal{H}}\mathcal{Q}$ is also called the \textit{contracted product}.
			The left action of $\mathcal{G}$ on $\mathcal{P}$ and the right action of $\mathcal{K}$ on $\mathcal{Q}$ make this a $(\mathcal{G},\mathcal{K})$-bitorsor in $\mathcal{T}$. 
		\end{definition}
		In general, any $\mathcal G$-torsor $\mathcal P$ is a bi-torsor under $(\uAut(\mathcal{P}),\mathcal{G})$, and if $\mathcal G$ is abelian, we can identify $\uAut(\mathcal{P})$ with $\mathcal{G}$. 
		In particular, if $\mathcal G=\mathcal H=\mathcal K$ are abelian, the contracted product of two $\mathcal G$-torsors is again a $\mathcal G$-torsor.
	
	\subsection{$G$-bundles}\label{sss:Gbundle}
		
		Let $Y$ be a sousperfectoid space and $\tau\in \{\et,v\}$. We regard $G$ as a diamond over $\Spa(K)$, hence it represents a sheaf on $Y_\tau$. There are then two equivalent notions of $G$-bundles on $Y_\tau$: 
		\begin{definition}
			\begin{enumerate}
			\item 
		By a (cohomological) $\tau$-$G$-bundle on $Y_\tau$, we shall mean a right $G$-torsor on $Y_\tau$: Explicitly, this is a sheaf $E$ on $Y_\tau$ with a right $G$-action $E\times_YG\to E$ such that $\tau$-locally on $Y$, there is a $G$-equivariant isomorphism $G\isomarrow E$. 
		We also refer to \'etale $G$-bundles simply as $G$-bundles.
		\item 
		On the other hand, a geometric $\tau$-$G$-bundle on $Y$ is a morphism of v-sheaves $E\to Y$ on $\Perf_K$ with a left $G$-action $E\times G\to E$ over $Y$ such that there is a $\tau$-cover $Y'\to Y$ in $Y_\tau$ over which there is a section $Y'\to E$ that induces a $G$-equivariant isomorphism $G\times Y'\xrightarrow{\sim} E\times_Y Y'$. 
		\end{enumerate}
		\end{definition}
		\begin{prop}[{\cite[Theorem 19.5.1]{ScholzeBerkeleyLectureNotes}, \cite[\S3.3]{heuer-G-torsors-perfectoid-spaces}}]
		For any geometric $\tau$-$G$-bundle $E\to Y$, the sheaf of sections $Y\to E$ over $Y$ is a (cohomological) $\tau$-$G$-bundle, inducing an equivalence of the two notions. 
		\end{prop}
		
		\begin{definition}\label{d:ad-bundle}
		For any  $G$-bundle $E$ on $Y_{\tau}$, we denote by $\ad(E)$ the adjoint bundle associated to $E$, defined as the vector bundle $\ad(E):=E\times^{G}\mathfrak g$ on $Y_v$ where $\mathfrak g$ is the Lie algebra with its adjoint action by $G$.
		\end{definition}

	\subsection{Smoothoid spaces}\label{s:rel-adic-groups}
	Next, we recall smoothoid spaces from \cite[\S 2]{Heu}.
	For any $d\in \N$, let $\TT^d:=\Spa(K\langle T_1^{\pm 1},\cdots, T_d^{\pm 1}\rangle)$ be the $d$-dimensional affinoid torus. Let $\mathbb B^d_S$ be the closed unit ball over $K$.
	Let $S$ be a sousperfectoid space over $K$. There is a good notion of smooth morphisms over $S$ by \cite[Def 1.7.10]{huber2013etale}: 
	\begin{definition}
		Let $h:X\to S$ be a morphism of adic spaces over $K$.
		\begin{enumerate}
			\item We say that $h$ is \textit{standard-\'etale} if $X$ and $S$ are affinoid and $h$ is the composition of finite \'etale maps with rational open immersions. We say that $h$ is \textit{\'etale} if locally on $X$ and $S$, it is standard-\'etale.
		\item We say that  is \textit{standard-smooth} if there exists a factorisation $X\xrightarrow{f} \mathbb T^d_S\to S$ for some $d$, where $f$ is standard-\'etale. We say that $h$ is \textit{smooth} if locally on $X$ and $S$, it is standard-smooth.
		\end{enumerate}
	\end{definition} 
	Note that for any smooth morphism $X\to S$ over a sousperfectoid space $S$, the space $X$ is itself sousperfectoid, hence sheafy. This also shows that the category of smooth adic spaces over $S$ has fibre products.
	
	\begin{definition}[{\cite[Definition 2.3]{Heu}}]\label{d:smoothoid}
		Let $Y$ be an adic space over $K$. A \textit{toric chart of $Y$} is a standard-\'etale morphism $f:Y\to \TT^d\times T$, where $T$ is an affinoid perfectoid space over $K$. We call $Y$ \textit{toric} if it admits a toric chart.
		We call an adic space $Y$ over $K$ \textit{smoothoid} if it admits an open cover by toric spaces.
	\end{definition}
	The prototypical example of a smoothoid space  is the product $X\times T$ of a smooth rigid space with a perfectoid space, or any object of its \'etale site. 
	We recall some basic properties of smoothoid spaces that we will need throughout, and refer to \cite[\S2]{Heu} for proofs and more details:
	Smoothoid spaces are sousperfectoid, in particular sheafy, and are moreover diamantine in the sense of \cite[\S 11]{HK}. In particular, the functor from smoothoid spaces over $K$ to diamonds over $K$ is fully faithful and identifies structure sheaves and \'etale sites. This will allows us to pass back and forth freely between a smoothoid space $Y$ and its associated diamond $Y^\dia$.
	Finally, there is a good notion of differentials on smoothoid adic spaces:
	\begin{definition}\label{d:differentials-smoothoids}
		Let $Y$ be a smoothoid space and let $\nu:Y_v\to Y_{\et}$ be the natural morphism of sites. Then \[\widetilde{\Omega}_Y:=\rR^1\nu_* \O_Y,\quad \widetilde{\Omega}^n_Y:=\rR^n\nu_* \O_Y\simeq\wedge^n\widetilde{\Omega}_Y\] are vector bundles on $Y_\et$ for any $n\in \N$. If $Y=X\times T$ for a smooth rigid space $X$ over $K$ and a perfectoid space $T$, then there is a canonical isomorphism $\widetilde{\Omega}_Y=\pi_1^{\ast}\Omega^1_X(-1)$, where $\pi_1:X\times T\to X$ is the projection.
	\end{definition}
		\subsection{Higgs bundles on smoothoid spaces}
	The intrinsic notion of differentials on smoothoid spaces allows us to define Higgs bundles in this setting: Recall that $G$ is any rigid group over $K$. Let $\tau\in \{\et,v\}$.
	\begin{definition}
		A \textit{$\tau$-$G$-Higgs bundle} on a smoothoid $Y$ is a pair $(E,\theta)$ of a $G$-bundle $E$ on $Y_\tau$ and a section $\theta\in \rH^0(Y, \ad(E)\otimes_{\O_Y} \widetilde{\Omega}_Y)$ such that $\theta\wedge \theta=0$, where $\ad(E)$ is the adjoint bundle (\Cref{d:ad-bundle}).
	\end{definition}
	
	\begin{secnumber}\label{s:Higgs-bundle-curve}
		We note that if $X$ is a smooth curve and $Y$ is an object of the \'etale site of $X\times T$ for some perfectoid space $T$, then $\widetilde{\Omega}_Y=\pi_1^{\ast}\Omega^1_X(-1)$ is a line bundle, rendering the Higgs field condition $\theta\wedge \theta=0$ vacuous for the definition of Higgs bundles on $Y$. Throughout most of this article, this is the situation that we encounter. 
	\end{secnumber}
	\begin{definition}\label{d:Aut-Higgs}
		Let $(E,\varphi)$ be a $\tau$-$G$-Higgs bundle on a smoothoid $Y$. We write $\uAut(E,\varphi)$ for the presheaf  
		on $Y_{v}$ defined by
		$
		U\mapsto \Aut_U( (E,\varphi)|_U)$, the automorphisms of $E$ over $U$ preserving $\varphi$.
		 Here, as usual, even if $E$ is an \'etale torsor, we identify $E$ with the v-sheaf represented by its , so this is always a v-sheaf of groups.
	\end{definition}
	\subsection{Background on Picard stacks}\label{s:Picard-stacks}
	From \S\ref{s:naHC-vstacks} onwards, we will freely use the notion of  (strictly commutative) Picard stacks on the v-site $Y_v$ of an adic space $Y$, or in fact, any v-stack $Y$. We briefly recall all the relevant definitions and refer to \cite[\S1.4]{Del73}\cite[\S A]{CZ17} for more details:
	\begin{definition}
	Let $Y$ be any v-stack. A \textit{Picard stack} $\CP$ on $Y_v$ is a v-stack $\CP\to Y$ with a bi-functor
	\[ \otimes:\CP\times_Y\CP\to \CP\]
	and the datum of natural equivalences of functors expressing the associativity and commutativity of $\otimes$,
	such that for every $U\in Y_v$, the bi-functor $\otimes$ turns $\CP(U)$ into a symmetric monoidal category in which every object admits an inverse with respect to $\otimes$. The symmetric monoidal structure means in particular that $\CP$ has a unit section $e:Y\to\CP$. There is an obvious notion of homomorphisms between Picard stacks. 
	\end{definition}
	An example is the v-stack of $\mathcal G$-torsors on $Y_v$ for an abelian sheaf $\mathcal G$ on $Y_v$, with $\otimes$ the contracted product.
	\begin{definition}[{\cite[\S A.2]{CZ17}}]\label{d:picstack-ses}
		A sequence of homomorphism of Picard stacks
		$\CP_1\xrightarrow{f} \CP_2\xrightarrow{g} \CP_3$
		on $Y_v$ is called \textit{left-exact} if $g\circ f\simeq e$ is trivial and the natural maps induce an equivalence of Picard stacks $\CP_1\simeq e\times_{\CP_3} \CP_2$. The sequence is called \textit{exact} if moreover $g$ is essentially surjective locally on $Y_v$.
	\end{definition}
	\begin{definition}\label{d:picstack-action}
		Let $\CP$ be a Picard stack on $Y_v$ and let $\mathcal F$ be any v-stack on $Y_v$. Then an \textit{action} of $\CP$ on $\mathcal F$ is the datum of a bi-functor
		\[ \otimes:\CP\times_Y\mathcal F\to \mathcal F\]
		together with  a natural equivalence of functors expressing the associativity, and an equivalence $e\otimes -\simeq \id_{\mathcal F}$.
		\end{definition}
		\begin{definition}\label{d:picstack-torsor}
		A $\CP$-\textit{torsor} is a v-stack $\mathcal F$ over $Y_v$ with a $\CP$-action satisfying the following properties: (i) every $U\in Y_v$ admits a cover $V\to U$ such that $\mathcal F(V)\neq \emptyset$, and (ii) the following functor is an equivalence:
		\[ \CP\times_Y\mathcal F\to \mathcal F\times_Y\mathcal F,\quad (P,F)\mapsto (P\otimes F,F).\]
		This is clearly equivalent to the definition of $\CP$-torsors given in \cite[\S A.5]{CZ17}.
	\end{definition}
	\begin{lemma}\label{d:picstack-ses-induces-torsor}
		Let $\CP_1\xrightarrow{f} \CP_2\xrightarrow{g} \CP_3$ be a short exact sequence of Picard stacks. Let $s:Y\to \CP_3$ be any section and let $\CH:=Y\times_{s,\CP_3}\CP_2$
		be the fibre of $g$ over $s$. Then the restriction of the group structure on $\CP_2$ to $\CP_1\times \CH$ makes $\CH$ into a $\CP_1$-torsor.
	\end{lemma}
	\begin{proof}
		It is clear that $f$ endows $\CP_2$ with a natural $\CP_1$ action that restricts to an action on $\CH$. For this action,
		condition (i) of \Cref{d:picstack-torsor} holds because $g$ is essentially surjective. Condition (ii) holds because the natural morphism $(e\times_{\CP_3} \CP_2)\times_{Y}\CH\to \CH\times_{Y}\CH$ has a quasi-inverse given by $(F_1,F_2)\mapsto (F_1\otimes F_2^{-1},F_2)$.
	\end{proof}
	We then have the following 2-categorical version of \Cref{d:twist}. Here we note that since $\CP$ is commutative, we are allowed to be less careful about whether we consider left actions vs right actions.
	\begin{definition}\label{d:picstack-twist}
		Let now $\CP$ be a Picard stack on $Y_v$, let $\mathcal F$ be a v-stack with a $\CP$-action, and let $\CH$ be a $\CP$-torsor. Then we can construct the \textbf{twist} $\CH\times^{\CP}\mathcal F$ of $\mathcal F$ by $\CH$: Consider the antidiagonal action of $\CP$ on the product $\CH\times \mathcal F$, i.e.\ the action on the second factor is via $[-1]:\CP\to \CP$. Following \cite[p420]{Ngo}, we can form the prestack quotient $[(\CH\times \mathcal F)/\CP]^{\mathrm{pre}}$ by this action: Indeed, for each $U\in Y_v$, we can form the 2-quotient $[\CH(U)\times \mathcal F(U)/\CP(U)]$. This is a priori a 2-category, but 
		since $\CH$ is a $\CP$-torsor, it is in fact a 1-category (cf.\ \cite[Lemma 4.7]{Ngo}).  By stackifying $[(\CH\times \mathcal F)/\CP]^{\mathrm{pre}}$, we thus obtain a v-stack 
		\[ \CH\times^{\CP}\mathcal F:=[(\CH\times \mathcal F)/\CP].\]
		This is a twist of $\mathcal F$ in the sense that over any v-cover of $Y$ where $\CH$ becomes $\simeq \CP$, it is equivalent to $\mathcal F$.
	\end{definition}
	\section{Hodge--Tate theory for commutative relative adic groups}\label{HT-rel-grp}
	A general theme of this article is that of ``abelianisation'', which means to study the Hodge theory of non-abelian group varieties by relating it to the relative  Hodge theory of families of commutative group varieties. The goal of this section is to establish the necessary foundations on relative rigid group varieties.
	
	Namely, our first aim in this section is to generalize a result of Fargues on the logarithm map of a rigid group \cite[Th\'eor\`eme 1.2]{Far19} to the relative setting of smooth rigid groups in sousperfectoid families.
	Second, we use this to prove a relative version of the $\HTlog$ exact sequence from \cite[Theorem~1.3.1]{HeuSigma}.

	\begin{definition}\label{d:smooth-rel-group}
		Let $K$ be any non-archimedean field extension of $\Qp$.
		Let $S$ be any sousperfectoid adic space or any rigid space over $K$. A \textit{smooth relative group over $S$} is a group object $\mathcal{G}\to S$ in the category of  smooth morphisms over $S$. We also call this a smooth $S$-group. 
		By an open subgroup of $\mathcal{G}$ we mean an open adic subspace $U\subseteq \mathcal{G}$ that has an $S$-group structure so that $U\subseteq \mathcal{G}$ is a homomorphism of smooth $S$-groups.
	\end{definition}
	
	\subsection{Local structure near the identity of smooth relative adic groups }
	Throughout the following, we fix an adic space $S$ over $K$ as in \Cref{d:smooth-rel-group}.
	\begin{lemma}\label{l:zero-section}
		Let $f:\mathcal{G}\to S$ be a smooth relative group. Then locally on $S$, there exists an open subspace $U\subseteq  \mathcal{G}$ containing the image of the identity section $e:S\to U$ with an isomorphism of adic spaces $U\cong \mathbb B^d_S$ over $S$ that identifies $e$  with the origin $0:S\to \mathbb B^d_S$.
	\end{lemma}
	\begin{proof}
		Let $x$ be any point in the image of the identity section $e:S\to \mathcal{G}$. Then locally on $S$ there is a quasi-compact open neighbourhood $x \in V$ for which $f|_V$ admits a factorisation $V\xrightarrow{g} \mathbb B_S^d\to S$ such that $g$ is \'etale and quasi-compact. Replacing $S$ by the preimage of $V\subseteq \mathcal{G}$ under $e$, we may assume that $e$ factors through $V$. Let $y:S\xrightarrow{e}V\xrightarrow{g} \mathbb B_S^d$ be the composition. After replacing $V\to \mathbb B_S^d$ with its base-change along the isomorphism $\mathbb B_S^d\xrightarrow{+y}\mathbb B_S^d$, we can assume that $S\xrightarrow{e}V\xrightarrow{g} \mathbb B_S^d$ coincides with the origin $0\in \mathbb B_S^d$.
		
		Consider now the inverse system of closed balls $\mathbb B_{S,\epsilon}^d\subseteq \mathbb B_S^d$ of radius $\epsilon$. We have $S=\varprojlim_{\epsilon\to 0} \mathbb B_{S,\epsilon}^d$ in the category of diamonds, and hence 
		\[ S_{\mathrm{\acute{e}t}\text{-}\mathrm{qcqs}}=2\text{-}\mathrm{colim}_{\epsilon\to 0} (\mathbb B_{S,\epsilon}^d)_{\mathrm{\acute{e}t}\text{-}\mathrm{qcqs}}\]
		by \cite[Prop.~14.9]{Sch18}.
		Here on the right, the transition maps are given by base-change. In particular, the \'etale map $V\to \mathbb B_{S}^d$ gives rise to the system of base-changes $V_\epsilon:=V\times_{\mathbb B^d_S}\mathbb B_{S,\epsilon}^d$. 
		Considering the splitting $S\xrightarrow{e} V\times_{\mathbb{B}^d_S,0} S\to S$ inside $S_{\mathrm{\acute{e}t}\text{-}\mathrm{qcsep}}$, we deduce that  it extends to a splitting $\mathbb B_{S,\epsilon}^d\xrightarrow{e_{\epsilon}} V_\epsilon\to \mathbb B_{S,\epsilon}^d$ for some $\epsilon>0$. Here the map $e_{\epsilon}:\mathbb B_{S,\epsilon}^d\to  V_\epsilon$ is a closed immersion, but also \'etale by \cite[Prop.~1.6.7.(iii)]{huber2013etale}, in particular open. Hence the image $U$ of $e_\epsilon$ is an open subspace of $V_\epsilon\subseteq \mathcal{G}$ isomorphic to $\mathbb B_{S,\epsilon}^d$ via $e_{\epsilon}$. As $\varprojlim_{\epsilon\to 0}e_{\epsilon}=e$ by construction, $U$ still contains $x$. We thus get the desired open subspace of $\mathcal{G}$.
	\end{proof}
	\begin{lemma}\label{l:open-nbhd-of-origin-contains-ball}
		Assume that $S$ is quasi-compact and let $U\subseteq \mathbb B_{S}^d$ be an open subspace that contains the image of the origin $0:S\to\mathbb B_{S}^d$. Then  $U$ contains $\mathbb B_{S,\epsilon}^d$ for some $\epsilon>0$.
	\end{lemma}
	\begin{proof}
		Using that $S$ is quasi-compact,
		we reduce to the case that $S$ is affinoid and $U$ is rational, defined as the locus inside $\mathbb B_{S}^d=\Spa(R\langle X_1,\dots,X_d\rangle)$  where $|f_i(x)|\leq |g(x)|\neq 0$ for some elements $f_1,\dots,f_r,g\in R^\circ\langle X_1,\dots,X_d\rangle$. For this the statement can be seen from the usual calculation: Write $f_i=\sum_{n\in \mathbb Z_{\ge 0}^d} a_{n,i}X^n$ and $g=\sum_{n\in \mathbb Z_{\ge 0}^d} b_{n}X^n$, then the condition that $0\in U$ implies that $b_0\in R^\times$ and $|a_{0,i}(x)|\leq |b_0(x)|$ for all $x\in S$. 
		Then for $\epsilon>0$ small enough, we have for any $x\in S$ that $\epsilon\leq |b_0(x)|$ and hence
		\[\textstyle |f_i(x)|\leq \max(|a_{0,i}(x)|,|\sum_{n\neq 0} a_{n,i}(x)\epsilon^n|)\leq |b_0(x)| = \max(|b_0(x)|,|\sum_{n\neq 0} b_{n}(x)\epsilon^n|)=|g(x)|.\]
		This shows that $\mathbb B_{\epsilon,S}^d\subseteq U$.
	\end{proof}
	\begin{lemma}\label{l:existence-of-subgroup-isom-to-ball}
		Let $f:\mathcal{G}\to S$ be a smooth $S$-group. Then locally on $S$, there exists an open $S$-subgroup $U\subseteq  \mathcal{G}$ and an isomorphism of adic spaces $U\cong \mathbb B_S^d$ over $S$ identifying $e:S\to U$ with the origin $0:S\to \mathbb B^d_S$.
	\end{lemma}
	\begin{proof}
		We may assume that $S=\Spa(R,R^+)$ is affinoid.
		By \Cref{l:zero-section}, there exists an open sub\textit{space} $U\subseteq \mathcal{G}$ containing the image of the identity section $e:S\to \mathcal{G}$ such that $U\cong \mathbb B_{S}^d$. It follows that $U\times_SU\cong \mathbb B_{S}^{2d}$ is an open neighbourhood of the identity in $\mathcal{G}\times_S\mathcal{G}$. Let $m:\mathcal{G}\times_S \mathcal{G}\to \mathcal{G}$ be the multiplication map, then $V:=m^{-1}(U)\cap U\times_S U\subseteq \mathcal{G}\times_S\mathcal{G}$ is an open subspace of $U\times_SU$ containing the image of the identity section. By Lemma~\ref{l:open-nbhd-of-origin-contains-ball}, it follows that $V\subseteq U\times_SU\cong \mathbb B_{S}^{2d}$ contains an open ball $\mathbb B_{S,\epsilon}^{2d}$ for some $0<\epsilon\leq 1$ with $0$ corresponding to the identity. Hence the multiplication map restricts to a map of open subspaces
		\[m:\mathbb B_{S,\epsilon}^{d}\times \mathbb B_{S,\epsilon}^{d}\to \mathbb B_{S}^{d}.\]
		On global sections, this is given in terms of coordinates $X,Y,Z$ by $Z_i\mapsto f_i=\sum_{n,k\in \mathbb Z_{\ge 0}^d} a_{n,k,i}X^nY^k$
		for some $a_{n,k,i}\in R$. As $m(0,-)$ and $m(-,0):\mathbb B_{S,\epsilon}^{d}\to \mathbb B_{S}^{d}$ are both the inclusion by construction, we know that the low degree terms of $f_i$ are
		\[ f_i=X_i+Y_i+ \text{[higher terms]}.\]
		It follows from this that for $n>0$ large enough, the function $p^{-n}f_i(p^{n}X_1,\dots,p^{n}X_d,p^{n}Y_1,\dots,p^{n}Y_d)$ is contained in $R^\circ\langle X_1,\dots,X_d,Y_1,\dots,Y_d\rangle$. This means that for $\epsilon>0$ small enough, the map $m$ restricts to 	\[m:\mathbb B_{S,\epsilon}^{d}\times \mathbb B_{S,\epsilon}^{d}\to \mathbb B_{S,\epsilon}^{d}.\]
		Thus the image of $\mathbb B_{S,\epsilon}^{d}\subseteq \mathbb B_{S}^{d}\cong U\subseteq \mathcal{G}$ is an open $S$-subgroup with the desired properties.
	\end{proof}
	This also shows that one can define for any smooth $S$-group its Lie algebra $\Lie \mathcal{G}$, a vector bundle on $S$.
	
	\subsection{The logarithm for commutative relative adic groups}
	From now on, we assume $\mathcal{G}$ is commutative. 
	\begin{prop}\label{p:nbhd-open-subgroup-ball}
		Let $S$ be a sousperfectoid adic space or a rigid space over $K$. 
		Let $\mathcal{G}\to S$ be a commutative smooth $S$-group. Then locally on $S$, there is an $S$-subgroup $U\subseteq \mathcal{G}$ and an isomorphism of smooth $S$-groups $U\xrightarrow{\sim} (\mathbb B_{S}^d,+)$ that on completions at the identity section is given by the formal Lie group logarithm.
	\end{prop}
	\begin{proof}
		By Lemma~\ref{l:existence-of-subgroup-isom-to-ball}, we may assume without loss of generality that $\mathcal{G}\cong \mathbb B^d_S$ as an adic space over $S$. The completion of $\mathcal{G}$ at the identity is then a formal $S$-group scheme 
		\[\mathcal{G}^{\wedge}_{|e}=\Spf(R[[X_1,\dots,X_d]])\]
		with group structure given by a formal group law $F_1(X,Y),\dots,F_d(X,Y)\in R[[X_1,\dots,X_d,Y_1,\dots,Y_d]]$.  Any such formal $S$-group is isomorphic to the additive formal group $(\mathbb G_a^{\wedge})^d$ via the logarithm map: More precisely, following the proof of \cite[Proposition 18.16]{Sch}, there are formal power series $\Phi_i(Z)=\sum \omega_{n,i} \frac{Z^n}{n!}\in R[[Z_1,\dots,Z_d]]$ such that 
		\[\Phi(0)=0,\quad \Phi(Y+Z)=F(\Phi(Y),\Phi(Z)).\]
		Moreover, it is shown in \text{loc.\ cit.\ }that  $\|\omega_{n,i}\|\leq \|\Phi'_i(0)\|^n$. It follows that there is $k>0$ such that $\Phi_i(p^kZ)\in R^\circ\langle Z_1,\dots,Z_d\rangle$, which means that $\Phi_i$ induces a homomorphism $(\mathbb B_{\epsilon}^d,+)\to \mathcal{G}$ for any $\epsilon<1/k$. Since the underlying morphism of adic spaces is of the form $\mathbb B_{S,\epsilon}^d\to \mathbb B_S^d$, sending $0$ to $0$, and is given by the identity on tangent spaces, it is automatic that this restricts to an isomorphism $\mathbb B_{S,\epsilon}^d\to \mathbb B_{S,\epsilon}^d$.
	\end{proof}
	
	The goal of this subsection is to give a more canonical and functorial way to describe an open subgroup $U$ of $\mathcal{G}$ related to the Lie algebra. Namely, we will show that there is a maximal open subgroup on which the logarithm converges. For its description, we need some further preparations:
	\begin{lemma}\label{c:G[p]-etale}
		Assume that $S$ is a smooth rigid space.
		Let $f:\mathcal{G}\to S$ be a commutative smooth $S$-group. Then  $[p]:\mathcal{G}\to \mathcal{G}$ is \'etale. In particular, for any $n\in \N$, the morphism $\mathcal{G}[p^n]\to  S$ is \'etale.
	\end{lemma}
	\begin{proof}
		We first note that the result holds if $S=\Spa(L,L^+)$ where $L$ is a field: By \cite[Proposition 1.7.5]{Huber-generalization}, we can immediately reduce to $L^+=\O_L$, and in this case the result holds by \cite[Lemme~1]{Far19}.
		
		Since $S$ and $f$ are smooth, $\mathcal{G}$ is smooth over $K$. Since $[p]:\mathcal{G}\to \mathcal{G}$ is \'etale in each fibre of $S$, we can now argue by miracle flatness (\cite[00R4]{StacksProject}) that $[p]$ is flat.
		By \cite[Proposition 1.7.5]{Huber-generalization}, it therefore suffices to check on $\Spa(K,\O_K)$-points of $S$ that $[p]:\mathcal{G}\to \mathcal{G}$ is \'etale. Here we have already seen the statement.
	\end{proof}
	
	We now generalise the notion of topological $p$-torsion subgroups of \cite[\S2.2]{Heu22b} to the relative setup:
	
	\begin{definition}
		A subspace $T\subseteq \mathcal{G}$ of a smooth $S$-group $\mathcal{G}$ is \textit{topologically $p$-torsion} if for any open subspace $U\subseteq \mathcal{G}$ with $\im(e)\subseteq U$ and any quasi-compact open subspace $T_0\subseteq T$, there is $n\in \N$ such that \[[p^n](T_0)\subseteq U.\]
	\end{definition}
	\begin{prop}\label{p:log-of-smooth-group}
		Let $S$ be a sousperfectoid adic space or a rigid space over $K$. 
		Let $\mathcal{G}\to S$ be a commutative smooth $S$-group. Then there exists a unique maximal  topologically $p$-torsion open subgroup $\widehat{\mathcal{G}}\subseteq \mathcal{G}$ and a unique homomorphism
		\[ \log_{\mathcal G}:\widehat{\mathcal{G}}\to \Lie \mathcal{G}\otimes_{\mathcal O_S} \mathbb G_a\]
		into the smooth $S$-group defined by $\Lie \mathcal{G}$ such that $\log_{\mathcal{G}}$ induces the identity on Lie algebras. Moreover:
		\begin{enumerate}
			\item We have $\ker \log_{\mathcal{G}} = \mathcal{G}[p^\infty]:=\varinjlim_n \mathcal{G}[p^n]$, the $p$-power torsion subgroup of $\mathcal{G}$.
			\item If $S$ is a smooth rigid space, then $\log_{\mathcal{G}}$ is \'etale, and its image  is an open subgroup of $\Lie \mathcal{G}\otimes_{\mathcal O_S} \mathbb G_a$.
			\item\label{i:log-ses} If $[p]:\mathcal{G}\to \mathcal{G}$ is surjective, then $\log_{\mathcal{G}}$ gives rise to a short exact sequence of smooth $S$-groups
			\[ 0\to \mathcal{G}[p^\infty]\to \widehat{\mathcal{G}}\xrightarrow{\log_{\mathcal{G}}}\Lie \mathcal{G}\otimes_{\mathcal O_S} \mathbb G_a\to 0\]
			for the \'etale topology. More generally, we still obtain such a short exact sequence if there exists an open subgroup $U\subseteq \mathcal{G}$ on which $[p]:U\to U$ is surjective.
			\item The assignment $\mathcal{G}\mapsto (\widehat{\mathcal{G}},\log_{\mathcal{G}})$ is functorial in $\mathcal{G}\to S$.
			\item The formation of $\widehat{\mathcal{G}}$ commutes with base-change: If $S'\to S$ is any morphism of adic spaces with $S'$ sousperfectoid or rigid, then $\mathcal{G}':=\mathcal{G}\times_SS'\to S'$ is a smooth $S'$-group and $\widehat{\mathcal{G}'}=\widehat{\mathcal{G}}\times_SS'$.
			\item In the category of sheaves on $S_v$, the evaluation at 1 defines a natural isomorphism
			\[ 	\FHom_S(\uZp, \mathcal{G})=\widehat{\mathcal{G}}.\]
		\end{enumerate}
	\end{prop}
	
	\begin{definition}
		We also call the group $\widehat{\mathcal{G}}$ in \Cref{p:log-of-smooth-group} the \textit{topologically $p$-torsion subgroup} of $\mathcal{G}$.
	\end{definition}
	\begin{proof}
		All statements are local on $S$, so we may without loss of generality replace $S$ by an open cover. Let $U\subseteq \mathcal{G}$ be an open subgroup as described in Prop~\ref{p:nbhd-open-subgroup-ball}. Then we claim that the open subgroup
		\[ \textstyle\widehat{\mathcal{G}}:=\bigcup_{n\in \N}[p^n]^{-1}(U)\]
		has all desired properties. It is clear that $U$ is topologically $p$-torsion, hence so is $[p^n]^{-1}(U)$ and therefore $\widehat{\mathcal{G}}$. On the other hand, it is clear that $\widehat{\mathcal{G}}$ contains every topologically $p$-torsion subgroup.
		
		 Next, we construct $\log_{\mathcal{G}}$: Since $\Lie \mathcal{G}\otimes_{\O_S}\mathbb G_a\to S$ is a uniquely $p$-divisible $S$-group (i.e., $[p]$ is an isomorphism), there is a unique way to extend $\log:U\to (\mathbb B^d,+)$ to a map $\log_{\mathcal{G}}:\widehat{\mathcal{G}}\to \mathbb G_a^d$. By the description of the completion at the origin in \Cref{p:nbhd-open-subgroup-ball}, we see that this becomes canonical and functorial if we identify the image with $\Lie(\mathcal{G})\otimes_{\O_S}\mathbb G_a$. 
		 Part (4) is then clear from the construction. 
		
		It is clear from the fact that $\log_{|U}$ is injective that $\ker \log_{\mathcal{G}}=\mathcal{G}[p^\infty]$. This shows (1). For (2), it follows from $[p]:\mathcal{G}\to \mathcal{G}$ being \'etale by \Cref{c:G[p]-etale} that $\log_{\mathcal{G}}:[p^n]^{-1}(U)\to \Lie(\mathcal{G})\otimes_{\O_S}\mathbb G_a$ is \'etale. This shows that its image is open. In the colimit over $n$, it follows that the same is true for $\log_{\mathcal{G}}$.
		
		For (3), it suffices to prove that $[p]:\wh{\mathcal{G}}\to \wh{\mathcal{G}}$ is surjective, for which it suffices to see that $[p](\mathcal{G})\cap \widehat{\mathcal{G}} = [p](\widehat{\mathcal{G}})$. This is immediate from the definition. The case of general $U$ follows because $\Lie U=\Lie \mathcal{G}$.
		
		For part (5), observe that $U':=U\times_SS'\subseteq \mathcal{G}'$ is an open subgroup satisfying the description of \Cref{p:nbhd-open-subgroup-ball}. Let $\psi:\mathcal{G}'\to \mathcal{G}$ be the base-change map, then it follows from the construction that 
		\[\textstyle \widehat{\mathcal{G}'}=\bigcup_{n\in \N} [p^n]^{-1}_{\mathcal{G}'}(U')=\bigcup_{n\in \N} [p^n]^{-1}_{\mathcal{G}'}(\psi^{-1}(U))=\bigcup_{n\in \N} \psi^{-1}([p^n]^{-1}_{\mathcal{G}}(U))=\psi^{-1}(\widehat{\mathcal{G}}).\]

		For (6), we can argue roughly like in \cite[Proposition 2.14]{Heu22b}: We first observe that by considering varying $S$, it suffices to prove that for any perfectoid space $S$, we have
		\[ \Hom(\uZp,\mathcal{G})=\wh {\mathcal{G}}(S).\]
		We first observe that this is true for $\mathcal{G}=\G_a^{+n}$ for any $n\in \N$, where $\G_a^+:=(\mathbb B^1,+)$. Indeed, if $S$ is a perfectoid space, we have $\mathrm{Map}(\uZp,\G_a^+)=\mathrm{Map}_{\cts}(\Z_p,\O^+(S))$, and thus \[\Hom(\uZp,\G_a^+)=\mathrm{Hom}_{\cts}(\Z_p,\O^+(S))=\O^+(S)=\G_a^+(S).\]
		
		For general $\mathcal{G}$, let $\varphi:\uZp\to \mathcal{G}$ be any morphism, then there is an open subgroup $p^n\uZp$ that maps into $U\cong \G_a^{+n}$. We deduce from the case of $\G_a^+$ that $\varphi(p^n)\in \wh{\mathcal{G}}(S)$. Hence $\varphi(1)\in \wh{\mathcal{G}}(S)$ by definition of $\widehat{\mathcal{G}}$.
		
		Conversely, let $s\in \wh{\mathcal{G}}(S)$ and consider the induced map $\varphi_1:\underline{\Z}\to \wh{\mathcal{G}}$, $n\mapsto n\cdot s$. Then there is $n$ such that $p^{n}s\in U(S)$. By the case of $\G_a^{+n}$, there is then a unique homomorphism $\varphi_2:p^n\uZp\to U$ sending $p^n$ to $p^ns$. Consider the direct sum of $\varphi_1$ and $\varphi_2$
		\[\underline{\Z}\times p^n\uZp \xrightarrow{\varphi_1,\varphi_2} \mathcal{G}.\]
		Since $\varphi_1$ and $\varphi_2$ agree on their intersection $p^n\underline{\Z}$, this map admits a unique factorisation through the categorical quotient $(\underline{\Z}\times p^n\uZp)/p^n\underline{\Z}=\uZp$. This defines the desired map $\varphi:\uZp\to \mathcal{G}$ sending $1\mapsto s$.
	\end{proof}
	\begin{lemma}\label{l:log-surj-alg-case}
		The final assumption of \Cref{p:log-of-smooth-group}.(3) is always satisfied if $\mathcal{G}\to S$ is the analytification of a smooth group scheme  $G^\alg\to S^\alg$ over an algebraic $K$-variety $S^\alg$.
	\end{lemma}
	\begin{proof}
		By [SGA 3.1, Exp $\mathrm{VI}_B$, Théorème 3.10, Proposition~3.11], there exists a maximal open subgroup scheme $G^{\alg\circ}\subseteq G^\alg\to S$ such that every geometric fibre is connected, and on this, $[p]$ is surjective.
	\end{proof}
	\subsection{The Hodge--Tate sequence for relative adic groups}
	We can now state the main result of \S\ref{HT-rel-grp}.

	\begin{theorem}[relative Hodge--Tate sequence for $\mathcal{G}$]\label{c:leray-seq-for-whG}
		Let $f:X\to Y$ be a smooth morphism of smooth rigid spaces over $K$. 
		Let $\mathcal{G}\to X$ be a commutative smooth relative group. Then the Leray spectral sequence for $\mu:Y_v\to Y_{\Et}$ induces a left-exact sequence of abelian sheaves on $Y_v$, functorial in $X$ and $\mathcal{G}$,
		\begin{equation} \label{eq:nu-G}
			1\to \mu^* (\rR^1f_{\Et\ast}\widehat{\mathcal{G}}) \to \rR^1f_{v\ast}\widehat{\mathcal{G}} \xrightarrow{\HTlog_{\mathcal{G}}} f_{v\ast}(\Lie G\otimes_{\O_X} \widetilde{\Omega}_X).
		\end{equation}
	\end{theorem}
	\begin{rem}
	We call this the Hodge--Tate sequence for $\mathcal{G}$ because, via the Primitive Comparison Theorem,  the case of $Y=\Spa(K)$, $\mathcal{G}=\G_a$ and proper $X$ recovers the Hodge--Tate sequence of $p$-adic Hodge theory \cite[\S3]{ScholzeSurvey}.
	We note that \eqref{eq:nu-G} is not always right-exact, i.e.\ consider $Y=\Spa(K)$ and $\mathcal{G}=\G_a^+$.
	\end{rem}
	
	The main technical input into the proof of the Theorem is the following:
	\begin{prop} \label{t:Rn-nu}
		Let $S$ be a smoothoid space and let $\mathcal{G}\to S$ be a commutative smooth $S$-group such that $[p]:\mathcal{G}\to \mathcal{G}$ is \'etale. Let $\nu:S_v\to S_{\et}$ be the natural map. Then for $n\ge 1$, there is a natural isomorphism  
		\begin{equation} \label{eq:Rn-nu}
			\rR^n\nu_{\ast}\widehat{\mathcal{G}}\xrightarrow{\sim}\Lie(\mathcal{G})\otimes_{\O_S}\widetilde{\Omega}_{S}^n
		\end{equation}
	\end{prop}
	\begin{proof}
		This follows from the long exact sequence of Proposition~\ref{p:log-of-smooth-group}.\eqref{i:log-ses}: For $n\ge 1$, we have
		\[\textstyle \rR^n\nu_{\ast}(\mathcal{G}[p^\infty])=\varinjlim_n\rR^n\nu_{\ast}(\mathcal{G}[p^n])=0\]
		since $\mathcal{G}[p^n]\to S$ is \'etale by assumption. On the other hand,  by the projection formula and \Cref{d:differentials-smoothoids},
		\[ \rR^n\nu_{\ast}(\Lie(\mathcal{G})\otimes_{\O_S}\mathbb G_a)=\Lie(\mathcal{G})\otimes_{\O_S}\rR^n\nu_{\ast}\mathbb G_a=\Lie(\mathcal{G})\otimes_{\O_S}\widetilde{\Omega}^n_X.\]
		
		It remains to see that for any open subgroup $U\subseteq \Lie(\mathcal{G})\otimes_{\O_S}\mathbb G_a$, the induced map on cohomologies $\rR^n\nu_{\ast}U\to \rR^n\nu_\ast(\Lie(\mathcal{G})\otimes_{\O_S}\mathbb G_a)$ is an isomorphism. As this question is local on $S$, we may assume that $\Lie(\mathcal{G})\otimes_{\O_S}\mathbb G_a\simeq \G_a^d$. Arguing exactly as in \cite[Lemma~3.10]{diamantinePicard}, we see that $\G_a^d/U$ satisfies the approximation property assumed in \cite[Proposition~2.14]{heuer-G-torsors-perfectoid-spaces}, so we may conclude that $\rR^n\nu_\ast(\G_a^d/U)=0$.
	\end{proof}
	
	\begin{proof}[Proof of \Cref{c:leray-seq-for-whG}]
		For any $b:T\to Y$ in $Y_v$, let $X_T:=X\times_YT\to T$ be the base-change of $X\to Y$ along $b$. This is a smoothoid space over $T$. Let $\mathcal{G}_T:=\mathcal{G}\times_X X_T\to X_T$ be the base-change, a smooth relative group over $X_T$. We consider the Leray sequence for the morphism $\nu:X_{T,v}\to X_{T,\et}$. For the abelian sheaf on $X_{T,v}$ represented by the smooth relative group $\widehat{\mathcal{G}}_T$, this gives a left-exact sequence
		\begin{equation} \label{eq:Leray-et-v}
			0\to \rH^1_{\et}(X_T,\widehat{\mathcal{G}}_T)\to \rH^1_{v}(X_T,\widehat{\mathcal{G}}_T)\to \rH^0(X_T,\rR^1\nu_{\ast}\widehat{\mathcal{G}}_T).\end{equation}
			By \Cref{c:G[p]-etale}, $[p]:\mathcal{G}\to \mathcal{G}$ is \'etale, hence so is $[p]:\mathcal{G}_T\to \mathcal{G}_T$. We can thus apply \Cref{t:Rn-nu}:
		\[ \rH^0(X_T,\rR^1\nu_{\ast}\widehat{\mathcal{G}}_T)=\rH^0(X_T,\Lie (\mathcal{G}_T)\otimes \widetilde{\Omega}_{X_T})=f_{v\ast}(\Lie(\mathcal{G})\otimes \widetilde{\Omega}_X)(T),\]
		where for the last equality we use that by \cite[Proposition 2.9.(2)]{Heu}, the sheaf $\widetilde{\Omega}_{X_T}$ is the pullback of $\wtOm_X$ along $X_T\to X$.
		This describes the last term in \eqref{eq:Leray-et-v}. For the first term, the \'etale sheafification of $T \mapsto \rH^1_{\et}(X_T,\widehat{\mathcal{G}}_T)$ becomes $\rR^1f_{\Et\ast}\widehat{\mathcal{G}}$: 
		Here we use that $\widehat{\mathcal{G}}_T=\widehat{\mathcal{G}}\times_XX_T$ (\Cref{p:log-of-smooth-group}). 
		
		Upon v-sheafification in $T$, the left-exact sequence thus attains the desired form \eqref{eq:nu-G}.
	\end{proof}

	\section{The canonical Higgs field on v-$G$-bundles}\label{s:can-Higgs}
	Throughout this section, let $K$ be a perfectoid field extension of $\Qp$ that contains all $p$-power roots of unity. Let $G$ be a rigid group over $K$, and let $Y$ be a smoothoid adic space over $K$ in the sense of \Cref{d:smoothoid}. The main goal of this section is to associate to any v-topological $G$-bundle $V$ on $Y$ a canonical Higgs field $\theta_V$. This generalises a construction of Rodr\'iguez Camargo \cite{camargo2022geometric} in the context of the work of Pan \cite{PanLocallyAnalytic}. For the construction, we first need to recall the local $p$-adic Simpson correspondence in this setup:

	\subsection{Local $p$-adic Simpson correspondence} We begin with some recollections from \cite{Heu}.
	\begin{secnumber}\label{s:data-induced-by-toric-chart}
		Let $Y$ be a smoothoid adic space with a toric chart $f:Y\to \TT^d\times T$ where $T$ is affinoid perfectoid (see \Cref{d:smoothoid}). Recall that the torus $\TT^d$ admits a pro-\'etale affinoid perfectoid cover
		\[
		\TT^d_{\infty}=\Spa(K\langle T_1^{\pm 1/p^{\infty}},\cdots,T_d^{\pm 1/p^{\infty}}\rangle) \to \TT^d
		\] which is a $\Z_p(1)^d$-torsor in the pro-\'etale site $\TT^d_\proet$. Pulling this back along the toric chart $f$, we obtain a pro-\'etale affinoid perfectoid cover $Y_{\infty}\to Y$ via pullback of $	\TT^d_{\infty}\to 	\TT^d$ along $f$. Let $\Delta_f:=\Gal(Y_{\infty}/Y)\simeq \Z_p(1)^d$ be the Galois group of this cover.
		Then by \cite[Lemma~2.17]{Heu}, the chart $f$ induces isomorphisms 
		\[
		\HT_f: \Hom_{\cont}(\Delta_f,\O_Y(Y))\xrightarrow{\sim} 
		\rH^1_{\cont}(\Delta_f,\O_{Y_v}(Y_{\infty})) \xrightarrow{\sim} 
		\rH^1_v(Y,\O_Y) \xrightarrow{\sim} 
		\rH^0(Y,\widetilde{\Omega}_Y)
		\]
		which can explicitly be described as follows: The chart $f$ induces a basis $\frac{\mathrm dT_1}{T_1},\dots,\frac{\mathrm dT_d}{T_d}$ of $\widetilde{\Omega}_Y$ over $\mathcal{O}_Y(Y)$. Let $\partial_1,\dots,\partial_d$ be the dual basis. Then the $\O_Y(Y)$-linear dual of $\HT_f$
		\[ \rho_f:\Delta_f\to \rH^0(Y,\widetilde{\Omega}^\vee_Y)\]
		is the $(1)$-twist of the map that sends the $i$-th basis vector $\gamma_i$ of $\Delta_f(-1)\simeq \Z_p^d$ to $\partial_i$.
		We denote by $\widetilde{\Omega}^+_{Y,f}$ the finite free $\O_Y^+$-submodule of $\widetilde{\Omega}_Y$ generated by the image of $\Hom_{\cont}(\Delta_f,\O^+_Y(Y))$ under $\HT_f$. 
		
		We also need the following more precise integral version of the first isomorphism in $\HT_f$:
		\begin{lemma}[{\cite[Lemma~2.14]{Heu22b}}]\label{d:def-gamma}
			There is a constant $\gamma\in \mathbb R_{>0}$ depending only on $Y$ and $f$ such that for any $s\in \N$ and $i\geq 1$, the following map 	has $\gamma$-torsion kernel and cokernel:
			\[\rH^i_{\cont}(\Delta_f,\O^+_Y(Y)/p^s)\xrightarrow{\sim} 
			\rH^i_{\cont}(\Delta_f,\O_{Y_v}^+(Y_{\infty})/p^s)\]
		\end{lemma}
		
	\end{secnumber}
		As before, let $G$ be any rigid group over $K$, written multiplicatively. Let $\Lie(G)$ be its Lie algebra and let $\mathfrak g:=\Lie(G)\otimes_K\mathbb G_a$ be the associated rigid group. We recall the $p$-adic Lie algebra exponential of $G$:
		\begin{lemma}[{\cite[\S 3.2, Proposition 3.5]{heuer-G-torsors-perfectoid-spaces} \cite{Sch}}]\label{l:ex-of-exp-general-G}
			There exists an open rigid subgroup $\fg^{\circ}\subseteq \mathfrak g$, isomorphic as a rigid space to a closed ball, for which there is a morphism of rigid spaces
			\[
			\exp:\fg^{\circ}\to G
			\]
			(but not a homomorphism, unless $G$ is commutative)
			that is  uniquely characterised by the following properties:
			\begin{enumerate}
				\item 
				$\exp$ is an open immersion onto an open subgroup $G_0$ of $G$.
				\item We have $\exp(0)=1$ and $\exp$ induces the identity map on tangent spaces.
				\item The group structures of $\mathfrak g^\circ$ and $G$ are related via $\exp$ by the Baker--Campbell--Hausdorff formula. 
			\end{enumerate}
		\end{lemma}
		We note that the subgroup $\fg^\circ\subseteq \fg$ is not in general uniquely determined. In the following, when we deal with any rigid group $G$, we will always tacitly fix such a group $\fg^\circ$ throughout. For reductive $G$, there is in fact a canonical choice \cite[Example~3.3, Lemma~4.20]{heuer-G-torsors-perfectoid-spaces}. But as we will never need any precise estimates or radii of convergence in the following, it is harmless to just make any choice.
		For any $k\in \Z_{\geq 0}$, we then set
		\[\mathfrak g_k:=p^k\mathfrak m_K \mathfrak g^\circ,\] which is on open subgroup of $\mathfrak g$. Then its image $G_k:=\exp(\mathfrak g_k)$ is an open subgroup of $G$ by  \cite[Proposition 3.5]{heuer-G-torsors-perfectoid-spaces}, and the map
		\[\exp:\mathfrak g_k\isomarrow G_k\]
		is an isomorphism of rigid spaces. We call its inverse $\log$. 
	\begin{secnumber}
		For the formulation of the local correspondence,
		we can now recall the notion of small $G$-bundles:
	\end{secnumber}
	
	\begin{definition}[\cite{Heu} Definition 6.2, Lemma~6.3]
		Let $Y$ be smoothoid with a fixed toric chart $f$. Let $\gamma$ be as in  \Cref{d:def-gamma}. Set $c:=5\gamma$ (we refer to \cite[Proposition 5.5]{Heu} for a motivation of this constant).
		\begin{enumerate}
			\item  A $G$-bundle $V$ on $Y_v$ is \textit{small} if $V$ admits a reduction of structure group to $G_c$. 
			\item A $G$-Higgs bundle $(E,\theta)$ on $Y_\et$ is \textit{small} if $E$ is trivial and there exists a trivialisation $E\cong G$ with respect to which $\theta$ is a section of the $\O_Y^+$-submodule
			$
			\mathfrak g_c\otimes_{\O_Y^+} \widetilde{\Omega}_{Y,f}^+
			$, where $\widetilde{\Omega}_{Y,f}^+$ was defined in \Cref{s:data-induced-by-toric-chart}.
		\end{enumerate}
	\end{definition}
	We can now recall the local $p$-adic Simpson correspondence. This is a generalisation of a result of Faltings (\cite[Theorem~3]{Fal05},  \cite[\S II.13]{AGT-p-adic-Simpson}) from $\GL_n$ to general rigid groups $G$ and to smoothoid spaces:
	\begin{theorem}[Local $p$-adic Simpson correspondence for $G$, {\cite[Theorem~6.5]{Heu}}]\label{t:small-corresp}
		Let $Y$ be a toric smoothoid space over $K$ and let $f:Y\to \TT^d\times T$ be a toric chart. 
		Then $f$ induces an equivalence of groupoids
		\[
		\LS_f: \big\{\textnormal{small $G$-Higgs bundles on $Y_{\et}$}\big\} \xrightarrow{\sim} 
		\big\{\textnormal{small v-$G$-bundles on $Y_v$}\big\}.
		\]
		In the case of $\GL_n$, this extends to an equivalence of categories
		\[
		\LS_f: \big\{\textnormal{small Higgs bundles on $Y_{\et}$}\big\} \xrightarrow{\sim} 
		\big\{\textnormal{small v-vector bundles on $Y_v$}\big\}.
		\]
	\end{theorem}
	We also recall for later reference that one can always apply $\LS_f$ locally on $Y$:
	\begin{lemma}[{\cite[Lemmas~6.4, 4.11]{Heu}}]\label{l:loc-small}
		For any v-$G$-bundle $V$ on $Y$, there is an \'etale cover $g:Y'\to Y$ with a toric chart $h$ of $Y'$ such that $g^\ast V$ is small with respect to $h$. The same holds for $G$-Higgs bundles.
	\end{lemma}
	\begin{secnumber} \label{sss:LS}
		We briefly review the construction of $\LS_f$, see \cite[\S6]{Heu} for details.
		Let $(E,\theta)$ be a small $G$-Higgs bundle where $E\cong G$. Via $\HT_f$, the section $\theta\in 	\mathfrak g_c\otimes_{\O_Y^+} \widetilde{\Omega}_{Y,f}^+$ corresponds to a continuous homomorphism 
		\begin{equation} \label{eq:rho-theta}
			\rho:\Delta_f\to \mathfrak g_c(Y).\end{equation}
		More explicitly, it is given by interpreting any $\gamma\in \Delta_f$ via $\HT_f$ as a function $\partial(\gamma):\widetilde{\Omega}_{Y,f}^+\to \O^+_Y$ and setting $\rho(\gamma):=\partial(\gamma)(\theta)$. It follows from this that the Higgs field condition for $\theta$  translates to the statement that $\rho$ has commutative image, in the sense that for any $\gamma,\gamma'\in \Delta_f$, we have $[\rho(\gamma),\rho(\gamma')]=0$ in $\mathfrak g_c(Y)$. By \Cref{l:ex-of-exp-general-G}.(3),
		this ensures that applying $\exp$ preserves the linearity, so we obtain a homomorphism
		\[\exp(\rho):\Delta_f\to \mathfrak g_c(Y)\xrightarrow{\exp}G_c(Y).\]
		The associated v-$G$-bundle $V=V_\rho$ on $Y_v$ is now defined for any $W\in Y_v$ by 
		\begin{equation} \label{eq:LS}
			V_\rho(W):=\big\{ s\in E(Y_{\infty}\times_Y W)\big| \gamma \cdot s=\exp(-\rho(\gamma))s,\quad \forall \gamma\in \Delta_f\big\},
		\end{equation}
		where the action of $\Delta_f$ on $E(Y_{\infty}\times_Y W)$ is induced by the Galois action of $\Delta_f$ on $Y_\infty$. More geometrically,   $V_\rho$ is isomorphic to the pushout of the $\underline{\Delta}_f$-torsor $Y_\infty\to Y$ along the morphism of v-sheaves $\underline{\Delta}_f\to G_c\to G$ attached to $\rho$.
		This shows that $V_\rho$ is a small v-$G$-bundle on $Y_v$. On the other hand, \eqref{eq:LS} shows that there is a canonical isomorphism
		\begin{equation}\label{eq:comp-LS-on-toric-cover}
			(V_\rho)|_{Y_\infty}\xrightarrow{\sim} E|_{Y_\infty}.
		\end{equation}
	\end{secnumber}
	\subsection{Canonical Higgs field for v-$G$-bundles}
	The main result of  \S\ref{s:can-Higgs} is now the following:
	\begin{theorem} \label{t:canonicalHiggs}
		Let $Y$ be a smoothoid adic space over $K$ and let $G$ be any rigid group over $K$. 
		\begin{enumerate}
			\item 
			There is a unique way to associate to any v-$G$-bundle $V$ on $Y$ a canonical Higgs field 
			\[\theta_V\in\rH^0(Y,\ad(V)\otimes_{\mathcal O_Y}\widetilde{\Omega}_Y)\] 
			in such a way that the following two conditions holds:
			\begin{enumerate}
				\item The association $V\mapsto (V,\theta_V)$ defines a fully faithful functor,  natural in $Y$ and $G$,
				\[\theta:\big\{\text{v-$G$-bundles on $Y$}\big\}\to \big\{\text{v-$G$-Higgs bundles on $Y$}\big\}.\]
				\item 
				When $Y$ admits a toric chart $f$ and $V$ is small on $Y$, let $(E,\theta_E)=\LS_f^{-1}(V)$ be the associated $G$-Higgs bundle via the local $p$-adic Simpson correspondence of \Cref{t:small-corresp}.
				Then the natural isomorphism $V(Y_{\infty})\simeq E(Y_{\infty})$ of  \eqref{eq:comp-LS-on-toric-cover} identifies the pullbacks of $(V,\theta_V)$ and $(E,\theta_E)$ to $Y_{\infty}$. 
			\end{enumerate} 
			\item\label{p:Cartierdescent} The morphism of topoi $\nu:\widetilde{Y}_v\to \widetilde{Y}_\et$ induces an equivalence of categories
			\[  \big\{\text{$G$-bundles on $Y_\et$}\big\}\to \big\{\text{v-$G$-bundles $V$ on $Y$ with $\theta_V=0$}\big\}.\]
			
		\end{enumerate}
	\end{theorem}
	\begin{definition} \label{d:canonicalHiggs}
		We call $\theta_V$ the \textit{canonical Higgs field} of  $V$. 
	\end{definition}
	The naturality in (1).(a) means in particular that the formation of $\theta_V$ is  compatible with localisation.

	\begin{rem}
		Part (2) means that $\theta_V$ can be viewed as a mixed characteristic analogue of the \textit{$p$-curvature} in mod $p$ geometry.
		Indeed, (2) can be regarded as an analogue of Cartier descent \cite[Theorem 5.1]{Ka71}. 
	\end{rem}
	\begin{rem}
		For $G=\GL_n$ on smooth rigid spaces, \Cref{t:canonicalHiggs} is due to Rodr\'iguez Camargo \cite{camargo2022geometric} (up to a small difference in technical setups: In \cite{camargo2022geometric}, it is assumed that $K$ is algebraically closed). Our result can thus be regarded as a generalisation to general $G$ and smoothoid $Y$, by a different proof.
	\end{rem}
	\begin{rem}
		Assume that $G$ is commutative, then $\ad(G)\otimes \wtOm=\Lie(G)\otimes \wtOm$. Therefore, in this special case,  \Cref{t:canonicalHiggs} is closely related to the short exact sequence of \Cref{c:leray-seq-for-whG}, which we may regard as a geometrisation and a generalisation to relative groups.
		Indeed, given a v-$G$-torsor $V$, we can regard $\theta:=\HTlog(V)$ as a Higgs field, and the left-exactness of the sequence corresponds to \Cref{t:canonicalHiggs}.\ref{p:Cartierdescent}.
		
		In fact, we think that there ought to be a generalisation of \Cref{t:canonicalHiggs} to relative groups on smoothoids.
	\end{rem}
	\begin{proof}[{Proof of  Theorem \ref{t:canonicalHiggs}}]
		We first assume that $Y$ is toric with a fixed toric chart $f$ and that $V$ is small with respect to $f$. Then by \Cref{t:small-corresp}, we can find a small $G$-Higgs bundle $(E,\theta)$ such that $V=\LS_f(E,\theta)$. We may thus assume that $V$ is as described in \eqref{eq:LS}. Unravelling the definition of $\ad(-)$, it  follows that
		\begin{equation}
			(\ad(V_\rho)\otimes_{\O_Y} \wtOm_Y)(Y)=\big\{ s\in \ad(E)(Y_{\infty})\otimes_{\O(Y)} \wtOm_Y(Y)\big| \gamma \cdot s=\ad(\exp(-\rho(\gamma)))\cdot s,\quad \forall \gamma\in \Delta_f\big\}
		\end{equation}
		where on the right, the adjoint action $\ad:G\to \uEnd(\ad(E))$ is obtained by deriving the action $G\to \uAut(E)$.
		We claim that $\theta$ is an element of this set. Since $\theta$ is fixed by the $\Delta_f$-action, it suffices to prove that
		\[\theta=\ad(\exp(-\rho(\gamma)))\cdot \theta.\]
		Via $\HT_f$, we can identify $\theta$ with $\rho\in \Hom(\Delta_f,\O(Y))$, so it suffices to prove that for any $\gamma,\gamma'\in \Delta_f$,
		\[\rho(\gamma')=\ad(\exp(-\rho(\gamma)))\cdot \rho(\gamma')\]
		inside $\mathfrak g_c(Y)$.
		By applying the bijection $\exp$ and using \cite[Lemma~3.10.3]{heuer-G-torsors-perfectoid-spaces}, this is equivalent to
		\[\exp(\rho(\gamma'))=\exp(\rho(\gamma))\cdot \exp(\rho(\gamma'))\cdot \exp(-\rho(\gamma)).\]
		This holds by \Cref{l:ex-of-exp-general-G}.(3) (see also \cite[Lemma~3.10.1]{heuer-G-torsors-perfectoid-spaces}) because $[\rho(\gamma),\rho(\gamma')]=0$ due to the assumption that $\theta$ is a Higgs field. This shows that $\theta$ defines a Higgs field $\theta_V$ on $V$, as we wanted to see.

		This construction is clearly functorial in $V$ and natural in $G$ and $(Y,f)$. As any v-$G$-torsor on $Y$ becomes small on some \'etale cover by \Cref{l:loc-small}, it remains to prove that $\theta_V$ is independent of the toric chart $f$. If this is the case, then the local definitions glue to a global Higgs field on $Y$. This will also show (1).(a).
		
		We note that for $G=\GL_n$, the above local definition of $\theta_V$ recovers that in \cite[Theorem 4.8, Remark~4.9]{Heu23}. This Theorem also proves  the independence of toric chart for $G=\GL_n$, as the proof still works without change for smoothoid $Y$. We are therefore left to reduce the general case to that of $\GL_n$.
		
		To this end, let $f'$ be any other toric chart of $Y$ and let $\theta'_V$ be the associated Higgs field computed with respect to $f'$. We wish to see that $\theta_V=\theta'_V$. Since we can check this on any \'etale cover of $Y$, we may shrink $G$: By \cite[Corollary~3.9]{heuer-G-torsors-perfectoid-spaces}, there are $k,r\in \N$  such that $G_k$ admits a homomorphism $\varphi:G_k\hookrightarrow \GL_r$ that is a locally closed immersion. After replacing $Y$ by some \'etale cover, we may assume that $V$ admits a reduction of structure group to $G_k\subseteq G_c$ for this $k$. By functoriality in $G$, it therefore suffices to prove the statement for $G_k$ instead of $G$. But then, as $\varphi$ is an immersion, the induced map
		\[ \Lie(\varphi)\otimes \wtOm:\mathfrak g\otimes \wtOm\to M_r(\O_Y)\otimes \wtOm\]
		is injective. We can thus indeed reduce to the case of $G=\GL_r$ to check that $\theta_V=\theta_V'$. This shows (1)(b).
		
		Part (2) follows immediately from  the local construction: Via the local $p$-adic Simpson correspondence \Cref{t:small-corresp}, the Higgs bundle $\LS_f^{-1}(V)$ has trivial Higgs field if and only $V$ is \'etale-locally trivial.
	\end{proof}
	\begin{rem}
		Alternatively, the independence of toric chart can be seen by a direct computation.
	\end{rem}

	\subsection{Moduli stacks}
	\begin{definition}[{\cite[\S7.2]{Heu}}]
		Let $X$ be a smooth rigid space over $K$. 
		\begin{enumerate}
			\item For $\tau\in \{\et,v\}$, we denote by $\CBun_{G,\tau}$ the prestack on $\Perf_K$ defined by 
			\[
			T\mapsto \{G\textnormal{-bundles on } (X\times T)_{\tau} \}. 
			\]
			\item Let $\CHig_{G,\tau}$ be the prestack of Higgs bundles defined as the fibered functor on $T\in \Perf_K$: 
			\[
			T\mapsto \{\textnormal{$\tau$-$G$-Higgs bundles on $X\times T$}\}.
			\]
		\end{enumerate}
			The key players in this article will be $\CBun_{G,v}$ and $\CHig_{G,\et}$. We therefore also set $\CHig_{G}:=\CHig_{G,\et}$.
	\end{definition}
	\begin{prop}[{\cite[Theorem~7.13]{Heu}}]\label{p:CBun-small}
		For $\tau\in \{\et,v\}$, both $\CBun_{G,\tau}$ and $\CHig_{G,\tau}$ are v-stacks. Moreover, $\CBun_{G,\et}$, $\CBun_{G,v}$ and $\CHig_{G,\et}$ are small v-stacks. 
	\end{prop}
	\begin{proof}
		Let us present a variant of the proof in \cite{Heu} of the first sentence:
		For $\CBun_{G,v}$, this follows from v-descent for v-$G$-bundles. For $\CHig_{G,v}$, note that
		a Higgs field $\theta$ on a v-$G$-bundle can be v-locally defined and the vanishing of $\theta\wedge \theta$ can be v-locally verified. 
		Thus the case of $\CHig_{G,v}$ follows from that of $\CBun_{G,v}$. 
		
		For  $\CBun_{G,\et}$, we can now use that by \Cref{t:canonicalHiggs}.(2), an \'etale $G$-bundle is equivalent to a v-$G$-bundle $V$ with $\theta_V=0$. We may therefore deduce v-descent for \'etale $G$-bundles from that of $(V,\theta_V)$, showing that  $\CBun_{G,\et}$ is a v-stack.
		The assertion for $\CHig_{G,\et}$ can be verified in a similar way as for $\CHig_{G,v}$. 
	\end{proof}
	We can now reinterpret
	Theorem \ref{t:canonicalHiggs} as saying that the natural morphism $\CHig_{G,v}\to \CBun_{G,v}$, defined by forgetting Higgs fields, admits a canonical section:
	\begin{equation}\label{eq:section-psi}
	\psi:\CBun_{G,v}\to \CHig_{G,v},\quad V\mapsto (V,\theta_V).
	\end{equation}

	\section{Abelianisation for reductive $G$}\label{s:abelianisation}
	Throughout this section, let $K$ be an algebraically closed complete extension of $\mathbb Q_p$. The first goal of this section is to recall the definition of the Hitchin base $\mathbf A$ and the Hitchin morphisms for $\CHig$ and $\CBun_v$.
	Second, following Ng\^o, we define a commutative smooth relative group $J\to X\times \mathbf A$ depending on $G$ that will be the key player for the phenomenon of ``abelianization'':
	Roughly speaking, \Cref{d:J-b-action}, Propositions~\ref{c:a_E-for-v-bundles} and \ref{p:can-section-tau_theta} below say that $J$ acts on both $G$-Higgs bundles and v-$G$-bundles in a natural way that remembers the Higgs field, respectively the canonical Higgs field. For $G=\GL_n$, this is closely related to the classical BNR-correspondence, as $J$ is then given by the group of units of the spectral curve.
	\subsection{Hitchin map and centralizers after Ng\^o} \label{ss:HitchinNgo}
	
	In this subsection, we review the Hitchin map over a curve following Ng\^o \cite{Ngo} (see also \cite[\S2]{CZ15}). Our setup will differ slightly from that of Ng\^o as we work in an analytic setting over $K$:
	Let $X$ be a smooth projective curve over $K$. Let $G$ be a connected reductive group over $K$. We will later consider both $X$ and $G$ as adic spaces over $K$. We fix a maximal torus of $G$ and denote by $\ft$ the associated affine group scheme over $K$. Let $W$ be the Weyl group of $G$. 

	\begin{secnumber}
		Let $\Lie G$ be the Lie algebra of $G$. We denote by $\fg=\G_a\otimes_K\Lie G$ the associated affine group scheme over $K$.
		We set $\fc=\Spec(\O(\fg)^G)$ where $\O(\fg)^G$ are the invariants for the adjoint action of $G$ on $\mathfrak g$. Let
		\begin{equation}
			\chi:\fg\to \fc
			\label{eq:Chevalley}
		\end{equation}
		be the Chevalley map induced by $\O(\fg)^G\to \O(\fg)$. 
		This is a morphism of $K$-varieties that is $G\times \Gm$-equivariant for the trivial $G$-action on $\fc$ and the $\Gm$-action on $\fc$ defined by the gradings on $\O(\ft)^W\simeq \O(\fg)^G$.
		
		We simply denote the line bundle $\widetilde{\Omega}_{X}$ on $X$ by $\Omega$, if there is no confusion. 
		Let \[\fc_{\Omega}:=\Omega\times^{\Gm}\fc\quad\text{and}\quad \fg_{\Omega}:=\Omega\times^{\Gm}\fg\] be the $\Gm$-twist of $\fc$ and $\fg$ by the geometric line bundle $\Omega$ over $X$, considered as schemes over $X$. Then the \textit{Hitchin base} ${A}_{G,X}$ may be defined as $\Sect(X,\fc_{\Omega})$, the scheme of sections of $\fc_{\Omega}$ over $X$ \cite[Lemma~2.4]{Ngo}. 
		\begin{definition}\label{d:univ-section}
			Let $u:X\times A \to \fc_{\Omega}$ be the universal section over $X$.
		\end{definition}
		\end{secnumber}
		\begin{secnumber}	
		We now pass to the analytic setup of adic spaces over $\Spa(K)$. Let $\A_{X,G}$ be the analytification of $A_{X,G}$. By abuse of notation, let us still denote by $X$ the adic space associated to $X$, and similarly for the other schemes considered above. Passing further from adic spaces over $K$ to v-sheaves over $K$, by \cite[Lemma~8.9]{Heu}, the analytified Hitchin base $\A_{X,G}$ then represents the functor of sections of $\fc_{\Omega}$ over $X$ on $\Perf_K$
		\[\A_{X,G}:\Perf_K\to \mathrm{Sets},\quad
		T\mapsto \{\textnormal{sections $X_T\to \fc_{\Omega,T}$ over $X$}\},
		\]
		where $X_T:=X\times T$ is the base-change of $X$ and similarly $\fc_{\Omega,T}$ is the base-change of $\fc_{\Omega}\to X$ to $X_T$. 
		We simply denote the Hitchin base by $\A$ if $X$ and $G$ are clear from the context. 
	\end{secnumber}
	
		The analytification $X\times \A \to \fc_{\Omega}$ of $u$ is then uniquely characterised by the property that for any $f:T\to \A$ in $\A_v$, the corresponding  section $s:X_T\to \fc_{\Omega,T}$ is given by the composition $X\times T\xrightarrow{\id\times f} X\times \A\xrightarrow{u}\fc_{\Omega}$.

	\begin{secnumber}
		We now give a construction of the Hitchin map in terms of quotient stacks, which is specific to curves:

		\begin{definition}\label{d:quotient-stack}	Let $Y$ be a smoothoid space, and let $V \to Y$ be a smooth morphism of adic spaces equipped with a left-action by $G$. Let $\tau\in \{\et,v\}$. 
		The \textit{quotient $\tau$-stack} $[V/G]_{\tau}$ of $V$ by $G$ is defined by sending each object $T\to Y$ of the big \'etale site $Y_{\Et}$ (resp. $Y_v$) to the groupoid of pairs $(E,\varphi)$, where $E$ is a $\tau$-$G$-bundle on $T$ and $\varphi$ is a section $T\to E\times^GV_T$ of the pullback of $V$ to $T$ twisted by $E$. 
		\end{definition}
		\begin{rem}\label{r:moduli-quotient-stack-equivar-version}
					This is equivalent to the datum of a $G$-equivariant morphism $\phi:E\to V_T$ with respect to the left action on $E$ by $G$ defined by $g \cdot e:=eg^{-1}$ for $g\in G$, $e\in E$. Indeed, given such a morphism $\phi$, the morphism $E\to E\times^GV_T$, $e\mapsto (e,\phi(e))$ is clearly constant and thus factors through $\varphi$.
		\end{rem} 
	\end{secnumber}
	
		We now apply this to the vector bundle $V=\fg_{\Omega}$ on $X$ equipped with the adjoint action of $G$. 
		\begin{lemma}\label{l:CHig-as-sheaf-of-sections}
			Let $T\in \Perf_K$ and let $Y\in X_{T,\et}$. Then the groupoid of $\tau$-$G$-Higgs bundles on $Y$ is naturally isomorphic to the groupoid of sections 
			\[s: Y \to [\fg_{\Omega}/G]_{\tau}\]
			over $X$.
			In particular, the v-stack $\CHig_{G,\tau}$ is isomorphic to the stack of sections of $[\fg_{\Omega}/G]_{\tau}$ over $X$.
		\end{lemma}
		\begin{definition}
		Given a $\tau$-$G$-Higgs bundle $(E,\varphi)$ on $Y$, we denote the associated section by $s_{E,\varphi}$.
		\end{definition}
		\begin{proof}
			Any such section $s$ is equivalent to a pair $(E,\varphi)$ consisting of a $\tau$-$G$-bundle $E$ over $Y$ and a section $\varphi\in \rH^0(Y,\ad(E)\otimes \widetilde{\Omega}_{X_T})$. Since $X$ is a curve, by \S\ref{s:Higgs-bundle-curve}, this is precisely the datum of a $G$-Higgs bundle.
		\end{proof}
		We now come to the key definition of the two Hitchin fibrations, one for $\CHig_{G}$, one for $\CBun_{G,v}$.
		 \begin{definition}\label{d:Hitchin-map}
		The $G$-invariant morphism $\chi:\fg\to \fc$ induces by twisting with $\Omega$ a natural morphism
		\[
		[\chi_{\Omega}]: [\fg_{\Omega}/G]_{\tau} \to \fc_{\Omega}.
		\]
		Passing to the associated v-sheaves of sections over $X$ and using \Cref{l:CHig-as-sheaf-of-sections}, this induces the \textit{Hitchin map}: 
		\[
		h_{\tau}:\CHig_{G,\tau}\to \A.
		\]
		When $\tau=\et$, we shall often drop the subscript $\et$ from notation if this is clear from context.
		\end{definition}
	Using the canonical Higgs field on v-$G$-bundles from \Cref{t:canonicalHiggs}, we also get a Hitchin map for $\CBun_{G,v}$:
	\begin{definition}\label{d:Hitchin-Betti}
		The \textit{Hitchin map for $\CBun_{G,v}$} is the morphism of small v-stacks
		\begin{equation}
			\widetilde{h}:\CBun_{G,v}\xrightarrow{\psi} \CHig_{G,v} \xrightarrow{h_v} \A.
		\end{equation}
		where $h_v$ was defined in \Cref{d:Hitchin-map} and $\psi$ in \eqref{eq:section-psi}. We call $\widetilde{h}$ the \textit{Hitchin map on the Betti side}.
	\end{definition}
		\begin{rem}\label{rem:Hitchin-fibration-general}
		As explained in detail in \cite[\S8]{Heu}, one can more generally also define Hitchin morphisms $\wt h$ and $h$ when $X$ is any smooth rigid space and $G$ is any rigid group. Here $\mathbf A_{X,G}$ is in general a certain v-sheaf (see \cite[Definition~8.8]{Heu}). It is clear that for $X$ a smooth projective curve and for reductive $G$, the definition of $h$ agrees with the one given above. For $\wt h$, this easily follows from \Cref{t:canonicalHiggs}.(1).
	\end{rem}
	\subsection{The commutative group $J$}
		We need some further constructions from \cite{Ngo}, for which we switch back to the setting of schemes over $K$ for a moment. Recall that the Chevalley map $\chi:\fg\to \fc$ has a  section
		\[ \kos:\fc\to \fg,\]
		unique up to conjugation,
		called the \textit{Kostant section}. 
		We refer to \cite[\S2]{Ngo} for more details on its definition.
		
		Next, we form the centralizer
		\[\gI=\{(g,x)\in G \times \fg\mid \ad_g(x)=x\}
		\]
		 as a relative group scheme over $\fg$.  There is a natural $G$-action on $\gI$ given by $h\cdot(g,x)=(hgh^{-1},\ad_h(x))$. 
		 
		Let $\fg^{\reg}\subseteq \fg$ be the open locus where $\gI\to \fg$ has dimension $=\dim \fc$. This is a dense subspace of $\fg$ over which $\gI\to \fg$ is smooth, see \cite[Th\'eor\`eme~2.1]{Ngo}. By a Theorem of Kostant, the restriction of $\gI$ to $\fg^\reg$ is commutative \cite[Proposition~14]{kostant63}.  Moreover, $\kos:\fc\to \fg$ factors through $\fg^{\reg}$. We deduce:
			\begin{definition}\label{d:J}
		The \textit{regular centralizer} $\gJ:=\kos^*\gI\to \fc$ is a smooth commutative relative group scheme.
		\end{definition}
		
		 By  \cite[Proposition 3.2]{Ngo}, there exists a canonical isomorphism of group schemes 
		\begin{equation}\label{eq:map-a-regular-locus}
		\chi^* \gJ|_{\fg^{\reg}} \xrightarrow{\sim} \gI|_{\fg^{\reg}}
		\end{equation}
		 over the regular locus $\fg^{\reg}$ of $\fg$, which extends uniquely to a homomorphism of group schemes over $\fg$
		\begin{equation}
			a:\chi^* \gJ \to \gI.
			\label{eq:map a}
		\end{equation}
		Since $\chi$ is $G$-invariant, the pullback $\chi^\ast \gJ\to \fg$ acquires a natural $G$-action such that $\chi^\ast \gJ\to \gJ$ is $G$-invariant. With respect to this action, $a$ is $G$-equivariant: As $\chi^\ast\gJ$ is flat, this can be checked over the dense open subspace $\fg^{\reg}\subseteq \fg$, where it is clear from the explicit description of $a$ in \cite[proof of Proposition 3.2]{Ngo}.
		
		There exists a natural $\Gm$-action on $\gI$ defined by $t\cdot (g,x)=(g,t\cdot x)$. It induces a $\Gm$-action on $\gJ$ such that $\gI\to \fg$ and $\gJ\to \fc$ and $a$ are $\Gm$-equivariant. 
	As before, we form the twists of $\gI,\gJ$ by the $\G_m$-torsor $\Omega$: \[\gI_{\Omega}=\Omega\times^{\Gm}\gI  \quad \text{and} \quad\gJ_{\Omega}=\Omega\times^{\Gm}\gJ.\]
	Due to the $\G_m$-equivariance, we  obtain morphisms $\gI_{\Omega}\to \fg_\Omega$ and $\gJ_{\Omega}\to \fc_\Omega$ over $X$ regarded as a scheme.

	\begin{definition} \label{d:def-univ-J} 
		We define a smooth relative group scheme $J$ over $X\times A$ as the fiber product:
		\[
		\xymatrix{
			J\ar[r] \ar[d] & \gJ_{\Omega} \ar[d] \\
			X\times A \ar[r]^-u & \fc_{\Omega}
		}
		\]
		where $u$ is the universal section from \Cref{d:univ-section}.
	\end{definition}
	
	\begin{secnumber}
		Once again, we now switch to the  analytic setting and consider all of the above schemes as analytic adic spaces over $K$ without changing the notation. In particular, we consider $J$ as a relative adic group $J\to X\times \mathbf A$.  By passing to quotient stacks, the $G$-equivariant morphism $\gI_{\Omega}\to \fg_{\Omega}$ then descends to a morphism of $\tau$-stacks $[\gI_{\Omega}/G]_{\tau}\to [\fg_{\Omega}/G]_{\tau}$ over $X$. This is still a relative group because $\gI\to \fg$ is.
		
		Second, the $G$-equivariant homomorphism $a:\chi^*\gJ\to \gI$ from \eqref{eq:map a} induces a homomorphism
		\begin{equation}\label{eq:aOmega}
	 	[a_\Omega]:[\chi_{\Omega}]^*\gJ_{\Omega}\to [\gI_{\Omega}/G]_{\tau}
	 	\end{equation}
	 	over $[\fg_{\Omega}/G]_{\tau}$,
		where $[\chi_{\Omega}]:[\fg_{\Omega}/G]_{\tau}\to \fc_\Omega$ is the map from \Cref{d:Hitchin-map}. This is the analogue in our setting of the map in \cite[Proposition~3.3]{Ngo}. Explicitly, in terms of the moduli description of $[\gI_{\Omega}/G]_{\tau}$ given in \Cref{r:moduli-quotient-stack-equivar-version}, this map is associated to the natural $G$-torsor $\chi_\Omega^*\gJ_{\Omega}\to [\chi_{\Omega}]^*\gJ_{\Omega}$ together with the $G$-equivariant map $a_\Omega:\chi_\Omega^*\gJ_\Omega=(\chi^\ast \gJ)_\Omega\to \gI_\Omega$ obtained from $a$ by twisting with $\Omega$.
	\end{secnumber}

	\subsection{Abelianisation in terms of $J$}
		Let $T\in \Perf_K$, let $Y\in X_{T,\et}$ and $(E,\varphi)$ a $\tau$-$G$-Higgs bundle on $Y$. Via \Cref{l:CHig-as-sheaf-of-sections}, this corresponds to a section $s_{E,\varphi}: Y\to [\fg_{\Omega}/G]_\tau$ over $X$. 
		We denote by $b$ the composition
		\begin{equation}\label{d:def-b}
			b:Y\xrightarrow{ s_{E,\varphi}}[\fg_{\Omega}/G]_\tau\xrightarrow{[\chi_{\Omega}]} \fc_{\Omega}.
		\end{equation}
		If $Y=X_T$, then by \Cref{d:Hitchin-map}, this section is precisely $b=h(E,\varphi)$.
		\begin{definition}
		We denote by $J_b\to Y$ the smooth relative group given by the pullback of $\gJ_{\Omega}$ along $b$:
		\begin{equation} \label{eq:defJb}
			\xymatrix{
				J_b\ar[r] \ar[d] & \gJ_{\Omega} \ar[d] \\
				Y \ar[r]^b & \fc_{\Omega}
			}
		\end{equation}
		When $Y=X_T$, this is equivalently the pullback of $J\to X\times \mathbf A$ along the map $X_T\to X\times \mathbf A$.
		\end{definition}
		
		Note that $J_b=b^\ast\gJ_{\Omega}= s_{E,\varphi}^\ast[\chi_{\Omega}]^*\gJ_{\Omega}$. Therefore, the pullback of \eqref{eq:aOmega} along $s_{E,\varphi}^\ast$ defines a canonical morphism of sheaves of groups on $Y_\tau$:
		\begin{equation} \label{eq:Jb to Aut-v1}
			s_{E,\varphi}^\ast[a_\Omega]: J_b\to  s_{E,\varphi}^*[\gI_{\Omega}/G]_{\tau}, 
		\end{equation}
	\begin{prop} \label{p:Jb to Aut}
		In the above situation, we have a canonical isomorphism  \[s_{E,\varphi}^*[\gI_{\Omega}/G]_{\tau}\simeq \uAut(E,\varphi).\]
	\end{prop}
	\begin{proof}
		By Definitions \ref{d:Aut-Higgs} and \ref{d:quotient-stack}, we need to see that the sections $Y\to E\times^G\mathcal I_{\Omega}$ correspond naturally to automorphisms of $E$ preserving $\varphi$. To compute $E\times^G\mathcal I_{\Omega}$, consider the $G$-equivariant commutative diagram
			\[
		\begin{tikzcd}
			\gI\ar[r] \arrow[d] & G\times \fg \ar[d] &(g,x)\arrow[d,mapsto]\\
			\fg \arrow[r,"\Delta"] & \fg\times \fg& (x,\ad(g)(x))
		\end{tikzcd}
		\] defining $\mathcal I$. We now twist this diagram over $\G_m$ with $\Omega$ and over $G$ with $E$, then the result is clearly still Cartesian. Recall that $E\times^G\mathfrak g_{\Omega}=\ad(E)\otimes \Om$. Second, since the $G$-action on the first factor $G$ on the top right is via conjugation, one verifies directly that $E\times^GG=\uAut(E)$. Consequently, the diagram becomes
		\[
		\begin{tikzcd}
			E\times^G\gI\ar[r] \arrow[d] & \uAut(E)\times \ad(E)\otimes \Om \ar[d] &(\psi,x)\arrow[d,mapsto]\\
			\ad(E)\otimes \Om \arrow[r,"\Delta"] & \ad(E)\otimes \Om\times \ad(E)\otimes \Om& (x,\ad(\psi)(x)).
		\end{tikzcd}
		\]
		Hence $E\times^G\gI_\Omega$ consists of pairs of an automorphism $\psi$ of $E$ and a section of $\ad(E)\otimes \Om $ fixed by $\psi$. The result follows as by \Cref{l:CHig-as-sheaf-of-sections}, the section of $\ad(E)\otimes \Om$ associated to $Y\to [\mathcal I_{\Omega}/G]_\tau\to [\mathfrak g_{\Omega}/G]_\tau$ is $\varphi$.
	\end{proof}
	\begin{definition}\label{d:J-b-action}
		Let $T\in \Perf_K$, let $Y\in X_{T,\et}$ and let $(E,\varphi)$ be a $\tau$-$G$-Higgs bundle on $Y$ with associated map $b:Y\to \fc_{\Omega}$ as in \eqref{d:def-b}.	
		Composing the map \eqref{eq:Jb to Aut-v1} with the isomorphism from \Cref{p:Jb to Aut}, we obtain a canonical and functorial homomorphism of sheaves over $Y_{\tau}$
		\begin{equation} \label{eq:Jb to Aut}
			a_{E,\varphi}: J_b\to  \uAut(E,\varphi).
		\end{equation}
		If $\tau=\et$, then the composition $J_b\to  \uAut(E,\varphi)\to \uAut(E)$ is represented by a homomorphism of smooth relative groups over $Y$. In particular, it represents a morphism of v-sheaves on $Y$, where we can identify the v-sheaf represented by $\uAut(E)$ with $\uAut(\nu^\ast E)$. We shall still denote this morphism of v-sheaves by $a_{E,\varphi}$.
	\end{definition}
	Using the canonical Higgs field from \Cref{t:canonicalHiggs},
	we arrive at an analogous morphism for v-$G$-bundles:
	\begin{prop}\label{c:a_E-for-v-bundles}
		Let $V$ be a v-$G$-bundle on $X_{T}$ with $b=\widetilde{h}(V)\in \A(T)$, where $\wt h$ is from \Cref{d:Hitchin-Betti}. Let $J_b\to X_T$ be the smooth relative group  of \eqref{eq:defJb}.
		Then there exists a canonical homomorphism on $X_{T,v}$
		\begin{equation} \label{eq:aE}
			a_V: J_b\to \uAut(V),
		\end{equation}
		that can be uniquely characterised as follows: Locally on any \'etale map from a smoothoid $Y\to X_T$ that admits a toric smooth chart $f$ such that  $V$ becomes small, the following diagram of sheaves on $Y_{v}$ commutes
		\[
		\begin{tikzcd}[row sep = 0cm,column sep = 1.2cm]
			& {\uAut(E)}  \\
			J_b \arrow[ru, "{a_{E,\theta}}"] \arrow[rd, "a_{V}"'] &                                      \\
			& \uAut(V) \arrow[uu, "\LS_f^{-1}"']     
		\end{tikzcd}\]
		 where $(E,\theta)=\LS^{-1}_f(V)$ is the Higgs bundle corresponding to $V$ via the local correspondence of \Cref{t:small-corresp}. 
	\end{prop}
	\begin{proof}
		Let $\theta_V$ be the canonical Higgs field of \Cref{t:canonicalHiggs}.
		\Cref{d:J-b-action} applied to the v-$G$-Higgs bundle $(V,\theta_V)$ yields a natural morphism 
		\[ a_V:J_b\to \uAut(V,\theta_V)\]
		of v-sheaves.
		The composition with the forgetful map $\uAut(V,\theta_V)\to \uAut(V)$ yields the desired morphism  $a_V$. Note that this forgetful map is an isomorphism over $Y_\et$ by \Cref{t:canonicalHiggs}.(1).(a). 
		It remains to see the commutativity of the diagram: It suffices to check this on the toric v-cover $Y_\infty\to Y$. By \Cref{t:canonicalHiggs}.(1).(b), $\LS_f$ induces an isomorphism $\nu^*E|_{Y_\infty}=V|_{Y_\infty}$ that identifies $\theta$ with $\theta_V$. This implies the commutativity.
	\end{proof}
	
	\subsection{The tautological section $\tau$} \label{ss:tau} 
		We now review the construction of the canonical section $\tau:\fc \to \Lie \gJ$ following \cite[\S 2.3]{CZ15}:
		The Lie algebra $\Lie (\gI|_{\fg^{\reg}})\subset \fg\times \fg^{\reg}$ of the relative group scheme $\gI|_{\fg^{\reg}}\to \fg^{\reg}$ admits a canonical section 
		$$ \fg^{\reg}\to \Lie (\gI|_{\fg^{\reg}})\subseteq \mathfrak g\times \fg^{\reg},\quad x\mapsto (x,x).$$
		This section is clearly $\G_m$-equivariant for the natural $\G_m$-action on both sides which is given by homotheties on each factor.
		Via the Kostant section, using that $\kos^\ast\Lie \gI=\Lie\gJ$, this induces a $\G_m$-equivariant section  
		\[\tau:\fc\to \Lie \gJ\]
		where the $\G_m$-action on $ \Lie \gJ$ is the one induced by the one on $\Lie (\gI|_{\fg^\reg})$.
		
		Recall now from \eqref{eq:map-a-regular-locus} that over $\fg^\reg$, there exists a canonical isomorphism $\chi^*\gJ|_{\fg^{\reg}}\xrightarrow{\sim} \gI|_{\fg^{\reg}}$. 
		Let $x\in \fg$, and let $a_{x}:\gJ_{\chi(x)}\to \gI_x \subset G$  be the fibre of this isomorphism over $x$.
		Then by \cite[Lemma 2.2]{CZ15},
		\begin{equation}
			da_x (\tau (\chi(x)))= x.
			\label{eq:tau}
		\end{equation}
		We note that \cite{CZ15} works over a finite field, but the proof still works over $K$.
	
	\begin{secnumber} \label{sss:twist-tau}
		Due to the $\G_m$-equivariance, all of these constructions are compatible with twisting:
		The $\Omega^{\times}$-twist $(\Lie \gJ)\times^{\Gm} \Omega^{\times}$ is isomorphic to $\Lie (\gJ_{\Omega})\otimes \pi^*(\Omega)$ viewed as a vector bundle over $\fc_{\Omega}$,
		where $\pi:\fc_{\Omega}\to X$ is the canonical morphism. Twisting $\tau$ with $\Omega$, we thus obtain a canonical section over $\fc_{\Omega}$:
		\begin{equation} \label{eq:tauOmega}
			\fc_{\Omega}\to \Lie \gJ_{\Omega}\otimes \pi^*(\Omega). 
		\end{equation}
		By \cite[Proposition 4.13.2 and its proof]{Ngo10}, the Lie algebra of the smooth group scheme $\gJ_{\Omega}$ over $\fc_{\Omega}$ is
		\[
		\Lie \gJ_{\Omega} \simeq \pi^*( \fc_{\Omega}^{\vee}\otimes \Omega),
		\]
		where $(-)^{\vee}$ is the dual vector bundle. 
		For the group scheme $J$ from \Cref{d:def-univ-J},  this implies that
		\begin{equation}
			\label{eq:splitting}
			\Lie J = u^{\ast}\Lie \mathcal J_\Omega \simeq \pr_X^*(\fc_{\Omega}^{\vee}\otimes \Omega)
		\end{equation}
		as $\pi\circ u:X\times A\to X$ is given by $\pr_X$.
		Since $K$ is algebraically closed, so $G$ is split, we have an isomorphism
		\begin{equation} \label{eq:cOmega-explicit}
			\fc_{\Omega}^{\vee}\simeq \Omega^{-e_1}\oplus \Omega^{-e_2}\oplus \cdots \oplus \Omega^{-e_r},
		\end{equation}
		where $e_1,\cdots,e_r$ denote the degrees of the invariant polynomials of $\fg$. 
		\begin{definition}
		Consider the geometric vector bundle over the Hitchin base ${\A}={\A}_{G,\Omega}$ \[{\A}_{J,\Omega}:={\A}\times \rH^0({X}, \fc_{\Omega}^{\vee}\otimes \Omega^{\otimes 2})\to \A.\]
		\end{definition}
		\begin{lemma}
		For any $T\in \Perf_K$, we have a canonical isomorphism
		\[ {\A}_{J,\Omega}(T)=\rH^0(X,\fc_{\Omega}^{\vee}\otimes \Omega^{\otimes 2})\otimes_K \O(T)=\rH^0(X_T,\fc_{\Omega}^{\vee}\otimes \Omega^{\otimes 2})=\rH^0(X_T,\Lie J_b\otimes \Omega).\]
		\end{lemma}
		\begin{proof}
			This follows from flat base change applied to $X\to \Spa(K)$, see \cite[Theorem 3.18.2(b)]{perfectoid-base-change}.
		\end{proof}
		\begin{definition}\label{eq:twist-tau}
		Since $\mathbf A$ parametrises sections of $\fc_{\Omega}$ over $X$,
		the global sections of the morphism \eqref{eq:tauOmega} over $X_T$ for each $T\in \Perf_K$ can thus be assembled to a canonical section of the vector bundle ${\A}_{J,\Omega}\to \A$
		\[\tau_{\Omega}:{\A}\to {\A}_{J,\Omega}.\]
	\end{definition}
	\end{secnumber}
	The crucial point is now that we can use the $J_b$-action on $E$ of \Cref{d:J-b-action} to recover the Higgs field:
	\begin{prop}\label{p:can-section-tau_theta}
		Let $T\in\Perf_K$ and let $(E,\varphi)$ be a $\tau$-$G$-Higgs bundle on $X_T$ with Hitchin image $b=h_\tau(E,\varphi):T\to \A$. 
		Then the canonical section $\tau_{\varphi}:=\tau_{\Omega}(b)\in\rH^0(X_T,\Lie J_b\otimes \Omega)$ has the property that
		\[da_{E,\varphi} \otimes \Omega:\Lie J_b\otimes \Omega\to \ad(E)\otimes \Omega\quad \text{ sends }\tau_\varphi\mapsto \varphi.\]
	\end{prop}
	
	\begin{proof}
		Since $X_T$ is reduced, we can verify the equality $da_{E,\varphi}(\tau_\varphi)=\varphi$ by checking that it holds in every geometric fibre $x:\Spa(C,C^+)\to X_T$. After choosing a local trivialisation of the pullbacks of $E$ and $\Omega$ to $x$,  we can identify $\varphi$ with a section of $\fg_{\Omega}$. Geometrically, this is a lift of the point $\Spa(C,C^+)\to [\fg/G]$ defined by $x^\ast(E,\varphi)$ to a point $\varphi_x:\Spa(C,C^+)\to\fg$. Then \eqref{eq:tau} says that $x^\ast da_{E,\varphi}(\tau_{\Omega}(b))=da_{x}(\tau(\chi(\varphi_x)))=\varphi_x$.
	\end{proof}
	
	\subsection{The case of $G=\GL_n$}\label{s:Ngo-G=GL_n}
	By way of example, assume now that $G=\GL_n$. In this case, the constructions of this section can be made more explicit in terms of the spectral curve, as we will now explain.
	
	We begin by describing $\A$: For $G=\GL_n$, we can identify $\fc$ with $\Spec(K[b_1,\cdots,b_n])$, where the $\Gm$-action on $b_i$ has weight $i$. Consequently,
	\[ \textstyle\A=\prod_{i=1}^{n}\rH^0(X,\Om^{\otimes i})\otimes \G_a.\]
	Consider the finite morphism of degree $n$:
	\[
	\mathfrak l:=\Spec( K[b_1,\cdots,b_n,T]/(T^n-b_1 T^{n-1}+\cdots+(-1)^n b_n)) \to \fc.
	\]
	Then we have $\gJ\simeq \Res_{\mathfrak l/\fc}\Gm$. In particular, for any $b\in \fc(K)$, we have a $\Gm$-equivariant decomposition
	\[
	\textstyle\Lie \gJ_b \simeq \bigoplus_{i=0}^{n-1} K T^i,
	\]
	where the $\Gm$-action on $T^i$ has weight $-i$. Note that the right hand side is independent of $b$.

	Let now $T\in \Perf_K$ and set $X_T=X\times T$. Let $b=(b_1,\dots,b_n)\in \A(T)$ be a point of the Hitchin base, i.e.\ a tuple consisting of $b_k\in \Om^{\otimes k}(X_T)$. Let now $(E,\theta)$ with $\theta:E\to E\otimes \Om$ be a Higgs bundle on $X_T$ with Hitchin image $b$, then we can think of the $b_i$ as the coefficients of the characteristic polynomial of $\theta$. 
	
	Indeed, the morphism $\theta:\O\to \uEnd(E)\otimes \Om$ induces a morphism of $\O_{X_T}$-algebras
	\[\theta:\Sym_{\O_{X_T}}\Om^\vee\to  \uEnd(E).\]
	
	Let $I_b$ be the ideal sheaf of $\Sym_{\O_{X_T}}\Om^\vee$ generated by the image of the morphism
	\[ \Om^{\otimes -n}\to \Sym\Om^\vee, \quad f\mapsto \textstyle\sum_{i=1}^n fb_i\]
	where $fb_i$ is considered as a section of $\Om^{ \otimes i-n}$.
	\begin{definition}\label{d:spectral-curve}
		The cover
		\[\pi:Z_b:=\underline{\Spa}_{X_T}(\Sym\Om^\vee/I_b)\to X_T\]
		is called the \textit{spectral curve}. By the local description, it is a finite flat cover of $X_T$ of degree $n$. As it lives in the cotangent bundle over $X_T$, there is on $Z_b$ a tautological differential that we denote by $\tau_{\can}\in \pi^{\ast}\Om$. By construction, there is a natural map
		\begin{equation}\label{eq:G=GL_n-action-B-on-E-varphi}
		\theta:\pi_{\ast}\O_{Z_b}\to \uEnd(E).
	\end{equation}
	\end{definition}
	
	Applied to  the universal case of $Y=\A$, the above construction results in the  universal spectral curve 
	\[ \pi:Z\to X\times  \A\]
	from which we recover $Z_b$ as the pullback along $b:T\to \A$.
	The composition \[\pi':Z\to X\times \A\to \A\] with the projection $X\times \A\to \A$
	is proper and all its fibres have dimension one. Moreover, there is an open dense locus $\A^\circ\subseteq \A$ over which $\pi'$ is smooth proper with geometrically connected fibres. Let 
	\[B:=\pi_{\ast}\O_{Z} \text{ on } (X\times \A)_{\et}.\]
	Then the relative rigid group $J\to X\times \A$ of \Cref{d:def-univ-J} can in this case be described as representing the subsheaf of units  $B^\times$. In fact,
	unravelling the definitions, we see:
	\begin{lemma}\label{l:action-Jb-G=GL_n}
		Let $T\in \Perf_K$ and let $(E,\varphi)$ be a Higgs bundle on $X\times T$. Then we have $J_b=\pi_{\ast}\O_{Z_b}^\times$. Under this identification, the morphism $a_{E,\varphi}$ from \Cref{d:J-b-action} is the homomorphism
		\[ \pi_{\ast}\O_{Z_b}^\times\to \uAut(E)\]
		given by the units of the ring morphism $\theta$ described in \eqref{eq:G=GL_n-action-B-on-E-varphi}.
	\end{lemma}
	Let us set $\B:=\nu^\ast B$, then when we regard $J$ as a v-sheaf, it is identified with the sheaf $\B^\times$ on $(X\times \A)_v$.
	\begin{secnumber} \label{sss:tau-GLn}
		Finally, let us explicitly describe the sections $\tau$ and $\tau_{\Omega}$ in the case of $G=\GL_n$. 
		Let $b\in \fc$, and let $x=\kos(b)\in \fg^{\reg}$ its Kostant section. Then the derivative of \eqref{eq:map a} at $x\in \fg^{\reg}$ is
		\begin{equation} \label{eq:JbIx}
			\textstyle da_x:\Lie \gJ_b\xrightarrow{\sim} \Lie \gI_x,\quad \sum c_i T^i \mapsto \sum c_i x^i,\quad T\mapsto x.	
		\end{equation}
		By \eqref{eq:tau} and \eqref{eq:JbIx}, we deduce that $\tau(b)$ corresponds to $1\cdot T\in \Lie \gJ_b$. 
		
		Second, by \eqref{eq:splitting} and \eqref{eq:cOmega-explicit}, the fibre of ${\A}_{J,\Omega}$ at $b$ is for $G=\GL_n$ given by \[{\A}_{J,\Omega,b}=\oplus_{i=0}^{n-1}\rH^0(X,\widetilde{\Omega}_X \otimes \widetilde{\Omega}_X^{\vee \otimes i})=\rH^0(X,\widetilde{\Omega}_X)\oplus \rH^0(X,\O_X) \oplus \cdots \oplus \rH^0(X,\widetilde{\Omega}_X^{\vee \otimes n-2}).\]
		It now follows from \eqref{eq:JbIx} after twisting with $\Omega$ that $\tau_{\Omega}$ is given by $\tau_{\Omega}(b)=(0,1,0,\cdots,0)\in {\A}_{J,\Omega,b}$.
	\end{secnumber}

	\section{The non-abelian Hodge correspondence for v-stacks}\label{s:naHC-vstacks}
	
	We now combine the preparations from all previous sections to prove our first main result, \Cref{t:intro-main-thm}.

		\subsection{The stacky relative Hodge--Tate sequence for $J$} 
		We keep the assumptions of \S\ref{ss:HitchinNgo}, i.e.\ $K$ over $\mathbb Q_p$ is complete algebraically closed, $X$ is a smooth projective curve over $K$ and $G$ is a connected reductive group.  
		As before, to simplify notation, we simply denote by $\Omega$ the line bundle $\wtOm_{X}$ on $X$ from \Cref{d:differentials-smoothoids}.
	\begin{secnumber}\label{s:setup-5}
	Let $J\to X\times \A$ be the commutative smooth relative group from \Cref{d:def-univ-J} and let $\whJ\to X\times \A$ be its maximal topologically $p$-torsion subgroup in the sense of \Cref{p:log-of-smooth-group}. 
		
	Recall that an object of $\A_v$ is a perfectoid space $T\in \Perf_K$ with a map $b:T\to \A$. Given such an object, let $X_T:=X\times T$ and let $\whJ_b\to X_T$ be the pullback of $\whJ$ via $b$. By \Cref{p:log-of-smooth-group}.(5), this is the maximal topologically $p$-torsion subgroup of the relative adic group $J_b\to X_T$ from \eqref{eq:defJb}.

		\begin{definition}
			Let $\CP_v\to \mathbf A$ be the v-stack on $\A_v$ of v-$\whJ$-bundles, defined by
			\[ \CP_v:=\CBun_{\whJ,v}:(b:T\to \A)\mapsto \{\text{$\whJ_b$-torsors on $(X_T)_v$}\}.\]
			The contracted product (see \S\ref{sss:contracted}) defines an operation $\otimes$ on  $\CP_v$ turning it into a Picard stack (\S\ref{s:Picard-stacks}) on $\A_v$.
		\end{definition}
		The Picard stack $\CP_{v}$ admits a natural Hodge--Tate logarithm: 	Recall from \S\ref{sss:twist-tau} that there is a vector bundle  ${\A}_{J,\Omega}\to \A$ whose sections over $b$ are given by $\rH^0(X_T,\Lie J_b \otimes_{\O_X} \Om)$. This is a commutative smooth relative group over $\A$, and in particular we may regard it as a Picard stack on $\A_v$.
		\begin{definition}\label{eq:HJb}
		There is a natural homomorphism of Picard stacks on $\A_v$ 
		\[
			\widetilde{h}_J: \CP_v\to \A_{J,\Omega},
		\]
		defined as follows: For any $b:T\to \A$, we send any v-$\wh J_b$-bundle on $X_T$ first to its isomorphism class in  $\rH^1_v(X_T,\wh J_b)$ and then to the image under the Hodge--Tate logarithm 
		of \Cref{c:leray-seq-for-whG}:
		\begin{equation} \label{eq:hUJb}
			\HTlog_{\wh J_b}: \rH^1_v(X_T,\wh J_b)\to \rH^0(X_T,\Lie J_b \otimes_{\O_X} \Om).
		\end{equation}
		\end{definition}
	\end{secnumber}
	\begin{definition}
		Recall from \Cref{eq:twist-tau} the canonical section $\tau_{\Omega}:\A\to \A_{J,\Omega}$.
		We define the v-stack 
		$\CH\to \A$ as the fibre 
		\[ \CH:=\wt h_J^{-1}(\tau_{\Omega})=\A\times_{\tau_{\Omega}, \A_{J,\Omega}} \CP_v.\]
		In other words,
		an object in $\CH(b:T\to \A)$ is a v-$\wh J_b$-bundle on $X_{T,v}$ with Hitchin image $\tau_{\Omega}(b)$ in $\A_{J,\Omega}(T)$. 
	\end{definition}
		There is also an \'etale variant of the Picard stack $\CP_{v}$: 
		\begin{definition}
			We define a prestack $\CP_\et\to \mathbf A$ by
			\[ \CP_\et:=\CBun_{\whJ,\et}:(b:T\to \A)\mapsto \{\text{$\wh J_b$-torsors on $X_{T,\et}$}\}.\]
		\end{definition}
		To simplify notation, we also write $\CP:=\CP_{\et}$. We can now formulate the main result of this section:
		\begin{theorem}\label{t:H P-torsor}
			\begin{enumerate}
			\item $\CP$ is a Picard stack on $\A_v$ and the natural maps
			\[1\to \CP\to  \CP_v\xrightarrow{	\widetilde{h}_J} \A_{J,\Omega}\to 0\]
			define a short exact sequence of Picard stacks (\Cref{d:picstack-ses}).
			\item The v-stack $\CH$ is a $\CP$-torsor.
			\end{enumerate}
		\end{theorem}
		\begin{proof}
			It follows from \Cref{c:leray-seq-for-whG} (or more precisely, \eqref{eq:Leray-et-v}) that $\CP$ is exactly the fibre of $0$ of the homomorphism $\wt h_{J}$ in \Cref{eq:HJb}. In particular, it is itself a v-stack. This also shows the left-exactness. For (1), it thus suffices to prove that $\wt h_J$ is essentially surjective. By \Cref{d:picstack-ses-induces-torsor}, this will also imply (2).

	Let $f:X\times \A\to \A$ be the projection and recall from \S\ref{s:setup} that $\mu:\A_{v}\to \A_{\Et}$ denotes the natural morphism of sites.
	Then the essential surjectivity follows from the following proposition:
	
	\begin{prop}\label{p:leray-seq-for-UJ}
		The Leray spectral sequence induces a short exact sequence of abelian sheaves on $\A_v$
		\begin{equation} \label{eq:nu-UJ}
			1\to \mu^* (\rR^1f_{\Et\ast}\whJ) \to \rR^1f_{v\ast}\whJ \xrightarrow{\HTlog} f_{v\ast}(\Lie  J\otimes \Omega)\to 0.
		\end{equation}
		Let $E:=\fc_{\Omega}^{\vee}\otimes \Om$, a vector bundle on $X$. Then the last term is isomorphic to the rigid vector group 
		\[f_{v\ast}(\Lie  J\otimes \Om)=\rH^0(X,E\otimes \Om)\otimes_K \mathcal O_\A\]
	\end{prop}
	\begin{proof}
		The left-exact sequence is obtained by applying \Cref{c:leray-seq-for-whG} to  $J\to X\times \A$.
		Recall from \eqref{eq:splitting} that the vector bundle $\Lie J$ on $X\times \A$ is isomorphic to $\pr_X^{\ast}E$, where $\pr_X:X\times \A\to X$ is the projection. We can thus apply \cite[Theorem~3.18.(a)]{perfectoid-base-change} to $g:=b:T\to \A$ to see that
		\[
		f_{v*} (\Lie  J \otimes \Om)\simeq \rH^0(X,E\otimes \Om) \otimes_K \O_{\A}. 
		\]
		It remains to see the right-exactness of \eqref{eq:nu-UJ}, for which we use the following:
		\begin{lemma}\label{l:RnfvastLieUJ}
			Let $E=\fc_{\Om}^{\vee}\otimes \Om$. Then for any $n\in \N$, we have
			\[\rR^nf_{v\ast}\Lie  J=\rH^n_v(X,E)\otimes \O_{\A},\quad \rR^nf_{\Et\ast}\Lie  J=\rH^n_\et(X,E)\otimes \O_{\A}.\]
		\end{lemma}
		\begin{proof}
			Since $\Lie  J=\pr_X^{\ast}E$ by \eqref{eq:splitting}, the statement  follows from \cite[Corollary~3.10]{perfectoid-base-change}.
		\end{proof}
		
		We now consider $\log_{J}$ which by \Cref{l:log-surj-alg-case} and \Cref{p:log-of-smooth-group}.(3)  fits into a short exact sequence
		\begin{equation}\label{eq:exp-for-UJ}
			0\to J[p^\infty]\to \whJ\xrightarrow{\log} \Lie  J\to 0
		\end{equation}
		where on the right we regard $\Lie  J$ as a rigid vector group over $X\times \A$. 
		Applying \Cref{c:leray-seq-for-whG} to the commutative smooth relative group $\Lie  J\to X\times \A$, we obtain a left exact sequence
		\begin{equation}\label{eq:Lie-ses-1}
		1 \to  \mu^*(\rR^1f_{\Et\ast}\Lie  J) \to    \rR^1f_{v\ast}\Lie  J  \to f_{v*} (\Lie  J \otimes \Om) \to 0 .
		\end{equation}
		It follows from \Cref{l:RnfvastLieUJ}
		that this is isomorphic to the short exact sequence of rigid vector groups
		\begin{equation}\label{eq:Lie-ses-2}
			0\to \rH^1_{\et}(X,E)\otimes \O_\A\to  \rH^1_{v}(X,E)\otimes \O_\A\to \rH^0(X,E\otimes \Om)\otimes \O_\A\to 0,
		\end{equation}
		and is in particular exact. 
		Indeed, any choice of flat lift $\bX$ of $X$ over $\BdR^+/\xi^2$ induces a splitting $\mathrm s_{\bX} $ of the last map in \eqref{eq:Lie-ses-2}, see \cite[Proposition~2.15]{Heu23}. We note that such a choice of a lift $\bX$ always exists by \cite[Proposition 7.4.4]{Guo}.
		Thus $\bX$ induces a splitting $s_{\mathbb{X}}$ of \eqref{eq:Lie-ses-1}.

		\begin{lemma}\label{l:proper-bc-U_Jp^n}
			The morphism $J[p^n]\to X\times \A$ is \'etale. Moreover, for any $m\in \N$, we have
			\[ \rR^mf_{v\ast}J[p^n]=\nu^{\ast}\rR^mf_{\et\ast}J[p^n]\]
			where $\nu:\A_v\to \A_\et$ is the natural map. Here the sheaf $\rR^mf_{\et\ast}J[p^n]$ is Zariski-constructible.
		\end{lemma}
		\begin{proof}
			The first part holds by \Cref{c:G[p]-etale}. Since $J\to X\times \A$ is algebraic by definition, it follows from \cite[Theorem~3.7.2]{huber2013etale} that, $\rR^mf_{\et\ast}J[p^n]$ is  algebraic and moreover Zariski-constructible. 
			The displayed isomorphism now follows by an application of \cite[Proposition 5.8.2]{PCT-char-p}.	
		\end{proof}
		\begin{prop}\label{p:log-seq-R^1fv}
			Applying $\rR^1f_{v\ast}$ to the logarithm sequence \eqref{eq:exp-for-UJ}
			induces a short exact sequence of sheaves on $\A_v$
			\[
				1\to \nu^\ast\rR^1f_{\et\ast}(J[p^\infty]) \to \rR^1f_{v\ast}\whJ \to \rR^1f_{v\ast}\Lie J\to 0.\]
		\end{prop}
		\begin{proof}
			It is clear that we obtain a long exact sequence, and it thus suffices to prove that for any $m\geq 0$, the boundary morphisms
			\[ \partial: \rR^mf_{v\ast}\Lie J\to \rR^{m+1}f_{v\ast}J[p^\infty]\]
			vanishes. To see this, we use that by \Cref{l:RnfvastLieUJ}, $\rR^mf_{v\ast}\Lie  J$ is an affine vector group represented by $ \G_a^d$ for some $d\in \N$. On the other hand, by \Cref{l:proper-bc-U_Jp^n}, the sheaf $\rR^{m+1}f_{v\ast}J[p^n]=\nu^{\ast}\rR^{m+1}f_{\et\ast}J[p^n]$ is Zariski-constructible. It follows from this that any map of the form $\partial$ is constant: Indeed, any section of $\rR^{m+1}f_{\et\ast}J[p^n]$ vanishes if it vanishes on fibres by \cite[Proposition~2.6.1]{huber2013etale}. Let now $V$ be any Zariski-open subset of $\A$ on which $\rR^{m+1} f_{\et *}J[p^n]$ is locally constant, then $\partial_{|V}$ is represented by a morphism of relative rigid groups over $V$. 
			Since $ \G_a^d$ is connected, we deduce that $\partial|_V$ is zero. 
			By repeating this argument over the complement of $V$ in $\A$, we conclude $\partial=0$, as we wanted to see.
		\end{proof}
		
		We now apply \Cref{c:leray-seq-for-whG} to \eqref{eq:exp-for-UJ}, which by naturality  yields a commutative diagram:
		\begin{equation} \label{eq:diagram-exp-HTlog}
			\begin{tikzcd}
				1 \arrow[r] & {\mu^*(\rR^1f_{\Et\ast}\Lie  J)} \arrow[r]& {\rR^1f_{v\ast}\Lie  J} \arrow[r] & f_{v\ast}(\Lie J\otimes \Om) \arrow[r] & 0\\
				1 \arrow[r] & \mu^*(\rR^1f_{\Et\ast}\whJ) \arrow[r] \arrow[u]       & \rR^1f_{v\ast}\whJ \arrow[r,"\HTlog"] \arrow[u,"\log"]       & f_{v\ast}(\Lie J\otimes \Om) \arrow[u,equal] \arrow[r] &0
			\end{tikzcd}
		\end{equation}
		The middle vertical map is surjective by \Cref{p:log-seq-R^1fv}. The map on the top right is surjective because it admits a splitting $ s_{\bX}$. 
		Hence the bottom  morphism is surjective and show the exactness of \eqref{eq:nu-UJ}.
	\end{proof}
	This finishes the proof of \Cref{t:H P-torsor}.
\end{proof}
	\begin{rem}
		In order to give a first indication of the role of  $\tau_{\Omega}$ from \Cref{eq:twist-tau} in this context,
		let $b\in \A(K)$, then $\mathscr{H}(b)$ is a $\wh J_b$-gerbe on $X_{\et}$. 
		Consider the Leray five term sequence of $\wh J_b$ for $X_v\to X_\et$:
		\[
		0\to \rH^1(X_{\et},\wh J_b) \to \rH^1(X_{v},\wh J_b) \to \rH^0(X,\Lie J_b \otimes_{\O_X} \Om)
		\xrightarrow{\lambda} \rH^2(X_{\et},\wh J_b) \to \rH^2(X_{v},\wh J_b).
		\]
		Then one can show that the class of $\mathscr{H}(b)$ in $\rH^2(X_{\et},\wh J_b)$ equals to $\lambda(\tau_{\Omega}(b))$. 
	\end{rem}

	\subsection{Twisting Higgs bundles}
		\begin{secnumber} \label{sss:Picstack} 
		There is a natural action of $\CP$ on $h:\CHig_G\to\A$ (cf. \cite[\S4]{Ngo}):
		Let $b:T\to \A$ be in $\A_v$ and let $(E,\varphi)$ be a $G$-Higgs bundle in $\CHig_G(T)$ with Hitchin image $h(E,\varphi)=b$. We still denote by $a_{E,\varphi}$ the restriction of the homomorphism from \Cref{d:J-b-action} to the subgroup $\wh J_b\subseteq  J_b$,
		\begin{equation}\label{eq:whJ-action-on-E}
		a_{E,\varphi}:\wh J_b\to J_b\to \uAut(E,\varphi).
		\end{equation}
		We can use this to twist $(E,\varphi)$ by any $\wh J_b$-bundle $F$ to obtain a new $G$-Higgs bundle 	$(F \times^{\wh J_b}E,F\times^{\wh J_b}\varphi)$: To see that this is well-defined, observe that the pushout $F\times^{\whJ_b}\uAut(E)$ along $a_{E,\varphi}$ is an $\uAut(E)$-torsor, so 
		\begin{equation}\label{eq:action-on-Higgs-via-twist}
		F\times^{\wh J_b}E=(F\times^{\whJ_b}\uAut(E)) \times^{\uAut(E)}E
		\end{equation}
		is a contracted product of bitorsors in the sense of \Cref{d:contracted-prod}. This shows that $F\times^{\wh J_b}E$ is again a $G$-torsor.
		Second, the $J_b$-action fixes $\varphi$, hence $\varphi$ considered as a morphism $\O\to \ad(E)\otimes \Om$ is $\wh J_b$-equivariant.
	\end{secnumber}
	\begin{secnumber}\label{s:action-on-V}
		Second, there is also a natural action of $\CP$ on $\wt h:\CBun_{G,v}\to \A$: Let $b:T\to \A$ be in $\A_v$ and let $V$ be a v-$G$-bundle on $X_T$ with $\wt h(V)=b$. By \Cref{c:a_E-for-v-bundles}, there is a natural action
		\[a_V:\wh J_b\to J_b\to \uAut(V).\]
		Like in \S\ref{sss:Picstack}, we can use this to twist $V$ by any \'etale $\wh J_b$-bundle $F$ to obtain a new v-$G$-bundle	
		$V':= \nu^*F \times^{\wh J_b}V$. Since $F$ is \'etale,  $V'$ has the same Hitchin image as $V$, so this indeed defines an action over $\A$.
	\end{secnumber}
	 We can now state the main result of this article, using the notion of twists from \Cref{d:picstack-twist}:
	\begin{theorem} \label{t:fCiso}
		There exists a canonical equivalence of v-stacks over $\A$, functorial in $X$,
		\begin{equation} \label{eq:fC}
			\CS:\CH\times^{\CP} \CHig_G \isomarrow \CBun_{G,v}.
		\end{equation}
	\end{theorem}	
	\begin{proof}
		Let $T$ be an affinoid perfectoid space over $K$, let $(E,\varphi)$ be a $G$-Higgs bundle on $X_T$ and let $b:T\to \mathbf A$ be its Hitchin image.  Exactly as in \Cref{d:J-b-action}, we still denote by $a_{E,\varphi}$ the composition of \eqref{eq:whJ-action-on-E} with the inclusion
		\[a_{E,\varphi}:\wh J_b\to \uAut(E,\varphi)\hookrightarrow \uAut(E),\]
		a morphism of v-sheaves represented by smooth relative groups.
		Note that $\nu^*E$ is an $(\uAut(E),G)$-bitorsor on $X_{T,v}$. In particular, $a_{E,\varphi}$ defines a left-action by $\wh J_b$ on $\nu^\ast E$.
		Let now $F$ be an object of $\CH(b)$, i.e.\ $F$ is a v-$\wh J_b$-bundle on $X_T$ with Hitchin image $\tau_{\Omega}(b)$. We will define $\CS$ in terms of  the  twist 
		\begin{equation} \label{eq:defC}
			(F,(E,\varphi))\mapsto F\times^{\wh J_b} \nu^* E
		\end{equation}
		(see \S\ref{sss:contracted}).
		This is now a v-$G$-bundle on $X_T$: Indeed, exactly as in \eqref{eq:action-on-Higgs-via-twist}, we can rewrite the twist as the contracted product of bitorsors in the sense of \Cref{d:contracted-prod}: 
		$F\times^{\wh J_b} \nu^* E=(F\times^{\wh J_b}\uAut(E))\times^{\uAut(E)}\nu^* E$.
		This shows that \eqref{eq:defC} defines a functor $\CH\times \CHig_G\to \CBun_{G,v}$. 
		
		To see that this induces the desired functor $\CS$, we need to compute the effect of the $\CP$-action:
		Let $Q\in \CP(T)$ be any $\whJ_b$-bundle on $X_{T,\et}$. Let $(E_Q:=Q\times^{\whJ_b}E,\varphi_{Q})$ be the $G$-Higgs bundle defined by the action of $Q$ on $(E,\varphi)$, see \S\ref{sss:Picstack}. By associativity of twists, we then have a natural isomorphism 
		\begin{equation}\label{eq:Q-equivar-of-CS}
		(F\times^{\wh J_b} Q)\times^{\whJ_b}	\nu^*E  
		\simeq F\times^{\wh J_b} (Q\times^{\whJ_b}	\nu^*E ) 
		\simeq  F\times^{\whJ_b} 	\nu^*E_Q.
		\end{equation}
		This shows that \eqref{eq:defC} induces the desired morphism of v-stacks $\CS:\CH\times^{\CP} \CHig_G\to \CBun_{G,v}$.
		
		It remains to prove that $\CS$ is an equivalence of categories and that $\CS$ commutes with the structure morphisms to $\A$. For either, we now show that  $\CS$ is compatible with the local $p$-adic Simpson correspondence:
	\begin{prop}\label{l:LC-via-twisting}
		Let $b:T\to \A$ be an object of $\A_v$ and let $Y\in X_{T,\et}$ be a toric smoothoid space, equipped with a fixed toric chart $Y\to \mathbb T\times T$. 
		Let $(E,\theta)$ be a $G$-Higgs bundle on $X_T$ such that $h(E,\theta)=b$ and let $F$ be an object of $\CH(b)$. 
		Then after replacing $Y$ by an \'etale cover, there exists an isomorphism of v-$G$-bundles
		\[
		\LS_f(E,\theta)\cong  F\times^{\whJ_b}  \nu^*E,
		\]
		where the left-action of $\whJ_b$ on $\nu^\ast E$ is via the homomorphism $a_{E,\varphi}:\whJ_b\to \uAut(E)$ from \S\ref{sss:Picstack}.
	\end{prop}
	\begin{proof}
		After replacing $Y$ by an \'etale cover, we may assume that $(E,\theta)$ is small with $E\cong G$. 
		By functoriality of the exponential map from  \Cref{p:nbhd-open-subgroup-ball}, we obtain a commutative diagram 
		\begin{equation} \label{eq:exp-commute}
		\begin{tikzcd}
			\whJ_b \arrow[r,"a_{E,\theta}"]                                & \uAut(E)                          \\
			\Lie J_b \arrow[r,"da_{E,\theta}"] \arrow[u, "\exp", dotted] & \ad(E) \arrow[u, "\exp", dotted]
		\end{tikzcd}\end{equation}
		where the dotted arrows mean that the maps are both defined on an open neighbourhood of the identity.
		
		Recall from \Cref{p:can-section-tau_theta} that we have a tautological section $\tau_\theta\in \rH^0(X_T,\Lie J_b\otimes \Om)$. As explained in \S\ref{s:data-induced-by-toric-chart}, we can use the chart $f$ to associate to this a continuous homomorphism
		\[ \rho_f(\tau_\theta):\Delta_f\to \Lie J_b(Y).\]
		After replacing $\Delta_f$ by an open subgroup, and thus $Y$ by the corresponding finite \'etale cover, we may assume that $\rho_f(\tau_\theta)$ has image in the open subgroup where the exponential converges. 
		Consider the v-$\whJ_b$-torsor $\mathcal F$ on $Y$ associated to the 1-cocycle $\exp(\rho_f(\tau_\theta)):\Delta_f\to \whJ_b(Y)$. Explicitly, this is defined for any $W\in Y_v$ by
		\begin{alignat*}{3}
			\mathcal F(W)&:=&&\big\{ s\in \whJ_b(Y_{\infty}\times_Y W)\big| \gamma \cdot s=\exp(-\rho_f(\tau_\theta)(\gamma))s,\quad \forall \gamma\in \Delta_f\big\}\\
		\Rightarrow\quad (\mathcal F\times^{\whJ_b}\nu^\ast E)(W)&=&&\big\{ s\in E(Y_{\infty}\times_Y W)\big| \gamma \cdot s=a_{E,\theta}(\exp(-\rho_f(\tau_\theta)(\gamma)))s,\quad \forall \gamma\in \Delta_f\big\}.
		\end{alignat*}
		Observe now that by \Cref{p:can-section-tau_theta}, we have $da_{E,\theta}(\tau_{\theta})=\theta$. It follows that for any $\gamma\in \Delta_f$, we have  $da_{E,\theta}(\rho_f(\tau_\theta)(\gamma))=\rho_f(\theta)(\gamma)$.
		We deduce from this and \eqref{eq:exp-commute} that
		\[a_{E,\theta}(\exp(-\rho_f(\tau_\theta)(\gamma)))=\exp(da_{E,\theta}(-\rho_f(\tau_\theta)(\gamma)))=\exp(-\rho_f(\theta)(\gamma)).\]
		Comparing to the explicit definition of $\LS_f(E,\theta)$ in \eqref{eq:LS},  it follows that
		\[  \mathcal F\times^{\whJ_b}\nu^\ast E=	\LS_f(E,\theta).\]
		It remains to compare $ \mathcal F$ and $F$. It is clear from the definition that 
		$\HTlog(\mathcal F)=\tau_\theta=\HTlog(F)$.
		By the Leray sequence of $\whJ$, \eqref{eq:Leray-et-v} and \Cref{p:leray-seq-for-UJ}, this implies $\mathcal F\simeq F$ on some \'etale cover of $Y$. 
	\end{proof}
	
	\begin{coro} \label{l:hfC}
		In the context of \eqref{eq:defC}, we have $\widetilde{h}(F\times^{\whJ_b}\nu^\ast E)=b$.
	\end{coro}
	
	%
	%
	
	We now continue with the proof of Theorem \ref{t:fCiso}: To see that $\CS$ is an equivalence, consider the stack
		\[\CH^{-1}=\A\times_{-\tau_{\Omega},\A_{J,\Omega}}\CBun_{\whJ,v}\]
		of v-$\whJ$-bundles with Hitchin image $-\tau_{\Omega}$. 
		By \Cref{t:H P-torsor}, this is a $\CP$-torsor, and it is an inverse of $\CH$ in the sense that the contracted product induces a canonical equivalence
		\[ \CH\times^{\CP}\CH^{-1}\simeq \CP.\]
		It therefore now suffices to construct an inverse to $\CH^{-1}\times^{\CP}\CS$ of the form
		\[\CH^{-1}\times^{\CP} \CBun_{G,v}\to \CHig_G.\]
		To this end, let $b:T\to \A$ be in $\A_v$, let $F\in\CH^{-1}(b)$ and let $V$ be an object of $\CBun_{G,v}(T)$ with $\wt h(V)=b$. Similar to the construction of $\CS$, we use the homomorphism $a_V:\wh J_b\to \uAut(V)$ from  (\ref{s:action-on-V}) to define
		\begin{equation} \label{eq:actionC-1}
			(F,V)\mapsto \widetilde{V}:=F\times^{\whJ_b} V.
		\end{equation}
		Exactly as in \S\ref{sss:Picstack},
		it is clear that $\widetilde{V}$ is a v-$G$-bundle on $X_T$. The crucial point is now:
		\begin{lemma}
			The v-$G$-bundle $\widetilde{V}$ is \'etale-locally trivial, so $E:=\nu_{\ast}\widetilde{V}$ is a $G$-torsor on $X_{T,\et}$ with $\nu^{\ast}E=\wt V$.
		\end{lemma}
		\begin{proof}
			Let $Y\in X_{T,\et}$ be toric  with a chart $f$ such that $V_{|Y}$ is small. By \Cref{l:loc-small}, we can find an \'etale cover of $X_{T}$ by such $Y$. Then by \Cref{t:small-corresp} there is a small $G$-Higgs bundle $(E,\theta)$ over $Y$ such that $V_{|Y}\simeq \LS_f(E,\theta)$. By \Cref{l:LC-via-twisting}, it follows that after replacing $Y$ by an \'etale cover, we have
			\[V|_Y\simeq \LS_f(E,\theta)\simeq  F^{-1}\times^{\whJ_b} \nu^*E.\]
			In particular, the canonical isomorphism $\LS_f:\uAut(E,\theta)\to \uAut(V|_Y)$  is given by the natural $\uAut(E)$-action on $\nu^*E$ on the left. 
			By \Cref{c:a_E-for-v-bundles}, this isomorphism identifies $a_{V}$ with $a_{E,\theta}$. Hence
			\[ \wt{V}|_Y\simeq F \times^{\whJ_b}V |_Y\simeq  F \times^{\whJ_b}F^{-1}\times^{\whJ_b} \nu^*E\simeq \nu^\ast E.\qedhere\]
		\end{proof}
		The canonical Higgs field $\psi_{V}$ on $V$ from \Cref{t:canonicalHiggs} induces a Higgs field on $\widetilde{V}$. Exactly as in \Cref{sss:Picstack}, one sees that this induces a Higgs field $\theta$ on $E$. 
		 We now define the morphism $\CS^{-1}$ by sending $(F,V)$ to $(E,\theta)$. 
		The construction is clearly functorial, and one sees as in \eqref{eq:Q-equivar-of-CS} that it factors through the quotient by the antidiagonal $\CP$-action.
		Since $\psi_{V}$ and $\theta$ have the same Hitchin image, $\CS^{-1}$ is a morphism over $\A$. 
		
		One now easily checks from the definition that $\CH^{-1}\times^{\CP}\CS$ and $\CS^{-1}$ are inverse to each other:
		Indeed, given an object $F$ of $\CH^{-1}$ and $V$ in $\CBun_{G,v}$, we have $F^{-1}\times^{\whJ_b} F\times^{\whJ_b} V\simeq V$. 
	Hence $(\CH^{-1}\times^{\CP}\CS)\circ\CS^{-1} $ is isomorphic to the identity map. 
		The other direction can be seen in the same way. 
	\end{proof}

\subsection{Functorialities of $\CS$ in the reductive group}
We have already stated as part of
\Cref{t:fCiso} the functoriality of $\CS$ in $X$.
In this subsection, we discuss the functoriality of $\CS$ with respect to homomorphism of reductive groups $f:G\to H$, which is more subtle as it requires compatibility of Kostant sections.
\begin{secnumber}
	Let $f:G\to H$ be a homomorphism of connected reductive groups over $K$ and let $df:\fg\to \fh$ be the associated morphism of Lie algebras. 
	This induces a natural morphism of invariants $\frak{c}_{G}\to \frak{c}_H$, compatible with Chevalley maps. Consequently, there is a natural morphism of centralisers $\mathcal{I}_G\to \mathcal{I}_H$ over $f$ and $df$:
	\[\begin{tikzcd}
		\mathcal I_G\arrow[r]\arrow[d]&\fg \arrow[r,"\chi_{\fg}"]\arrow[d,"df"']&\fc_G\arrow[d]\\
		\mathcal I_H\arrow[r]&\fh \arrow[r,"\chi_{\fh}"]&\fc_H
		\end{tikzcd}\]
Moreover, by twisting $\fc_G\to \fc_H$ with $\Omega$, we also obtain a natural morphism of Hitchin bases
$\A_G \to \A_H$. 

	Suppose now that the Kostant sections of $G$, $H$ are compatible with $f$ in the sense that the diagram
	\begin{equation}\label{eq:compat-Kostant}
	\xymatrix{
	\fg \ar[r]^{df} & \fh \\
	\frak{c}_{G} \ar[u]^{\kos_G} \ar[r] & \frak{c}_H \ar[u]_{\kos_H}
	}
	\end{equation}
	commutes.
	Then the morphism $\mathcal{I}_G\to \mathcal{I}_H$ induces a homomorphism over $\fc_G$ of regular centralisers:
	\begin{equation} \label{eq:J-functoriality}
		\mathcal{J}_G \to \mathcal{J}_H\times_{\fc_H}\fc_G.
\end{equation}
	The tautological section $\tau_G:\fc_G \to \Lie \mathcal{J}_G$ (\S\ref{ss:tau}) is defined by the diagonal map of $\fg^{\reg}$ and the Kostant section, and is therefore compatible with the derivative of \eqref{eq:J-functoriality} in the natural way.

	\begin{rem}\label{r:det-Kostant} The diagram \eqref{eq:compat-Kostant} does not always commute. However, we do have commutativity in many  case of interest,  for example for the determinant $\det:\GL_n\to \Gm$ and the canonical embeddings $\SO_{2n+1}\to \SL_{2n+1}$, $\Sp_{2n} \to \SL_{2n}$, $\SO_{2n+1}\to \SO_{2n+2}$. In fact, we have the following criterion:
		\end{rem}
\end{secnumber}

\begin{lemma}\label{l:criterion-Kostant}
	The Kostant sections are compatible with $f:G\to H$ if and only if there is a regular nilpotent element in $\fg$ that is sent to a regular nilpotent element of $\fh$. 
\end{lemma}
\begin{proof}
	We refer to \cite[\S 3.2]{chriss1997representation} for properties of regular nilpotent elements. 
	Suppose  $e\in \fg$ is a regular nilpotent element such that  $e':=df(e)\in \fh$ is regular nilpotent.
	Let $\{e,\check{\rho},\widetilde{e}\}$ be a principal $\sl_2$ of $\fg$. Then $\{e', \check{\rho}'=f(\check{\rho}), \widetilde{e}'=f(e)\}$ is a principal $\sl_2$ of $\fh$. 
	Recall from \cite[Théorème 2.1.3]{Ngo} that the Kostant section identifies $\fc_G$ (resp. $\fc_H$) with the subspace $e+\fg^{ \widetilde{e} }$ of $\fg$ (resp. $e'+\fh^{\widetilde{e}'}$ of $\fh$). 
	Thus $f$ is compatible with Kostant sections. 
	The converse is clear by considering the image of $0\in \fc$. 
\end{proof}
\begin{prop}\label{p:functoriality-of-CS}
	Let $f:G\to H$ be a homomorphism of reductive groups over $K$ such that $f$ is compatible with Kostant sections. 
	Then there is a natural 2-commutative diagram of v-stacks over $\A_G \to\A_{H}$:
	\[\begin{tikzcd}[row sep =0.4cm]
		\CH_G\times^{\CP_G} \CHig_G \arrow[r,"\CS_G"] \ar[d] &  \CBun_{G,v} \ar[d] \\
		\CH_H \times^{\CP_H} \CHig_H \arrow[r,"\CS_H"] & \CBun_{H,v}
	\end{tikzcd}
	\]
\end{prop}
\begin{proof}
	The natural pushout functor defines morphism of stacks over $\Perf_{K,v}$:
	\[
	\CHig_G\to \CHig_H,\quad \CBun_{G,v}\to \CBun_{H,v}.
	\]
	These morphisms are compatible via the Hitchin maps in the natural way. 

	By twisting \eqref{eq:J-functoriality}, we obtain a natural homomorphism 	$
	J_G \to J_H \times_{\A_H}\A_G$ over $X\times \A_G$.
By pushout along this homomorphism, we thus obtain natural morphisms of Picard stacks over $\A_{G,v}$:
	\[
	\CP_G \to \CP_H\times_{\A_H} \A_G,\quad \CH_G \to \CH_H\times_{\A_H}\A_G.
	\]
	These are compatible with the actions of $\CP_-$ on $\CH_{-}$, $\CHig_{-}$ and $\CBun_{-,v}$ for $-\in \{G,H\}$.
	Since contracted products are compatible with pushout of torsors, we deduce the proposition. 
\end{proof}

	\subsection{The twisted isomorphism of coarse moduli spaces}
	Passing from v-stacks to sheaves of isomorphism classes, we get a version of \Cref{t:fCiso} for coarse moduli spaces:
	\begin{definition}
		Let $\bfBun_{G,v}$ be the v-sheafification of the presheaf given by sending $T\in \Perf_{K,v}$ to the set of isomorphism classes in $\CBun_{G,v}(T)$, i.e.\ to $\rH^1_v(X_T,G)$. 
		We similarly define a v-sheaf $\bfHig_G$ of isomorphism classes of $G$-Higgs bundles on $X\times T$ up to v-sheafification in $T$. 
		Both sheaves admit Hitchin maps to $\A$.  
	\end{definition}
	\begin{definition}
		Let $\mathbf P:=\mu^\ast R^1f_{\Et\ast}\wh J$, this is the sheaf on $\A_v$ obtained from $\CP$ by passing to isomorphism classes and sheafifying.
		Let $\mathbf H$ be the v-sheaf on $\A_v$ obtained in the same way from $\CH$.
	\end{definition}
	\begin{coro}\label{c:twisted-isom-moduli}
		The v-sheaf $\mathbf H$ is a $\mathbf P$-torsor on $\A_v$ and there is a canonical isomorphism of v-sheaves
		\[\cH\times^{\cP}\bfHig_G\xrightarrow{\sim} \bfBun_{G,v}.\]
	\end{coro}
	\begin{proof}
		The first part follows from \Cref{t:H P-torsor}, the second from \Cref{t:fCiso} by sheafifying.
	\end{proof}
	We now explain that one can extract from our proof also a finer variant of this isomorphism, which is however less  canonical. Namely, for the formulation, we choose a flat lift $\bX$ of $X$ over $\BdR^+/\xi^2$.  
	\begin{definition}\label{d:H_X}
		Let $\mathbf P_{\bX,v}\subseteq \rR^1f_{v\ast}\whJ$ be the sub-v-sheaf defined as the equaliser of the two morphisms in \eqref{eq:diagram-exp-HTlog} defined by:
		\[
		\mathbf{P}_{\mathbb{X},v}:= \Eq\biggl(\rR^1f_{v\ast}\whJ \xrightrightarrows[\log]{s_\bX\circ \HTlog}  \rR^1f_{v\ast}\Lie J \biggr).
		\]
	\end{definition}
	We then have the following analogue of \Cref{t:H P-torsor}, which we can deduce from its proof:
	\begin{prop}\label{c:H_XX}
		We have a pullback diagram of short exact sequences of sheaves on $\A_v$
		\begin{equation} \label{eq:H-red-strgrp}
			\begin{tikzcd}
				1 \arrow[r] & \cP[p^\infty]\arrow[d,equal] \arrow[r]& {\mathbf P_{\bX,v}} \arrow[r,"\HTlog"]\arrow[d]& {\A}_{J,\Omega} \arrow[r]\arrow[d,"s_{\bX}"] & 0\\
				1 \arrow[r] & \nu^\ast\rR^1f_{\et\ast}J[p^\infty] \arrow[r]      & \rR^1f_{v\ast}\whJ \arrow[r,"\log"]        & \rR^1f_{v\ast}\Lie J \arrow[r] &0
			\end{tikzcd}
		\end{equation}
	\end{prop}
	\begin{proof}
		By \Cref{p:log-seq-R^1fv}, we have $\cP[p^\infty]=\nu^\ast\rR^1f_{\et\ast}J[p^\infty]$.
		The map $\HTlog$ in \eqref{eq:H-red-strgrp} is surjective by the proof of \Cref{p:log-seq-R^1fv}. It follows that the kernel of $\HTlog$ is contained in $\cP[p^\infty]$. On the other hand, $\cP[p^\infty]$ is clearly contained in $\mathbf{P}_{\mathbb{X},v}$. Hence $\ker\HTlog=\cP[p^\infty]$. This also shows the pullback property.
	\end{proof}
	\begin{lemma} \label{c:bfLX}
		The morphism $\cP_{\bX,v}\to \mathcal{\A}_{J,\Omega}$ is a torsor under $\cP[p^\infty]$. In particular, it is of the form $\cP_{\bX,v}=\nu^\ast P_{\bX}$ for some ind-constructible sheaf $P_{\bX,v}$ on $\A_{\et}$.
	\end{lemma}
	\begin{proof}
		The first part follows from \Cref{c:H_XX}. The second part follows from the fact that by \Cref{l:proper-bc-U_Jp^n} and \cite[Proposition~14.7-14.8]{Sch18}, the pullback morphism induces an isomorphism
		\[ \rH^1_{\et}(\A,\rR^1f_{\et\ast}J[p^\infty])\to \rH^1_{v}(\A,\rR^1f_{v\ast}J[p^\infty]).\qedhere\]
	\end{proof}
	
	\begin{definition}\label{def:cHX}
		Let  $\cH_{\bX}\to \A$ be the fibre of $\cP_{\bX,v}\to \mathcal{\A}_{J,\Omega} $ over $\tau_{\Omega}$.
	\end{definition}
	\begin{lemma}\label{l:cHbX}
		$\cH_{\bX}$ is a  $\cP[p^\infty]$-torsor on $\A_v$ and there is a natural isomorphism 	$\cH_{\bX}\times^{\cP[p^{\infty}]}\cP=\cH$.
	\end{lemma}
\begin{proof}
	The first part follows from \Cref{c:bfLX}, the second from considering the fibre of \eqref{eq:H-red-strgrp} over $\tau_\Omega$.
\end{proof}
	\begin{coro}\label{t:coarse-moduli-twisted}
		There is a canonical isomorphism of v-sheaves on $\A$
		\[\cH_{\bX}\times^{\cP[p^{\infty}]}\bfHig_G\xrightarrow{\sim} \bfBun_{G,v}.\]
	\end{coro}
	\begin{proof}
		Combining \Cref{c:twisted-isom-moduli} and \Cref{l:cHbX}, we have
		\[\cH_{\bX}\times^{\cP[p^{\infty}]}\bfHig_G=\cH_{\bX}\times^{\cP[p^{\infty}]}{\cP}\times^{\cP}\bfHig_G=\cH\times^{\cP}\bfHig_G=\bfBun_{G,v}.\qedhere\]
	\end{proof}
	The following is an alternative formulation of \Cref{t:coarse-moduli-twisted} in terms of pullbacks instead of twists:
	\begin{coro}\label{c:comp-coares-moduli-pullback}
	There is a canonical isomorphisms of v-sheaves on $\A$
	\[
	\bfHig_G\times_{\A}\cH_{\bX}\xrightarrow{\sim} \bfBun_{G,v}\times_{\A}\cH_{\bX}.\]
	\end{coro}
	\begin{proof}
		Since $\cH_{\bX}$ is a $\cP[p^{\infty}]$-torsor over $\A$, there exists a canonical isomorphism $\cH_{\bX}\times_\A \cH_{\bX}\simeq \cP[p^{\infty}]\times_\A \cH_{\bX}$. 
		The Corollary follows from applying $\times_{\A}\cH_{\bX}$ to both sides of \Cref{t:coarse-moduli-twisted}.
	\end{proof}
	\subsection{The case of $G=\GL_n$}\label{s:CS-for-GL_n}
	Finally, let us make $\CS$ slightly more explicit in the case of $G=\GL_n$. Let $T\in \Perf_K$, let $(E,\theta)$ be a Higgs bundle on $X_T$ with Hitchin image $b:T\to \A$ and let $L\in \CH(b)$. Let $\pi:Z_b\to X_T$ be the spectral curve over $b$ from \Cref{d:spectral-curve} and recall from \S\ref{s:Ngo-G=GL_n} that we can describe the v-sheaf represented by $J_b$ as being $\mathcal B^\times$ where $\mathcal B:=\nu^\ast\pi_{\et\ast}\O_{Z_b}$. Then we can regard $L$ as being a $\B^\times$-torsor on $X_T$, or in other words, an invertible $\B$-module. Using \Cref{l:action-Jb-G=GL_n} to describe $a_{E,\varphi}$, we deduce:
	\begin{lemma}\label{l:CS-for-GLn}
	Under the above identifications, $\CS$ sends $L$ and $(E,\varphi)$ to the v-vector bundle on $X_T$
	\[\CS(L,(E,\varphi))=\nu^\ast E\otimes_{\B}L.\]
	\end{lemma}
	From this perspective, at the heart of our $p$-adic Simpson functor lies a twisting construction which generalises that of \cite{Heu23} from $\GL_n$ to general reductive groups, as well as to perfectoid families.
	
	In the easiest special case of $\GL_n=\G_m,$ we have $Z_b=X_T$ and hence $\B=\O_{X_T,v}$. Thus $\CH(b)$ is then given by the v-line bundles $L$ on $X_T$ with $\HTlog(L)=b$, and $\CS$ is thus given by twisting with $L$.
	
	For $G=\G_m$, the v-sheaf $\bfBun_{G,v}$ is the v-Picard variety of \cite{Heu22b}, which is represented by a rigid group. In this setting, \Cref{c:comp-coares-moduli-pullback} was previously proven in \cite[Theorem~5.4]{Heu22b}, and the above explicit description shows that these isomorphisms agree. Consequently, we may regard the Corollary as a generalisation of this result to higher rank and further to reductive $G$.
	
	To understand the precise relation of our moduli-theoretic $p$-adic Simpson correspondence \Cref{t:fCiso} to the categorical $p$-adic Simpson correspondences  of \cite{Fal05}\cite{Heu22b}\cite{Heu23}, the goal of the next section is to provide a new moduli-theoretic perspective on the role of the exponential.
	
	\section{The moduli space of exponentials} \label{s:exponentials}
	We now explain how we can derive from the twisted isomorphism of \Cref{t:fCiso} an equivalence of categories over strictly totally disconnected test objects, depending on the choice of an exponential. 

	For simplicity, let us first restrict attention to the case of $G=\GL_n$. In the notation of \S\ref{s:Ngo-G=GL_n},
	 the first main goal of this section is to prove the following structure result about the v-sheaf $\rR^1f_{v\ast}\whJ=\rR^1f_{v\ast}\wh{\B^\times}$:
	\begin{theorem}\label{t:struct-thm-R1fvastU}
		Let $\Lambda:=\rR^1\pi'_{v\ast}\uZp$ be the \'etale cohomology of the spectral curve $\pi':Z\to \A$. Then:
		\begin{enumerate}
			\item 
			There is a natural morphism of short exact sequences of v-sheaves on $\A$:
			\[\begin{tikzcd}
				1 \arrow[r] & {\rR^1f_{v\ast}J[p^\infty]} \arrow[r] \arrow[d,"\sim"] & \rR^1f_{v\ast}\whJ \arrow[r,"\log"] \arrow[d]     & \rR^1f_{v\ast}\Lie J \arrow[d] \arrow[r]  & 0 \\
				1 \arrow[r] & \Lambda\otimes_{\uZp}\mu_{p^\infty} \arrow[r]         & \Lambda\otimes_{\uZp}\widehat{\G}_m \arrow[r,"\log"] & \Lambda\otimes_{\uZp}\G_a \arrow[r] & 0
			\end{tikzcd}\]
			\item The square on the right is a pullback square.
			\item  The sheaf on the left is isomorphic to the \'etale sheaf $\varinjlim_n\nu^{\ast}(\rR^1\pi'_{\et\ast}\mu_{p^n})$.
		\end{enumerate}
	\end{theorem}
	\begin{coro}\label{c:exp-splits-L_X}
		There is a natural  Cartesian square of v-sheaves over $\A$:
		\[\begin{tikzcd}
			\cH_{\bX} \arrow[r] \arrow[d]     & \A\arrow[d] \\
			\Lambda\otimes_{\uZp}\widehat{\G}_m \arrow[r,"\log"] & \Lambda\otimes_{\uZp}\G_a
		\end{tikzcd}\]
	\end{coro}
	\begin{proof}
		This follows from the tower of commutative diagrams
		\[\begin{tikzcd}
			\cH_{\bX} \arrow[r] \arrow[d] & \cP_{\bX} \arrow[d,"\HTlog"] \arrow[r] &\rR^1f_{v\ast}\whJ \arrow[d,"\log"] \arrow[r] & 	 \Lambda\otimes_{\uZp}\widehat{\G}_m\arrow[d,"\log"] \\
			\A \arrow[r,"\tau_\Omega"]             & \A_{J,\Omega} \arrow[r,"s_{\bX}"]     & \rR^1f_{v\ast}\Lie  J \arrow[r]         & 	 \Lambda\otimes_{\uZp}{\G}_a             
		\end{tikzcd}\]
		in which the first square is Cartesian by \Cref{def:cHX}, the second square is Cartesian by \Cref{c:H_XX} and the third square is Cartesian by \Cref{t:struct-thm-R1fvastU}.
	\end{proof}
	We deduce that on the level of $K$-points, splittings of $\log$ induce a splitting of $\cH_\bX$. More generally:
	\begin{definition}
		Let $S=\Spa(R,R^+)$ be a strictly totally disconnected perfectoid space. \textit{An exponential for $S$} is a continuous splitting of the logarithm map $\log:1+R^{\circ\circ}\to R$.
	\end{definition}
	\begin{coro}\label{c:exp-induces-splitting}
		Let $S$ be a strictly totally disconnected space in $\A_v$. Then any exponential for $S$ induces a splitting of the torsor $\cH_{\bX}\times_\A S\to S$ over $S$. In particular, it induces a section of $\cH_\bX(S)\to \A(S)$.
	\end{coro}
	\begin{proof} 
		This is immediate from \Cref{c:exp-splits-L_X} 
		as evaluation on $S$ preserves the Cartesianess.
	\end{proof}
	
	\begin{rem}
		The third vertical map in \Cref{t:struct-thm-R1fvastU}.(1) is usually not an isomorphism, as we can see on fibres $\Spa(K)\to \A$: The reason is that for a finite flat morphism of rigid spaces $g:Z\to Y$, the map $\rH^1_v(Y,\nu^\ast g_{\ast}\O)\to \rH^1_v(Z,\O)$ is in general neither injective (e.g.\ $Z$ non-reduced) nor surjective (e.g.\ $g$ ramified).
	\end{rem}
	For the proof of the Theorem, we start with some preparations.

	\subsection{Étale cohomology of the spectral curve}\label{s:et-cohom-spectral-curve}
	For the results of this section, we can more generally let $\pi':Z\to \A$ be any proper finite type morphism of adic spaces over $K$ whose fibres are all of pure dimension one.
	Instead of considering $\Lambda:=\rR^1\pi'_{v\ast}\uZp$, we can without changes consider more generally for any $n\in \N$
	\[ \Lambda^n:=\rR^n\pi'_{v\ast}\uZp.\]
	\begin{definition}
		Let $\mathcal G$ be any $\uZp$-module on $\Perf_{K,v}$. For example, by \Cref{p:log-of-smooth-group}.(6), this $\mathcal G$ could be any topologically $p$-torsion rigid group like $\wh{\G}_m$ or $\G_a$. Then there is for any $n\in \N$ a natural morphism
		\[\varphi_\cG:\Lambda^n \otimes_{\uZp} \cG\to \rR^n\pi'_{v\ast}\cG,\] functorial in $\cG$, constructed as follows: There is a natural map on $\A_v$
		\[ \varphi_0:\FHom_{\uZp}(\uZp,\cG)\to \FHom_{\uZp}(\rR^n\pi'_{v\ast}\uZp,\rR^n\pi'_{v\ast}\cG)\]
		defined for any $S\in \A_v$ by sending any homomorphism $h:\underline{\Z}_{p|S}:=\uZp\times S\to \cG_{|S}:=\cG\times S$ over $S$ to the morphism on $S_v$ obtained by sheafifying $T\mapsto \big(\rH^n_v(Z\times_{\A}T,\pi'^\ast\underline{\Z}_{p|T})\xrightarrow{\pi'^\ast h_{|T}}\rH^n_v(Z\times_{\A}T,\pi'^\ast \cG_{|T})\big)$. Note that
		$\FHom_{\uZp}(\uZp,\cG)=\cG$.
		The map $\varphi_{\cG}$ is then induced by $\varphi_0$ via the adjunction of $\otimes$ and $\FHom$.
	\end{definition}
	The aim of this subsection is to show that $\varphi_{-}$ gives rise to the following isomorphisms:
	\begin{prop}\label{p:bottom-part-of-comp-fund-seq-to-log}	There is a natural isomorphism of short exact sequences
		\[
		\begin{tikzcd}
			1 \arrow[r] & {\Lambda^n} \otimes_{\uZp}\mu_{p^\infty} \arrow[r]   \arrow[d,"\sim"',"\varphi_{\mu_{p^\infty}}"]      & {\Lambda^n} \otimes_{\uZp}\widehat{\G}_m\arrow[r,"\log"] \arrow[d,"\sim"'," \varphi_{\widehat{\G}_m}"]     & {\Lambda^n} \otimes_{\uZp}\G_a\arrow[r] \arrow[d,"\sim"'," \varphi_{\G_a}"] & 0\\
			1 \arrow[r] & {\rR^n\pi'_{v\ast}\mu_{p^\infty}} \arrow[r]  & \rR^n\pi'_{v\ast}\widehat{\G}_m \arrow[r,"\log"] &\rR^n\pi'_{v\ast}\O_{Z} \arrow[r]& 0.
		\end{tikzcd}\]
	\end{prop}
	We begin with some lemmas on the v-cohomology of the spectral curve $\pi':Z\to \A.$
	\begin{lemma}\label{l:et-cohom-Z/p^k-spectral-curve}
		Let  $n,k,l\in \Z_{\geq 0}$.
		\begin{enumerate}
			\item We have
			$\rR^n\pi'_{v\ast}\Z/p^k\Z=\nu^{\ast}\rR^n\pi'_{\et\ast}\Z/p^k\Z$ as sheaves on $\A_v$.
			\item We have a short exact sequence
				\begin{equation} \label{eq:R1-modpk}
					0\to \rR^n\pi'_{v\ast}\Z/p^l \Z\xrightarrow{\cdot p^k} \rR^n\pi'_{v\ast}\Z/p^{k+l}\Z\to  \rR^n\pi'_{v\ast}\Z/p^{k}\Z\to 0.
				\end{equation}
			\item The natural map $\Lambda^n\to \varprojlim _k\rR^n\pi'_{v\ast}(\Z/p^{k}\Z)$ is an isomorphism.
			\item We have $\Lambda^n/p^k=\rR^n\pi'_{v\ast}(\Z/p^{k}\Z)$.
			\item The $\uZp$-module $\Lambda^n$ is $p$-torsionfree.
		\end{enumerate}
	\end{lemma}
	\begin{proof}
		Since $\pi'$ is a proper morphism of finite type, part (1) is an application of \cite[Corollary 5.5]{PCT-char-p}.
		
		To deduce (2), it thus suffices to prove the statement for $\rR^n\pi'_{\et\ast}$ instead of $\rR^n\pi'_{v\ast}$. It suffices to prove the vanishing of the boundary maps of the natural long exact sequence. For this it suffices by \cite[Proposition 2.6.1]{huber2013etale} to prove the vanishing in every geometric fibre of $\pi':Z\to \A$ over $\Spa(L,L^+)\to \A$. It thus suffices to see that for any proper rigid curve $C\to \Spa(L,L^+)$, the sequence
		\[ 
		0\to \rH^n_\et(C,\Z/p^l \Z)\to \rH^n_\et(C,\Z/p^{k+l}\Z)\to \rH^n_\et(C,\Z/p^k\Z)\to 0
		\]
		is exact. We may reduce to the case that $C$ is connected. Second, we we may assume that $L^+=L^\circ$ since finite \'etale sites are insensitive to passing from $(L,L^+)$ to $(L,L^\circ)$. We are thus in the setting of classical rigid spaces. Hence, for $n>2$, the sequence vanishes. For $n=0$, the statement is clear. For $n=2$, the statement follows from the following fact: For any $m\in \mathbb{N}$, 
		by \cite[Proposition 8.4.1.(2)]{FvdP}, \cite[\S9.2, Corollary~14]{BLR}, we have 
		\[\rH^2_\et(C,\Z/p^m\Z)\simeq \Pic(C)/p^m\Pic(C)\simeq (\Z/p^m\Z)^r\]
		for some $r\in \Z$.
		Finally, the case of $n=1$ follows from those of $n=0,2$ by the long exact sequence.
		
		Assertion (3) follows from (2) using that the v-site is replete and \cite[Proposition~3.1.10]{BS-proetale}. 
		
		Finally, assertions (4), (5) follow from (3) by applying $\varprojlim_{l\in \N}$ to \eqref{eq:R1-modpk}. 
	\end{proof}
	
	\begin{lemma}\label{c:PCT-spectral-curve}
		For any $n,k\in \Z_{\geq 0}$, the natural morphism of v-sheaves on $\A_v$
		\[(\rR^n\pi'_{v\ast}\Z/p^k\Z)\otimes_{\Z/p^k\Z}\O^+_X/p^k \to \rR^n\pi'_{v\ast}\O^+_{Z}/p^k\]
		is an almost isomorphism. In the limit over $k$, it follows that the following map is an isomorphism:
		\[\varphi_{\G_a}:\Lambda^n\otimes_{\uZp}\G_a\isomarrow \rR^n\pi'_{v\ast}\O_{Z}.\]
	\end{lemma}
	\begin{proof}
		The first part is an application of \cite[Theorem 4.3]{PCT-char-p}, a consequence of Scholze's Primitive Comparison Theorem. The second part follows using \Cref{l:et-cohom-Z/p^k-spectral-curve} by taking $\varprojlim_k$ and inverting $p$.
	\end{proof}
	
	\begin{proof}[Proof of \Cref{p:bottom-part-of-comp-fund-seq-to-log}]
		It is clear from functoriality that applying $\varphi_{-}$ to the $\log$ sequence defines the desired commutative diagram.
		The left vertical map is an isomorphism because by  \Cref{l:et-cohom-Z/p^k-spectral-curve}.(4), we have
		\[\textstyle \rR^n\pi'_{v\ast}\mu_{p^\infty}=\varinjlim_k (\rR^n\pi'_{v\ast}\Z/p^k\Z)\otimes_{\Z/p^k}\mu_{p^k}=\varinjlim_k\Lambda^n\otimes_{\uZp}\mu_{p^k}=\Lambda^n\otimes_{\uZp}\mu_{p^\infty}.\]
		The right vertical map is an isomorphism by  \Cref{c:PCT-spectral-curve}.
		The top sequence is short exact by \Cref{l:et-cohom-Z/p^k-spectral-curve}.(5). 
		
		We can now prove the result by induction on $n$: Assume that the bottom sequence is left-exact, which is clear for $n=0$. Then the middle arrow is an isomorphism by the 5-Lemma, and the exactness of the bottom row follows. We can deduce left-exactness for $n+1$, and continue inductively.	
	\end{proof}
	\subsection{Comparison along the spectral cover}\label{d:comp-spectral-curve}
	\begin{definition} \label{d:psi-prime}
		Let $\pi':Z\xrightarrow{\pi} Y\xrightarrow{f} S$  be any morphisms of rigid spaces over $K$ and let $\cG\to Z$ be any commutative smooth relative rigid group in the sense of \S\ref{HT-rel-grp}. Then there is a natural map
		\[ \psi'_\cG:\rR^1f_{v\ast}\FHom_Y(\uZp,\pi_{v\ast}\cG)\to \rR^1\pi'_{v\ast}\wh{\cG}, \]
		functorial in $\cG$,
		defined as follows: By \cite[Lemma~14.4]{Sch18}, we have $\pi_v^\ast\uZp=\uZp$, so we have a natural adjunction isomorphism
		\[\pi_{v\ast}\FHom_Z(\uZp,\cG)=\FHom_Y(\uZp,\pi_{v\ast}\cG).\]
		The Grothendieck spectral sequence for $\pi'=f\circ \pi$ therefore defines a natural map
		\[\rR^1f_{v\ast}\FHom_Y(\uZp,\pi_{v\ast}\cG)\to \rR^1\pi'_{v\ast}\FHom_Z(\uZp,\cG).\]
		This has the desired form because by \Cref{p:log-of-smooth-group}.(6), we have $\FHom_Z(\uZp,\cG)=\wh{\cG}$.
	\end{definition}
	We apply this construction to the setup of the spectral curve $\pi':Z\xrightarrow{\pi}X\times \A\xrightarrow{f:=\pr_\A} \A$. In fact, for the following definition, we can more generally allow $\pi:Z\to X\times \A$ to be any finite flat cover. 
	\begin{definition}\label{d:psi-R^1B^x->R^1pi'whGm}
		The natural base-change map
		\[ \mathcal B:=\nu^{\ast}\pi_{\et\ast}\O_{Z}\to \pi_{v\ast}\O_{Z}\]
		induces by passing to units and applying $\FHom(\uZp,-)$ a natural map of v-sheaves on $Y:=X\times \A$
		\[h:\wh{\mathcal B^\times}\to \FHom_{Y}(\uZp,\pi_{v\ast}\G_m).\]
		Here on the left, we again use \Cref{p:log-of-smooth-group}.(6). In summary, we have thus constructed a natural map
		\[ \psi_{\G_m}:\rR^1f_{v\ast}\widehat{\B^\times}\xrightarrow{\rR^1f_{v\ast}h}\rR^1f_{v\ast}\FHom_Y(\uZp,\pi_{v\ast}\G_m)\xrightarrow{\psi_{\G_m}'} \rR^1\pi'_{v\ast}\wh{\G}_m.\]
		In exactly the same way, we obtain for $\mathcal B$ and $\mathcal B^\times[p^\infty]$ two morphisms of sheaves on $\A_v$ 
		\[ \psi_{\G_a}:\rR^1f_{v\ast}\B\to \rR^1f_{v\ast}\FHom_Y(\uZp,\pi_{v\ast}\G_a)\xrightarrow{\psi_{\G_a}'} \rR^1\pi'_{v\ast}\G_a,\]
		\[ \psi_{\mu_{p^\infty}}:\rR^1f_{v\ast}\B^\times[p^\infty]\to \rR^1f_{v\ast}\FHom_Y(\uZp,\pi_{v\ast}\mu_{p^\infty})\xrightarrow{\psi_{\mu_{p^\infty}}'} \rR^1\pi'_{v\ast}\mu_{p^\infty}.\]
	\end{definition}
	\begin{prop}\label{p:top-part-of-comp-fund-seq-to-log}
		Let $\pi:Z\to X\times \A$ be the spectral curve.
		Then there is a commutative diagram
		\[
		\begin{tikzcd}
			1 \arrow[r] & {\rR^1f_{v\ast}\B^\times[p^\infty]} \arrow[r] \arrow[d,"\sim"',"\psi_{\mu_{p^\infty}}"] & \rR^1f_{v\ast}\widehat{\B^\times} \arrow[r,"\log"] \arrow[d,"\psi_{\G_m}"]     &\rR^1f_{v\ast}\B \arrow[r]\arrow[d,"\psi_{\G_a}"] & 0 \\
			1 \arrow[r] & {\rR^1\pi'_{v\ast}\mu_{p^\infty}} \arrow[r] & \rR^1\pi'_{v\ast}\widehat{\G}_m \arrow[r,"\log"]      & \rR^1\pi'_{v\ast}\O \arrow[r] & 0
		\end{tikzcd}\]
		of sheaves on $\A$
		with short exact rows. Moreover,
		the morphism $\psi_{\mu_{p^\infty}}$ is an isomorphism.
	\end{prop}
	\begin{proof}
		 The bottom row is short exact by \Cref{p:bottom-part-of-comp-fund-seq-to-log}.
		It is clear from functoriality of $\psi'$ and the construction that $\psi$ defines a morphism between these sequences. We observe that $\B^\times[p^\infty]=\nu^{\ast}\B^\times[p^\infty]$. For this reason, the v-topological version of proper base-change \cite[Corollary~5.5]{PCT-char-p} shows that
		\[\rR^1f_{v\ast}\B^\times[p^\infty]=\nu^{\ast}\rR^1f_{\et\ast}\B^\times[p^\infty],\quad \rR^1\pi'_{v\ast}\mu_{p^\infty}=\nu^{\ast}\rR^1\pi'_{\et\ast}\mu_{p^\infty}.\]
		Under this identification, using that $\whJ=\wh{\B^\times}$ and $\Lie J=\B$, we see that the top row is  short exact by \Cref{p:log-seq-R^1fv}.
		It remains to prove that the natural map
		$\rR^1f_{\et\ast}\B^\times[p^\infty]\to \rR^1\pi'_{\et\ast}\mu_{p^\infty} $
		is an isomorphism.
		
		For this, we use the Grothendieck spectral sequence for $\pi'=f\circ \pi$, which yields an exact sequence
		\[ 0\to \rR^1f_{\et\ast}\pi_{\et\ast}\mu_{p^n}\to \rR^1\pi'_{\et\ast}\mu_{p^n} \to f_{\et\ast}\rR^1\pi_{\et\ast}\mu_{p^n}.\]
		Since $\pi$ is finite, the last term vanishes by \cite[Corollary 2.6.6]{huber2013etale}. Finally, we have $\pi_{\et\ast}\mu_{p^\infty}=\B^\times[p^\infty]$.
	\end{proof}
	\begin{proof}[Proof of \Cref{t:struct-thm-R1fvastU}]
		Part (1) follows from combining \Cref{p:top-part-of-comp-fund-seq-to-log} and \Cref{p:bottom-part-of-comp-fund-seq-to-log}. For (2), it suffices to show that the left vertical map is an isomorphism, which follows from the corresponding statements in \Cref{p:top-part-of-comp-fund-seq-to-log} and \Cref{p:bottom-part-of-comp-fund-seq-to-log}.
		Part (3) follows from \Cref{l:proper-bc-U_Jp^n}.
	\end{proof}
	
	\subsection{Comparison along the cameral cover}
	We now generalise \S\ref{d:comp-spectral-curve} to a reductive group $G$ over $K$. 
	\begin{secnumber}
		We first briefly review the cameral cover following \cite{DG02,Ngo10,CZ17}. Recall that $T$ is a fixed maximal torus of $G$, that $\ft$ is its Lie algebra and $\rW$ denotes the Weyl group of $G$, which acts on $T$. 

		Via the Chevalley isomorphism $K[\ft]^{\rW}\simeq K[\fg]^G$, we have a $\Gm$-equivariant, finite flat morphism $\ft\to \fc$ of degree $ |\!\rW\!|$, equipped with a $\rW$-action. 
		We twist this map with $\Omega$ and consider its base change
		\[
		\xymatrix{
		\widetilde{X} \ar[r] \ar[d]_{\pi} & \ft_{\Omega} \ar[d] \\
		X\times \A \ar[r] & \fc_{\Omega}
		}
		\]
		with the universal section $X\times \A\to \fc_{\Omega}$. 
		The morphism $\pi$ is the universal cameral cover, and is equipped with a $\rW$-action. 
		We set $\widetilde{\pi}=f\circ \pi:\widetilde{X}\to \A$. 
		There exists a commutative smooth relative group $J^1\to X\times \A$:
		\begin{equation}\label{eq:def-J1}
		J^1:=(\pi_*(T\times \widetilde{X}))^{\rW},
		\end{equation}
		where the $\rW$-action is given by the diagonal action on $T\times \widetilde{X}$, see \cite[\S~2.4]{Ngo10}, \cite[\S~3.1]{CZ17}. 

		By \cite[Proposition 2.4.2]{Ngo10}, there exists a homomorphism $J\to J^1$ over $X\times \A$ which is  an open embedding.
		Its cokernel is an \'etale sheaf of $2$-torsion abelian groups, supported on certain closed subschemes of $X\times \A$ (see \cite[(3.1,3)]{CZ17} for an explicit description). 
		If $p\neq 2$, then we have natural isomorphisms:
		\begin{equation}
			J[p^{\infty}]\xrightarrow{\sim} J^1[p^{\infty}], \quad
			\Lie J\xrightarrow{\sim} \Lie J^1.
			\label{eq:J-J1}
		\end{equation}
		From \Cref{p:log-of-smooth-group} and \Cref{l:log-surj-alg-case}, we deduce that we have a natural isomorphism $\whJ\xrightarrow{\sim} \widehat{J}^1$.
	\end{secnumber}

	\begin{theorem} \label{t:R1fvUJ-cameral}
		Assume $G$ is non-commutative and $p\nmid |\!\rW\!|$. 	
	Set $\Lambda:= (\rR^1 \widetilde{\pi}_{v\ast}\varprojlim_k T[p^k])^{\rW}$, where $\rW$  acts diagonally on $T[p^k]\times \widetilde{X}$. Then	
	there is a natural morphism of short exact sequences of v-sheaves over $\A$
			\[\begin{tikzcd}
				1 \arrow[r] & { \rR^1f_{v\ast}J[p^{\infty}]} \arrow[r] \arrow[d,"\sim"] & \rR^1f_{v\ast}\whJ \arrow[r,"\log"] \arrow[d]     & \rR^1f_{v\ast}\Lie J \arrow[d] \arrow[r]  & 0 \\
				1 \arrow[r] & \Lambda\otimes_{\uZp}\mu_{p^\infty} \arrow[r]         &
				\Lambda\otimes_{\uZp}\widehat{\G}_m \arrow[r,"\log"] & 
				\Lambda\otimes_{\uZp}\G_a \arrow[r] & 0
			\end{tikzcd}\]
	where $f:X\times \A \to \A$ is the projection. In particular, the square on the right is a pullback square.
\end{theorem}
Exactly as in \Cref{c:exp-splits-L_X}, \Cref{t:R1fvUJ-cameral} implies:
\begin{coro}\label{c:exp-induces-splitting-general-G}
	Let $S$ be a strictly totally disconnected space in $\A_v$. Then any exponential for $S$ induces a splitting of the torsor $\cH_{\bX}\times_\A S\to S$ over $S$. In particular, it induces a section of $\cH_\bX(S)\to \A(S)$.
\end{coro}
For the proof, we need a variant of \Cref{l:et-cohom-Z/p^k-spectral-curve} for the cameral cover that incorporates the $W$-invariants:
\begin{lemma} \label{l:Tp^k-cohomology}
	Let  $n,k,l \in \Z_{\geq 0}$.
		\begin{enumerate}
			\item We have
				$\rR^1f_{v\ast} ( (\pi_{v\ast}T[p^k])^{\rW})\simeq (\rR^1\widetilde{\pi}_{v\ast}T[p^k])^{\rW}$. 
			\item We have a short exact sequence
				\begin{equation} 
					1\to (\rR^1\widetilde{\pi}_{v\ast}T[p^l] )^{\rW}\xrightarrow{\cdot p^k} 
					(\rR^1\widetilde{\pi}_{v\ast}T[p^{k+l}] )^{\rW}
					\to  (\rR^1\widetilde{\pi}_{v\ast}T[p^k])^{\rW}\to 1.
				\end{equation}
			
			\item We have $\Lambda\simeq \varprojlim_k (\rR^1 \widetilde{\pi}_{v\ast}T[p^k])^{\rW}$ and $\Lambda/p^k \Lambda\simeq (\rR^1 \widetilde{\pi}_{v\ast}T[p^k])^{\rW}$. 

			\item The $\uZp$-module $\Lambda$ is $p$-torsionfree.
		\end{enumerate}
\end{lemma}

\begin{proof}
	(1) Since $\pi,\widetilde{\pi},f$ are algebraisable proper maps, we may by \cite[Proposition 5.8]{PCT-char-p} replace $\rR^1\widetilde{\pi}_{v\ast}$ with $\rR^1\widetilde{\pi}_{\et\ast}$ and work with algebraic \'etale cohomology. 
	Set $\mathcal{F}=\pi_{\et\ast}(T[p^k])$ over $X\times \A$, equipped with the diagonal $\rW$-action. 
	By the Grothendieck spectral sequence for the composition
	\[
	\delta:= (-)^{\rW}\circ f_{\et\ast}=f_{\et\ast}\circ (-)^{\rW}: \Ab((X\times \A)_{\et},\Z[\rW]) \to \Ab(\A_{\et}),
	\]
	we obtain natural maps:
	\[
	\rR^1f_{\et\ast}(\mathcal{F}^{\rW}) \to \rR^1 \delta (\mathcal{F}) \to (\rR^1 f_{\et\ast} \mathcal{F})^{\rW}.
	\]
	Since $p\nmid |\rW|$, higher Galois cohomology of $\rW$ with values in $p$-power torsion abelian groups vanishes. Considering the full Grothendieck spectral sequences, we deduce by working on geometric fibres that these maps are isomorphisms.
	Moreover, $\pi_{\et\ast}$ is exact since $\pi$ is finite. This shows assertion (1).

	(2) The morphism $\widetilde{\pi}:\widetilde{X}\to \A$ satisfies the assumptions of \Cref{l:et-cohom-Z/p^k-spectral-curve}.
	By \Cref{l:et-cohom-Z/p^k-spectral-curve}.(2) applied to $\widetilde{\pi}$, we have the short exact sequence before taking $(-)^{\rW}$. 
	After taking $\rW$-invariants, the assertion follows again from the fact that higher Galois cohomology of $\rW$ with values in $p$-power torsion abelian groups vanishes. 

	(3) By \Cref{l:et-cohom-Z/p^k-spectral-curve}.(3), we have $\rR^1 \widetilde{\pi}_{v\ast} \varprojlim_k T[p^k]\simeq  
	\varprojlim_k \rR^1\widetilde{\pi}_{v\ast} T[p^k]$, so the first part follows by taking $W$-invariants, which commutes with $\varprojlim$.
	The second part of (3) and (4) follow from (2) by taking $\varprojlim_l$.
\end{proof}

\begin{proof}[Proof of \Cref{t:R1fvUJ-cameral}]
	We first show that the left vertical arrow is an isomorphism. By \eqref{eq:def-J1} and \eqref{eq:J-J1}, we have $J[p^{k}]\simeq (\pi_{v\ast}T[p^{k}])^{\rW}$, hence $\rR^1f_{v\ast}J[p^k]\simeq (\rR^1 \widetilde{\pi}_{v\ast} T[p^k])^{\rW}$ by \Cref{l:Tp^k-cohomology}.(1). Consequently,
	\[\textstyle
	\Lambda\otimes_{\uZp} \mu_{p^{\infty}}
	\simeq \varinjlim_k \Lambda\otimes_{\uZp}\mu_{p^k}
	\simeq \varinjlim_k (\rR^1 \widetilde{\pi}_{v\ast} T[p^k])^{\rW} 
	\simeq (\rR^1 \widetilde{\pi}_{v,\ast} T[p^{\infty}])^{\rW},
	\]
	where the second isomorphism follows from \Cref{l:Tp^k-cohomology}.(3). 

	We now choose an isomorphism $T\simeq \Gm^{r}$ and apply \Cref{p:bottom-part-of-comp-fund-seq-to-log} to $\widetilde{\pi}:\widetilde{X}\to \A$ for each factor of $T$. 
	From this we obtain a $\rW$-equivariant isomorphism of short exact sequences:
	\[
		\begin{tikzcd}[column sep = 0.4cm]
			1 \arrow[r] & \rR^1\widetilde{\pi}_{v\ast} (\varprojlim_{k} T[p^k]) \otimes_{\uZp}\mu_{p^\infty} \arrow[r]   \arrow[d,"\sim"']      & \rR^1\widetilde{\pi}_{v\ast} (\varprojlim_{k} T[p^k]) \otimes_{\uZp}\widehat{\G}_m\arrow[r,"\log"] \arrow[d,"\sim"']     & \rR^1\widetilde{\pi}_{v\ast} (\varprojlim_{k} T[p^k]) \otimes_{\uZp}\G_a\arrow[r] \arrow[d,"\sim"'] & 0\\
			1 \arrow[r] & { \rR^1\widetilde{\pi}_{v\ast}T[p^{\infty}] } \arrow[r]  & 
			\rR^1\widetilde{\pi}_{v\ast}\wh T \arrow[r,"\log"] & \rR^1\widetilde{\pi}_{v\ast} \Lie T \arrow[r]& 0.
		\end{tikzcd}
	\]
	Taking $\rW$-invariants and using \Cref{l:Tp^k-cohomology}.(3), we arrive at the following commutative diagram:
	\begin{equation} \label{eq:compare-lambda-log-sequences}
		\begin{tikzcd}
			1 \arrow[r] & {\Lambda} \otimes_{\uZp}\mu_{p^\infty} \arrow[r]   \arrow[d,"\sim"']      & {\Lambda} \otimes_{\uZp}\widehat{\G}_m\arrow[r,"\log"] \arrow[d]     & {\Lambda} \otimes_{\uZp}\G_a\arrow[r] \arrow[d,"\sim"'] & 0\\
			1 \arrow[r] & { (\rR^1\widetilde{\pi}_{v\ast}T[p^{\infty}])^{\rW} } \arrow[r]  & (\rR^1\widetilde{\pi}_{v\ast}\wh T)^{\rW} \arrow[r,"\log"] & (\rR^1\widetilde{\pi}_{v\ast} \Lie T)^{\rW} \arrow[r]& 0.
		\end{tikzcd}\end{equation}
	In this diagram, the bottom line is left-exact while the top line is short exact by \Cref{l:Tp^k-cohomology}.(4). Hence the middle arrow is an isomorphism by the 5-lemma, and it follows that also the bottom line is short exact.
	
	Using again that $T\simeq\G_m^r$ and applying \Cref{d:psi-R^1B^x->R^1pi'whGm} to each of the factors, we obtain a canonical map:
	\[
	\psi_{T}: \rR^1f_{v\ast} \whJ^1 \to \rR^1f_{v\ast} \FHom_{X\times \A} (\uZp,\pi_{v\ast}T) \xrightarrow{\psi'_T} \rR^1 \widetilde{\pi}_{v\ast} \wh T
	\]
	which clearly factors through $(\rR^1 \widetilde{\pi}_{v\ast} \widehat{T})^W$. Recall that $\wh J=\wh J^1$. Since
	$J[p^{\infty}] \simeq J^1[p^{\infty}]\simeq (\pi_{v\ast}T[p^{\infty}])^{\rW}$ by \eqref{eq:J-J1}, \eqref{eq:def-J1}, we may define $\psi_{T[p^{\infty}]}$ and $\psi_{\Lie T}$ in the same way to obtain the desired commutative diagram:
	\[
		\begin{tikzcd}
			1 \arrow[r] & {\rR^1f_{v\ast}J[p^\infty]} \arrow[r] \arrow[d,"\psi_{T[p^{\infty}]}","\simeq"'] & \rR^1f_{v\ast}\whJ \arrow[r,"\log"] \arrow[d,"\psi_{T}"]     &\rR^1\pi_{v\ast}\Lie J \arrow[r]\arrow[d,"\psi_{\Lie T}"] & 0 \\
			1 \arrow[r] & {(\rR^1\widetilde{\pi}_{v\ast}T[p^{\infty}])^{\rW}} \arrow[r] & (\rR^1\widetilde{\pi}_{v\ast}\wh T)^{\rW} \arrow[r,"\log"]      & (\rR^1\widetilde{\pi}_{v\ast}\Lie T)^{\rW} \arrow[r] & 0.
		\end{tikzcd}
	\]
	The first row is exact by \Cref{p:log-seq-R^1fv}. The second row is exact by \eqref{eq:compare-lambda-log-sequences}.
	
	Combining this diagram with the isomorphism of short exact sequences \eqref{eq:compare-lambda-log-sequences} proves the Theorem.
\end{proof}
	\subsection{The homeomorphism between topological spaces}
	As the main application of the results in this section, we can now deduce a topological comparison of moduli spaces, like in complex geometry.
	
	For every v-sheaf $F$ on $\Spa(K)$, one can endow the set of $K$-points $F(K)$ with a natural topology, using the condensed formalism of Clausen--Scholze \cite{Condensed}. Let us make this explicit in our specific context:
	\begin{definition}
	 For any profinite space $S$, there is an associated v-sheaf $\underline{S}$ on $\Spa(K)_v$, which is represented by the strictly totally disconnected space $\Spa(\mathrm{Map}_{\cts}(S,K))$. We thus obtain a condensed set
	\[\uHig_G: S\mapsto \bfHig_{G}(\underline{S}).\]
	More precisely, to avoid set-theoretic issues, we fix a cut-off cardinal $\kappa$ and only test by $\kappa$-small profinite sets $S$. This endows $\bfHig_{G}(K)=\uHig_G(K)$ with a natural compactly generated topology: Explicitly, this is the finest topology such that for every $\underline{S}\to \bfHig_{G}$, the associated map of $K$-points $S\to \bfHig_{G}(K)$ is continuous.

	 Similarly, we endow $\bfBun_{G,v}(K)$ with a natural topology by considering the condensed set 
	\[\underline{\mathrm{Bun}}_{G,v}: S\mapsto \bfBun_{G,v}(\underline{S}).\]
	\end{definition}
	\begin{theorem}\label{t:homeom-moduli}
		Assume either that $G=\GL_n$ or that $G$ is a reductive group with $p\nmid |\rW|$. Then
		choices of a flat lift $\mathbb{X}$ of $X$ to $\BdR^+/\xi^2$ and of an exponential $\Exp$ for $K$ induce an isomorphism of condensed sets
		\[\uHig_G\isomarrow  \underline{\mathrm{Bun}}_{G,v}.\]
		On $K$-points,
		this induces a homeomorphism 
		\[ \mathbf S:\bfHig_{G}(K) \xrightarrow{\sim} \bfBun_{G}(K).\]
	\end{theorem}
	This is a very close analogue of Simpson's homeomorphism between coarse moduli spaces in complex non-abelian Hodge theory \cite[Theorem 7.18]{SimpsonModuliII}.
	\begin{proof}
		Since $\underline{S}$ is strictly totally disconnected, any exponential $\Exp$ for $K$ induces an exponential for $S$
		\[ \Exp_S:\mathrm{Map}_{\cts}(S,K)\to \mathrm{Map}_{\cts}(S,1+\mathfrak m_K),\]
		functorial in $S$. 
		Here we use that $\mathfrak m_K=K^{\circ\circ}$, so $\mathrm{Map}_{\cts}(S,1+\mathfrak m_K)= \mathrm{Map}_{\cts}(S,K)^{\circ\circ}$.
		 By \Cref{c:exp-induces-splitting}, respectively \Cref{c:exp-induces-splitting-general-G}, we thus obtain natural splittings
		\[ s_{\Exp,S}:\A(S)\to \cH_{\bX}(S)\]
		for each $S$, or in other words, a morphism $\underline{\mathrm A}\to \underline{\mathrm H}_{\bX}$ between the condensed sets  associated to $\A$ and $\cH_{\bX}$. We deduce from this a morphism of condensed sets
		\[\uHig_{G}\xrightarrow{s_{\Exp}} \underline{\mathrm H}_{\bX}\times_{\underline{\mathrm A}}\uHig_{G}\isomarrow \underline{\mathrm H}_{\bX}\times_{\underline{\mathrm A}}\underline{\mathrm{Bun}}_{G}\to \underline{\mathrm{Bun}}_{G}.\]
		As we can similarly construct a morphism into the other direction, this is an isomorphism of condensed sets. By \cite[Remark 1.6]{Condensed}, this gives the desired homeomorphism between $K$-points. 
	\end{proof}
	\begin{rem}
	The bijection $\mathbf S$ from \Cref{t:homeom-moduli} can be shown to extend to an equivalence of categories, at least after making further auxiliary choices.
	In the case of $G=\GL_n$, this works as follows:
	Unravelling the definitions, we see from \Cref{l:CS-for-GLn} that $\mathbf S$ is compatible with the categorical $p$-adic Simpson correspondence $S$ from \cite[Theorem~1.1]{Heu23}. Indeed, it shows that for any choice of $\Exp$, the following diagram commutes, where the vertical maps are the passage to groupoids of isomorphism classes:
		\[
	\begin{tikzcd}[row sep =0.3cm]\{\text{Higgs bundles on $X$ of rank $n$}\} \arrow[r,"S"]\arrow[d] & \{\text{v-vector bundles on $X$ of rank $n$}\} \arrow[d] \\
	\bfHig_{\GL_n}(K)\arrow[r,"\mathbf S"]       & 	\bfBun_{\GL_n,v}(K).   
	\end{tikzcd}\]
	Namely, the further choices for $\mathbf S$ are necessary to define notions of rigidifications of $\B^\times$-torsors. Combining ideas from \cite[\S3]{Heu23} with the constructions of this article,
	it ought to be possible (but probably quite technical) to define a notion of rigidifications also for strictly totally disconnected families. By the Tannakian formalism \cite[\S3]{HWZ}, this would lead to an extension of \Cref{t:homeom-moduli} to all linear algebraic $G$.
	\end{rem}
	
\section{The v-stack of Higgs bundles} \label{s:v-stack-Higgs}
	The goal of this subsection is to show that the v-stack of Higgs bundles $\CHig_{G}$ is essentially a classical object. For this we need to compare the notion of analytic vector bundles on smoothoid spaces to algebraic vector bundles on smooth schemes. The key is therefore to prove a perfectoid GAGA Theorem.

	\subsection{Perfectoid GAGA for curves}
	Throughout this section, let $X^\alg$ be any smooth proper variety over $K$. We will later restrict attention to curves, but in the beginning, we can work in this greater generality.
	Let $X$ be the adic analytification of $X^\alg$. We
	recall that there is a natural morphism of locally ringed spaces
	$X\to X^{\alg}$
	over $\Spa(K,K^+)\to \Spec(K)$.
	Let now $(R,R^+)$ be any perfectoid $K$-algebra.
	To simplify notation, we write 
	\[ X_R:=X\times \Spa(R),\quad  X^{\alg}_R:=X^{\alg}\times \Spec(R).\]
	Consider $\Spa(R)$ as a locally ringed space. By the universal property of $\Spec(R)$, the identity $R\to R$ then corresponds to a morphism of locally ringed spaces
	$\Spa(R,R^+)\to \Spec(R)$.
	As the fibre product in schemes is the same as that in locally ringed spaces, we obtain from this a natural morphism of locally ringed spaces
	\[h:X_R\to X_R^{\alg}.\]
	The goal of this subsection is to prove the following GAGA Theorem for vector bundles:
	\begin{theorem}[Perfectoid GAGA for curves]\label{t:perf-GAGA}
		Let $X^{\alg}$ be any smooth proper variety over $K$, and let $X$ be the associated adic space. Let $(R,R^+)$ be any perfectoid $K$-algebra.
		\begin{enumerate}
			\item Pullback along $h$ defines a fully faithful exact tensor functor
			\[-^{\an}:\{ \text{vector bundles on $X^{\alg}\times \Spec(R)$}\}\to \{\text{analytic vector bundles on $X\times \Spa(R)$}\}.\]
			\item	For any vector bundle $F$ on $X^{\alg}_R$, we have 
			$\RG(X^{\alg}_R,F)=\RG(X_R,F^{\an})$.
			\item Assume moreover that $X^\alg$ is a curve. Then any analytic vector bundle on $X\times \Spa(R)$ lies in the essential image of $-^\an$ after replacing $\Spa(R)$ by an open cover.
		\end{enumerate}
	\end{theorem}
	\begin{rem}
		Part (3) should hold for any smooth proper variety $X$, and without passage to an open cover. In fact, Scholze has informed us that this should follow from his joint work with Clausen, similarly as in \cite[\S13]{CondensedComplex}. That it holds for line bundles in this generality is shown in \cite{diamantinePicard}. 
	\end{rem}
	
	The proof of \Cref{t:perf-GAGA} will take the whole section. We first show (1) and (2). The proof of (3) is more difficult, and we will prove it by reducing to the case of $\P^1$, where it can be seen by explicit computations.
	\begin{proof}[Proof of \Cref{t:perf-GAGA}.(1) and (2)]
		We first note that for any vector bundle $F$ on $X^{\alg}_R$, the pullback
		\[ F^\an:=h^{\ast}F=h^{-1}F\otimes_{h^{-1}\O_{X^{\alg}_R}}\O_X\]
		is again a vector bundle. More generally, while there is in general no good theory of coherent sheaves on adic spaces, it still makes sense to consider the pullback $F^\an:=h^{\ast}F$ of any quasi-coherent sheaf $F$ on $X^{\alg}_R$. 
		
		\begin{prop}\label{p:GAGA-comp-cohom}
			For any vector bundle $F$ on $X^{\alg}_R$, we have
			$\RG(X^{\alg}_R,F)=\RG(X_R,F^{\an})$.
		\end{prop}
		\begin{proof}
			It suffices to see that the natural map
			$\RG(X_R^{\alg},F)\to\RG (X_R,F^\an)$ induces an isomorphism on $\rH^n$ for every $n\in \N$. 
			Choose a rigid approximation $R=\varinjlim_i R_i$ by $R_i$ that are topologically of finite type over $K$ as in \cite[Proposition~3.17]{diamantinePicard}. For any locally ringed space $Y$, let  $\mathrm{VB}(Y)$ be the set of isomorphism classes of finite locally free modules on $Y$. Then by \cite[02JO]{StacksProject}, we have
			\[ \textstyle\mathrm{VB}(X_R^{\alg})=2\text{-}\mathrm{colim}_{i} \mathrm{VB}(X_{R_i}^{\alg}).\]
			We may therefore assume that $F$ is the pullback of a vector bundle $F_i$ on $X_{R_i}^{\alg}$. For any $j\geq i$,  let $F_j$ be the base-change to $X_{R_j}^{\alg}$. We now apply \cite[Proposition~3.2]{perfectoid-base-change} to the diagram
			\[
			\begin{tikzcd}
				X_R \arrow[d] \arrow[r] & \Spa(R) \arrow[d] \\
				X_{R_i} \arrow[r]       & \Spa(R_i).    
			\end{tikzcd}\]
			This says that there is a perfect complex $K^\bullet$ on $\Spa(R_i)$ such that
			\[\textstyle \rH^n(X_R,F^\an)=\rH^n(K^\bullet\otimes_{R_i}R)=\varinjlim_j \rH^n(K^\bullet\otimes_{R_i}R_j)=\varinjlim_j  \rH^n(X_{R_j},F_j^\an),\]
			where in the last step, we have applied  \cite[Proposition~3.2]{perfectoid-base-change} in the rigid setup to $\Spa(R_j)\to \Spa(R_i)$. We can now apply K\"opf's relative rigid GAGA \cite{Koepf_GAGA} to see that
			\[\textstyle\varinjlim_j  \rH^n(X_{R_j},F_j^\an)=\varinjlim_j \rH^n(X_{R_j}^{\alg},F_j)=\rH^n(X_{R}^{\alg},F).\qedhere\]
		\end{proof}
		\begin{coro}\label{c:GAGA-fully-faithful}
			The functor $-^{\an}$ from \Cref{t:perf-GAGA} is fully faithful.
		\end{coro}
		\begin{proof}
			Let $E_1,E_2$ be vector bundles on $X^\alg_R$. 
			Working locally, we see that
			$ \FHom(E_1,E_2)^{\an}=\FHom(E_1^\an,E_2^\an)$.
			We now apply \Cref{p:GAGA-comp-cohom} to $F:=\FHom(E_1,E_2)$ and take $\rH^0$ to deduce:
			\[ \Hom(E_1,E_2)=\rH^0(X_{R}^{\alg},F)=\rH^0(X_{R},F^\an)=\Hom(E_1^\an,E_2^\an).\qedhere\]
		\end{proof}
		This completes the proof of \Cref{t:perf-GAGA}.(1) and (2).
	\end{proof}
	For the proof of \Cref{t:perf-GAGA}.(3), we start with a lemma for which we do not yet need that $X$ is a curve.
	\begin{lemma} \label{l:vb-analytic-algebraic}
		In the setting of \Cref{t:perf-GAGA},
		let $E$ be a quasi-coherent sheaf of finite type on $X^{\alg}_R$ and assume that $E^{\an}$ is a vector bundle of rank $r$. Then $E$ is also a vector bundle of rank $r$.
	\end{lemma}
	\begin{proof}
		We first consider the case that $(R,R^+)=(C,C^+)$ is an algebraically closed affinoid field. In this case, $a:X_C\to X^{\alg}_C$ is the classical GAGA map which identifies $C$-points of $X_C$ with closed points of $X^{\alg}_C$. Moreover, given any such point $x\in X_C(C)$, the natural map between completed stalks
		$\O^\wedge_{X^{\alg}_C,a(x)}\to \O^\wedge_{X_C,x}$
		is an isomorphism. Since a finitely generated module $E$ on a Noetherian scheme is locally free of rank  $r$ if and only if all its completed stalks are free of rank $r$, this shows the result in this case.
		
		In order to deduce the general case, assume that $E$ on $X^{\alg}_R$ is not locally free of rank $r$. It suffices to prove that there is a morphism $\psi:(R,R^+)\to (C,C^+)$ into an algebraically closed affinoid field such that for the base change $b:X^{\alg}_C\to X^{\alg}_R$, the sheaf $b^{\ast}E$ is still not locally free of rank $r$. As  $b^{\an\ast} E^{\an}=(b^{\ast}E)^\an$ is locally free of rank $r$, this will contradict the case of $(R,R^+)=(C,C^+)$ already discussed.

		To find $\psi$, we consider the Fitting ideals $\mathrm{Fit}_i(E)\subseteq \O_{X^{\alg}_R}$: By \cite[0C3G]{StacksProject}, $E$ was locally free of rank $r$ if
		\[\mathrm{Fit}_{r-1}E=0\quad \text{and} \quad \mathrm{Fit}_{r}E=\O_{X^{\alg}_R}.\]
		Assume first that there is an affinoid $U=\Spec(A)\subseteq X^{\alg}$ such that there is $0\neq f\in \mathrm{Fit}_{r-1}E|_{U\times \Spec(R)}$. Then $f$ is a section of $\O(U\times \Spec(R))=A\otimes_KR$. Recall that any $s\in R$ vanishes if and only if $s(x)=0$ for all $x\in \Spa(R)$. Equivalently, this means that for any morphisms $\psi:(R,R^+)\to (C,C^+)$ into algebraically closed affinoid fields, the image of $s$ is $=0$. Since $f\neq 0$, and $A$ is a $K$-vector space, we deduce that there is $\psi:(R,R^+)\to (C,C^+)$ such that $\psi(f)\neq 0$. Consider the base-change of $E$ along $b:X^{\alg}_C\to X^{\alg}_R$, then by \cite[0C3D]{StacksProject}, we have
		\[ \mathrm{Fit}_{r-1}(b^{\ast}E)= b^{-1}(\mathrm{Fit}_{r-1}E)\cdot \O_{X^{\alg}_C}\neq 0\]
		since $\psi(f)\neq 0$. Hence $b^{\ast}E$ is not locally free of rank $r$, which we have already seen is impossible.
		
		It remains to see that $\mathrm{Fit}_{r}E=\O_{X^{\alg}_R}$. Suppose not, then we can find a closed point $x$ of $X^{\alg}_R$ such that $\mathrm{Fit}_{r}(E_x)\subsetneq \O_{X^{\alg}_R,x}$. Locally on some affine open $x\in U=\Spec(B)\subseteq X^{\alg}_R$, this corresponds to a maximal ideal $\mathfrak m_x$ such that $\mathrm{Fit}_{r}(E_{|U})\subseteq \mathfrak m_x \subsetneq B$. Since $X$ is proper, the projection $\pi:X^{\alg}_R\to \Spec(R)$ is universally closed, hence $\pi(x)$ is itself a closed point, corresponding to a maximal ideal $\mathfrak m\subseteq R$. Then $\mathfrak m \cdot B\subseteq \mathfrak m_x,$ hence 
		\[ \mathrm{Fit}_{r}(E_{|U\times_{\Spec(R)} \Spec(R/\mathfrak m)})\subseteq  \mathfrak m_x B/\mathfrak m B \subsetneq B/\mathfrak m B.\]
		This means that the base-change of $E$ to the fibre of $X^{\alg}_R$  over $\pi(x)$ is not locally free of rank $r$.
		
		We now use that by \cite[Lemma~1.4]{Huber-generalization}, there exists $x\in \Spa(R)$ with $\mathrm{supp}(x)=\mathfrak m$. In particular, by \cite[Propositions~2.25, 2.27]{perfectoid-spaces}, there is a morphism $\psi:(R,R^+)\to (C,C^+)$ into an affinoid field (wlog algebraically closed) such that $\ker \psi=\mathfrak m$. Let $b:X^{\alg}_C\to X^{\alg}_R$ be the base-change, then since $R/\mathfrak m\to C$ is a field extension,  the module $b^{\ast}E$ is still not locally finite free of rank $r$ over $X^{\alg}_C$.
	\end{proof}
	
	\subsection{Perfectoid GAGA for $\mathbb P^1$}
	As the first case of \Cref{t:perf-GAGA}, we now prove that it holds for $X=\mathbb P^1$:
	\begin{prop}\label{p:rigid-approx-vb-P1}
		For any affinoid perfectoid $K$-algebra $(R,R^+)$, the analytification functor
		\begin{equation} \label{eq:an-P1}
			-{^\an}: \{ \text{vector bundles on $\mathbb P^1\times \Spec(R)$}\}\to \{ \text{vector bundles on $\mathbb P^1\times \Spa(R)$}\}
		\end{equation}
		is fully faithful, and any object lies in the essential image after replacing $S=\Spa(R)$ by an open cover.
	\end{prop}
	\begin{proof}
		For the proof, we take inspiration from  Grothendieck's description of vector bundles on $\mathbb P^1$, and then from Serre's proof of the GAGA Theorems:
		We consider the cover of $\mathbb P^1$ by two open unit discs $\mathbb B^1$ around $0$ and $\infty$ with intersection a unit tyre $U=\Spa(K\langle T^{\pm 1}\rangle)$. Let us denote these by $\mathbb B^1(0)$ and $\mathbb B^1(\infty)$.
		We set $S=\Spa(R,R^+)$ and write $\mathfrak V$ for the cover of $\mathbb P^1_S$ by the base-changes $\mathbb B^1_S(0)$ and $\mathbb B^1_S(\infty)$ to $S$.  
		\begin{lemma}\label{l:Lutkebohmert-Quillen-Suslin}
			Let $d\in \N$.
			Any vector bundle on $\mathbb B^d_S$ becomes trivial after replacing $S$ by an open cover.
		\end{lemma}
		\begin{proof}
			We can write $R\langle X_1,\dots,X_d\rangle=(\varinjlim R_i^+\langle X_1,\dots,X_d\rangle)^\wedge[\tfrac{1}{p}]$ where $S_i=\Spa(R_i)$ is a rigid approximation of $S$ as in the proof of \Cref{p:GAGA-comp-cohom}. By \cite[Corollary 5.4.41]{GabberRamero}, the pullback on isomorphism classes
			\[\textstyle\varinjlim_i\mathrm{VB}(S_i\times \mathbb B^d)\to \mathrm{VB}(S\times \mathbb B^d)\]
			is a bijection. It therefore suffices to prove the statement when $S$ is an affinoid rigid space. Here it holds by a result of L\"utkebohmert, based on the work of Quillen--Suslin \cite[Satz~1]{Lutkebohmert_Vektorraumbundel}.
		\end{proof}
		Let $E$ be a vector bundle on $\mathbb P^1_S$ of rank $r$. By \Cref{l:Lutkebohmert-Quillen-Suslin}, we may assume after replacing $S$ by an open cover that $E$ is free over $\mathbb B^1_S(0)$ and $\mathbb B^1_S(\infty)$.  Hence we may assume without loss of generality that $E$ is glued from two trivial vector bundles along $U\times S$. Such vector bundles are described by the double coset
		\begin{equation}\label{eq:double-coset-GL_r}
			\check{\rH}^1(\mathfrak V,\GL_r)=\GL_r(R\langle T\rangle)\backslash\GL_r(R\langle T^{\pm 1}\rangle)/\GL_r(R\langle T^{-1}\rangle).
		\end{equation}
		
		\begin{lemma}\label{l:H1(P1-1+mMnO+)}
				We have 
				$\rH^1_{\et}(\mathbb P^1_S,1+\mathfrak m_KM_n(\O^+))= 1$.
		\end{lemma}
		\begin{proof}
			We first note that, as $\mathbb P^1$ has no non-split connected finite \'etale covers, \cite[Proposition~3.9]{Heu22b} says
			\begin{equation}\label{l:H1(P1-1+mMnO+)-part-1}
			 \rH^1_{\et}(\mathbb P^1_S,\O^+/p)\aeq \rH^1_{\et}(\Spa(R),\O^+/p)\aeq 0.
			 \end{equation}
			To simplify notation, let us write $\rH^i(-)$ for $\rH^i_\et(\mathbb P^1_S,-)$.
			Consider the long exact sequence
			\[ \rH^0(\GL_n(\O^+))\to \rH^0(\GL_n(\O^+/\mathfrak m_K))\to \rH^1(1+\mathfrak m_KM_n(\O^+))\to \rH^1(\GL_n(\O^+))\to \rH^1(\GL_n(\O^+/\mathfrak m_K))\] of pointed sets. By considering $\GL_n(\O^+)\subseteq M_n(\O^+)$ as a subsheaf, we deduce from \eqref{l:H1(P1-1+mMnO+)-part-1} that the first map is surjective. It thus
			suffices to see that the last map has trivial kernel. For this we argue as in \cite[
			Lemma 2.24]{heuer-G-torsors-perfectoid-spaces}:
			Let $V$ be any finite locally free $\O^+$-module $V$ for which $V/\mathfrak{m}_K$ is trivial. Then since $V/\mathfrak{m}_K=\varinjlim_{\epsilon\to 0} V/p^{\epsilon}$, already $V/p^{\epsilon}$ is trivial for some $\epsilon>0$. We consider the long exact sequence of
			\[ 0\to  \mathfrak m_K V/p^\epsilon\to   V/ p^{n\epsilon}\mathfrak m_K\to V/p^{(n-1)\epsilon}\mathfrak m_K\to 0\]
			for $n\in \N$.
			Since $\rH^1_{\et}(\mathfrak m_K V/p^\epsilon)=0$ by \eqref{l:H1(P1-1+mMnO+)-part-1}, we see inductively that we may find compatible lifts of generators of $V/p^\epsilon$ to $V/p^{n\epsilon}$ for every $n$. In the limit, this shows that $V$ is trivial. 
		\end{proof}
		
		\begin{lemma}\label{l:double-coset-approx}
			Any element of \eqref{eq:double-coset-GL_r} is represented by some $A\in \GL_r(R\langle T^{\pm 1}\rangle)$ such that $A\in T^{-m}M_r(R\langle T\rangle)$.
		\end{lemma}
		\begin{proof}
			To simplify notation, for any Huber pair $(S,S^+)$ over $K$, let us write $U(S):=1+\mathfrak m_KM_r(S^+)$. Then
			\[ \GL_r(R\langle T^{\pm 1}\rangle)=\GL_r(R[T^{\pm 1}])U(R\langle T^{\pm 1}\rangle)\]
			because the first factor is dense in $\GL_r(R\langle T^{\pm 1}\rangle)$, while the second factor is an open subgroup. Write $A=A_0A_1$ such that $A_0\in \GL_r(R[T^{\pm 1}])$ and $A_1\in U(R\langle T^{\pm 1}\rangle)$. We consider the double-coset
			\begin{equation}\label{eq:double-coset-U}
				\check{\mathrm H}^1(\mathfrak V,U)=U(R\langle T\rangle)\backslash U(R\langle T^{\pm 1}\rangle)/U(R\langle T^{-1}\rangle).
			\end{equation}
			This can be interpreted as \v{C}ech cohomology, which injects into $\rH^1(\P^1_S,1+\mathfrak m_K M_r(\O^+))$. But this is trivial by \Cref{l:H1(P1-1+mMnO+)}, hence \eqref{eq:double-coset-U} is trivial. It follows that we can write $A_1=A_1^+A_1^-$ for some $A_1^+\in U(R\langle T\rangle)$ and $A_1^-\in U(R\langle T^{-1}\rangle)$. We then see that in \eqref{eq:double-coset-GL_r}, $A=A_0A_1^+A_1^-$ represents the same class as
			$A':=A_0A_1^+$
			because  $A_1^-\in \GL_r(R\langle T^{-1}\rangle)$. Then $A'$ has the desired property.
		\end{proof}
		
		\begin{lemma}\label{l:local-gen-VB-on-P1S}
			Let $E$ be a vector bundle on $\mathbb P^1_S$ of rank $r$. Then after replacing $S$ by an open cover, we can for any $m\in \N$ large enough find $s_1,\dots,s_r\in \rH^0(\mathbb P^1_S,E(m))$ such that the map
			$s_1,\dots,s_r:\O^r\to E(m)$
			is an isomorphism over $\mathbb B^1_S(\infty)\subseteq \mathbb P^1_S$.
		\end{lemma}
		\begin{proof}
			By \Cref{l:Lutkebohmert-Quillen-Suslin} we can assume that $E$ is defined by an element $A\in\GL_r(R\langle T^{\pm 1}\rangle)$. By \Cref{l:double-coset-approx}, we can find $m$ such that $T^mA\in M_r(R\langle T\rangle)$. It is clear that $E(m)$ corresponds to the element $T^mA$. Thus after replacing $E$ by $E(m)$, we may assume that $A\in M_r(R\langle T\rangle)$.
			
			Considering the \v{C}ech complex of the cover $\mathfrak V$ by $\mathbb B^1_S(0)$ and $\mathbb B^1_S(\infty)$, we see that 
			\begin{equation} \label{eq:Cech-H0}
				\rH^0(\mathbb P^1_S,E)=\ker\big(R\langle T\rangle^r\oplus R\langle T^{-1}\rangle^r\xrightarrow{f,g\mapsto f-A\cdot g} R\langle T^{\pm 1}\rangle^r\big).\end{equation}
			Since $A\in M_r(R\langle T\rangle)$, the standard basis vectors $e_i=(0,\dots,1,\dots,0)$ of $R\langle T^{-1}\rangle^r$ satisfy $Ae_i\in R\langle T\rangle$, thus $(Ae_i,e_i)$ is contained in the right hand side for $i=1,\dots,r$. This defines the desired sections $s_1,\dots,s_r$. As the $e_i$ generate $R\langle T^{- 1}\rangle^r$, and $E|_{\mathbb B_S(\infty)}$ gets identified with $R\langle T^{- 1}\rangle^r$, the $s_1,\dots,s_r$ generate $E$ over $\mathbb B_S(\infty)$.
		\end{proof}
		\begin{prop} \label{l:rightexactseq-of-O}
			Let $E$ be a vector bundle on $\mathbb P^1_S$ of rank $r$. Then after replacing $S$ by an open cover, we can find $n,m,r,s\in \N$ for which there is a right-exact sequence
			\begin{equation} \label{eq:presentation-O} \O(-n)^{t}\to\O(-m)^{s}\to E\to 0.\end{equation}
		\end{prop}
		\begin{proof}
			We first apply \Cref{l:local-gen-VB-on-P1S}. We then apply \Cref{l:local-gen-VB-on-P1S} once again also to $q^{\ast}E$, where $q=\left(\begin{smallmatrix}
				0&1\\1&0
			\end{smallmatrix}\right)$ is the automorphism exchanging  $\mathbb B_S(\infty)$ and $\mathbb B_S(0)$, to find $s_{r+1},\dots,s_{2r}$ such that
			$s_{r+1},\dots,s_{2r}:\O^r\to E(m)$
			is an isomorphism over $\mathbb B^1_S(0)\subseteq \mathbb P^1_S$. Forming the sum of both maps and twisting by $(-m)$, it follows that
			\[ s_{1},\dots,s_{2r}:\O^{2r}(-m)\to  E\]
			is surjective. Moreover, by construction, it becomes split surjective on each of $\mathbb B^1_S(0)$ and $\mathbb B^1_S(\infty)$, where $E$ is free. Hence the kernel $F$ is again a vector bundle. We then repeat the above construction for $F$.
		\end{proof}

		Now we are ready to prove \Cref{p:rigid-approx-vb-P1}: Let $E$ be any vector bundle on $\P^1_S$. After replacing $S$ by an open cover, we can find a right exact sequence as in \Cref{l:rightexactseq-of-O}. Note that the morphism $\phi:\O(-n)^{t}\to\O(-m)^{s}$ in \eqref{eq:presentation-O} is a morphism between algebraic vector bundles. By \Cref{c:GAGA-fully-faithful}, it is therefore analytification of a map $\phi^{\alg}$. Let $E_0$ be its cokernel.
		
		The functor $-^{\an}$ of \eqref{eq:an-P1} is right-exact, being the pullback of a morphism of locally ringed spaces. Hence \[E_0^{\an}=\mathrm{coker}(\phi^{\alg})^\an=\mathrm{coker}(\phi)=E.\]
		Finally, by  \Cref{l:vb-analytic-algebraic}, it follows that $E_0$ is finite locally free. Hence $E$ is in the essential image of $-^\an$.
	\end{proof}
	
	\subsection{The case of curves}
	Finally, we deduce the GAGA result for curves from GAGA for $\P^1$:
	\begin{proof}[Proof of \Cref{t:perf-GAGA}.(3)]
		The choice of any non-constant rational function $f\in K(X)$ defines a finite flat map $f:X\to \mathbb P^1$. Let $S=\Spa(R)$, then $f_S:X\times S\to \mathbb P^1\times S$ is still a finite locally free map of smoothoid adic spaces, i.e.,\ $A:=f_{S\ast}\O$ is a finite locally free module on $\mathbb P^1\times S$. It follows that any  vector bundle $E$ on $X\times S$ defines a finite projective $A$-module $M=f_{S\ast}E$ on $\mathbb P^1\times S$. Consequently, $M$ is a vector bundle on $\mathbb P^1\times S$. Hence, after replacing $S$ by an open cover, we can by \Cref{p:rigid-approx-vb-P1} find an algebraic vector bundle $M^\alg$ on $\mathbb P^1\times \Spec(R)$ such that $(M^{\alg})^{\an}=M$. We then clearly have $\FEnd(M^\alg)^\an=\FEnd(M)$.
		
		Let now $f^\alg_S:X^\alg\times \Spec(R)\to \mathbb P^1\times \Spec(R)$ and $A^\alg:=f^\alg_{S\ast}\O$, then $(A^{\alg})^{\an}=A$. 
		By \Cref{c:GAGA-fully-faithful}, the map $A\to \FEnd(M)$ encoding the $A$-module structure now comes from a morphism $A^\alg\to \FEnd(M^\alg)$. Using that $-^\an$ is faithful, we see that this is a ring homomorphism, endowing $M^\alg$ with an $A$-module structure. Since $f^\alg_S$ is affine, this defines a finitely generated quasi-coherent sheaf $E^\alg$ on $X^\alg\times \Spec(R)$. By construction, we have $(E^{\alg})^{\an}=E$. By \Cref{l:vb-analytic-algebraic}, we deduce that $E^\alg$ is a vector bundle.
	\end{proof}
	This completes the proof of \Cref{t:perf-GAGA}. \qed
	
	\subsection{Comparison to the algebraic stack of Higgs bundles} We deduce from \Cref{t:perf-GAGA}:
	\begin{coro}\label{c:GAGA-std}
		Let $\Spa(R)$ be a totally disconnected perfectoid space. Let $G$ be any linear algebraic group over $K$. Then the following natural GAGA functors 
		are equivalences of categories:
		\begin{align}
			 -^{\an}:\{ \text{$G$-torsors on $(X^{\alg}\times \Spec(R))_\et$}\}&\isomarrow \{\text{$G$-torsors on $(X\times \Spa(R))_\et$}\}\label{c:GAGA-std-1}\\
			 \mbox{}\vspace{-\baselineskip}
			-^{\an}:\{ \text{$G$-Higgs bundles on $(X^{\alg}\times \Spec(R))_\et$}\}&\isomarrow \{\text{$G$-Higgs bundles on $(X\times \Spa(R))_\et$}\}\label{c:GAGA-std-2}
		\end{align}
	\end{coro}
	\begin{proof}
		For \eqref{c:GAGA-std-1}, the case of $\GL_n$ holds by \Cref{t:perf-GAGA} as any open cover of $\Spa(R)$ splits. 
		The  general case follows by the Tannakian formalism: This says that the left hand side is equivalent to the category of exact tensor functors
		\[ \mathrm{Rep}(G)\to \Bun(X^{\alg}\times \Spec(R)).\]
		Due to the case of $\GL_n$, we have $\Bun(X^{\alg}\times \Spec(R))=\Bun_\an(X\times \Spa(R))$. But by \cite[Theorem 19.5.2]{ScholzeBerkeleyLectureNotes}, exact tensor functors
		$ \mathrm{Rep}(G)\to \Bun_\an(X\times \Spa(R))$ are equivalent to $G$-torsors on $X\times \Spa(R)$.
		
		For \eqref{c:GAGA-std-2}, recall that a Higgs field on a $G$-torsor $E$ (algebraic or analytic) is a section of $\ad(E)\otimes \wtOm$. If now $E$ is an algebraic $G$-torsor on  $X^{\alg}\times \Spec(R)$, then $\ad(E)^\an=\ad(E^\an)$. Thus by \Cref{t:perf-GAGA}.(3),
		\[ \rH^0(X^{\alg}\times \Spec(R),\ad(E)\otimes \wtOm)=\rH^0(X\times \Spa(R),\ad(E^\an)\otimes \wtOm).\qedhere\]
	\end{proof}

	\begin{definition}
		Let $\mathcal S$ be any fppf-stack fibered over the category of schemes over $K$. Following \cite{diamantinePicard}, there is a natural way to associate to this a v-stack $\mathcal S^\dia$:
		This is defined as the v-stackification of the functor
		\[h:\Perf^{\aff}_{K,v}\to \mathrm{Grp},\quad \Spa(R)\mapsto \mathcal S(R).\]
		This clearly defines a natural \textit{diamondification} functor
		\begin{equation}\label{eq:diamondification-stacks}
			-^{\dia}:\{\text{algebraic stacks over $K$}\}\to \{\text{small v-stacks over $K$}\}.
		\end{equation}
	\end{definition}
	In terms of this language, the goal of this subsection is to prove the following:
	\begin{theorem}\label{t:comparison-Higgs-stacks}
		Let $X$ be a smooth projective curve over $K$. Let $G$ be any linear algebraic group over $K$.
		\begin{enumerate}
			\item 
			Let  $\CBun^{\alg}_{G}$ be the algebraic stack of (\'etale) $G$-bundles on $X$. Then the following natural map is an isomorphism of v-stacks:
			\[ (\CBun^\alg_{G})^\dia\isomarrow \CBun_{G,\et}.\]
			\item Let  $\CHig^{\alg}_{G}$ be the algebraic stack of $G$-Higgs bundles on $X$.
			The following natural map is an isomorphism of v-stacks:
			\[ (\CHig^{\alg}_{G})^\dia\isomarrow \CHig_{G,\et}.\]
		\end{enumerate}
	\end{theorem}

	\begin{proof}
		Recall that $\CBun_{G,\et}$ is a v-stack by \cite[Theorem 1.4]{Heu}. There is a natural analytification functor $ (\CBun^\alg_{G})^\dia\to \CBun_{G,\et}$. We need to see that this is an equivalence.
		
		Since the definition of $(\CBun^\alg_{G})^\dia$ involves v-stackification, we may restrict the test category to totally disconnected spaces $S=\Spa(R)$, which form a basis of $\Perf^{\aff}_{K,v}$. In this case, the natural functor 
		\[ \CBun_{G}^\alg(R)\isomarrow \CBun_{G,\et}(R)\]
		is an equivalence by \eqref{c:GAGA-std-1}. This proves part 1.
		Part 2 follows in exactly the same way from \eqref{c:GAGA-std-2}.
	\end{proof}

	\subsection{The Hitchin fibration over the regular locus}
	As an application, we get a version of \Cref{t:fCiso} for rigid analytic moduli spaces: In this subsection, let $X$ be a smooth projective curve of genus $g\geq 2$, and $G=\GL_n$ for some $n\in \N$. Then there is a non-empty Zariski-open locus $\A^{\circ}\subseteq \A$ where the spectral curve $Z\to \A$ from \S\ref{s:Ngo-G=GL_n} is smooth proper with connected geometric fibres \cite[Remarks~3.1, 3.5]{BNR-spectral-curves}.

	\begin{definition}
		We call $\A^\circ$ the \textit{regular locus} of the Hitchin base. Let $Z^\circ:=Z\times_{\A}\A^\circ\to \A^\circ$ and 
		\[\CHig_n^\circ\subseteq \CHig_n,\quad \bfHig_n^\circ\subseteq \bfHig_n\]
		be the respective preimages of the regular locus $\A^{\circ}$ under the Hitchin morphism $h$ (\Cref{d:Hitchin-map}).	For any $T\in \Perf_K$, we call a Higgs bundle $(E,\theta)$ over $X_T$ \textit{regular} if $h(E,\theta)\in  \A^\circ(T)$. 
		One can show that the locus $\CHig^\circ_n$ describes the part of $\CHig_n$ where $\theta$ is regular semi-simple as an endomorphism of $E$.
		
		 Similarly, let
		\[\CBun_{n,v}^\circ\subseteq \CBun_{n,v},\quad \bfBun_{n,v}^\circ\subseteq \bfBun_{n,v}\]
		be	the preimages of the regular locus under $\wt h$ (\Cref{d:Hitchin-Betti}). We call a v-bundle $V$ regular if $\wt h(V)\in  \A^{\circ}$.
\end{definition}

There is an analogue in our $p$-adic analytic setup of a well-known result of Beauville--Narasimhan--Ramanan, called the ``BNR correspondence''.
This construction is often referred to as ``abelianisation'', since it reduces the non-abelian Hodge theory of the group $\GL_n$ to that of the abelian group $\G_m$:
\begin{prop}\label{c:BNR-an}
	Let $T\in \Perf_K$ and let $(E,\theta)$ be a regular Higgs bundle on $X_T$ with Hitchin image $b:T\to \A^\circ$. Let $\pi:Z_b\to X_T$ be the fibre of the spectral curve over $b$.
	Then there is a line bundle $L$ on $Z_b$ such that 
	$(E,\theta)\simeq \pi_{\ast}(L,\tau_{\can})$ where $\tau_{\can}$ is the tautological section of $\pi^\ast \widetilde{\Omega}_{X_T}^1$ from \Cref{d:spectral-curve}. 
	Moreover, $\pi_{\ast}$ induces an isomorphism $\uAut(L)\simeq \uAut(E,\theta)$. 
\end{prop}
\begin{proof}
	Due to \Cref{t:perf-GAGA}, this follows from the algebraic case {\cite[Proposition~3.6]{BNR-spectral-curves}}. More precisely, due to the  localisation on $T$ in \Cref{t:perf-GAGA},  we a priori get the statement locally on $T$, but we can then glue the local line bundles using the isomorphism $\uAut(L)\simeq \uAut(E,\theta)$. 
\end{proof}

	\begin{prop}\label{p:Hitchin-for-Higgs}
		\begin{enumerate}
			\item The regular Hitchin morphism $h:\CHig_n^\circ\to  \A^\circ$ is canonically isomorphic to the relative Picard groupoid
			\[ \CPic_{ Z^\circ| \A^\circ}\to  \A^\circ.\]
			\item In terms of coarse moduli spaces, the Hitchin map $H:\bfHig_n^\circ\to  \A^\circ$ is canonically isomorphic to 
			\[ (\bfPic_{Z^\circ|\A^\circ})^\dia \to  \A^\circ\]
			where $(\bfPic_{Z^\circ|\A^\circ})^\dia$ is the analytification of the algebraic Picard variety of $Z^\circ\to A^\circ$.
		\end{enumerate}
	\end{prop}
	\begin{proof}
		Part (1) is immediate from \Cref{c:BNR-an}. Part (2) follows from this using \Cref{t:perf-GAGA}.
	\end{proof}
	In particular, this shows that $\bfHig_n^\circ$ is represented by a smooth rigid space. Moreover, we deduce the following, which justifies calling the analytic Hitchin morphism a ``fibration'':
	\begin{coro}
		Let 
		$\bfHig_{n,d}^{\circ}\subseteq \bfHig_n$ be the subsheaf of isomorphism classes of regular Higgs bundles of degree $d$. Then the geometric fibres of the Hitchin morphism
		$H:\bfHig_{n,d}^{\circ}\to  \A^\circ$
		are abelian varieties.
	\end{coro}
	\subsection{The regular Hitchin fibration on the Betti side}
	We now turn our attention to the coarse moduli space $\bfBun_{n,v}^\circ$ of regular v-vector bundles.
	\begin{lemma}
			The restriction  of the sheaf $\cH_{\bX}$ on $\A_v$ from \Cref{def:cHX} to the regular locus  $\A^\circ$ is an \'etale-locally constant sheaf $\cH_{\bX}^\circ$. In particular, it is represented by a relative rigid group  in $\A^\circ_\et$.
	\end{lemma}
	\begin{proof}
		By \Cref{c:exp-splits-L_X}, the map $\cH_{\bX}\to \A_v$ is an \'etale torsor under $\Lambda\otimes \mu_{p^\infty}$. 
		By \Cref{t:struct-thm-R1fvastU}, it thus suffices to prove that for every $m\in \N$, the sheaf $\rR^1\pi_{\ast}\mu_{p^m}$ is \'etale-locally constant over the regular locus. Since over the regular locus $\pi$ is smooth and proper, this holds by \cite[Theorem 10.5.1]{ScholzeBerkeleyLectureNotes}.
		
		It follows that $\cH_{\bX}^\circ\to  \A^\circ$ defines an object in $\A_{\et}^\dia=\A_\et$. Hence it is represented by a rigid space. 
	\end{proof}
	
	We now come to the version of our main theorem in which the ``twisted isomorphism'' of moduli spaces takes its strongest geometric incarnation, as a morphism of rigid spaces between coarse moduli spaces:
	\begin{theorem} \label{t:moduli-spaces-circ}
		Let $X$ be a smooth projective curve of genus $g\geq 2$ and let $n\in \N$.
		\begin{enumerate}
			\item The v-sheaf $\bfBun_{n,v}^\circ$ is represented by a smooth rigid space over $K$.
			\item The Hitchin fibration $\wt H:\bfBun^\circ_{n,v} \to  \A^\circ$ over the regular locus is a torsor under $\bfP^\circ:=(\bfPic_{Z^\circ|\A^\circ})^\dia$. In particular, its geometric fibres are disjoint unions of abelian varieties.
			\item There is a natural isomorphism of rigid spaces
			\[\cH_{\bX}^\circ\times^{\bfP^\circ[p^\infty]}\bfHig_n^\circ\isomarrow \bfBun_{n,v}^\circ.\]
		\end{enumerate}
	\end{theorem}
	\begin{coro}
		 There is a natural isomorphism of rigid spaces
		\[\bfHig_n^\circ \times_{ \A^\circ}\cH^\circ_{\bX}\isomarrow \bfBun_{n,v}^\circ \times_{ \A^\circ}\cH^\circ_{\bX}.\]
	\end{coro}
	\begin{rem}
		More generally, for reductive $G$, using \cite{DG02} and the arguments below,  representability and an analogue of \Cref{t:moduli-spaces-circ} holds for $\bfHig_G$, $\bfBun_{G,v}$ over the \textit{very regular locus} of $\A_G$ \cite[\S4]{Ngo}.
	\end{rem}
	\begin{rem}
		Combined with \Cref{p:Hitchin-for-Higgs}, this says that both $\bfHig_n^\circ$ and $\bfBun^\circ_{n,v}$ are $\bfP^\circ$-torsors over $ \A^\circ$ via their respective Hitchin fibrations. But  the former torsor is split, whereas the latter is non-split:
	\end{rem}
	\begin{example}
		Assume $n=1$, then $\A=\A^\circ=\rH^0(X,\wtOm)\otimes \G_a$ and $Z=X$ and $\mathbf P^\circ=\mathbf P=(\bfPic_X)^\dia$. In this case, $\bfHig_n=\mathbf P\times \A$, while $\bfBun_{1,v}$ sits in a non-split short exact sequence of rigid group varieties
		\[ 0\to \mathbf P\to \bfBun_{1,v}\to \A\to 0.\]
		This case was previously shown  in \cite[Theorem 1.3]{diamantinePicard} and \cite[\S5]{Heu22b}, so we may regard \Cref{t:moduli-spaces-circ} as a generalisation from $\G_m$ to the non-abelian case of $\GL_n$, at least when $X$ is a curve.
	\end{example}
	\begin{proof}
		We have already seen in \Cref{t:coarse-moduli-twisted} that we have an isomorphism as in (3) on the level of v-sheaves.  Observe now that we can interpret \Cref{p:Hitchin-for-Higgs} as saying that $\bfHig_n^\circ$ is a trivial torsor under $\bfP^\circ$.  Since  $\cH_{\bX}$ is an \'etale $\bfP[p^\infty]:=(\bfPic_{Z|\A})^\dia[p^\infty]$-torsor, it follows formally that $\bfBun_{n,v}^\circ$ is the pushout
		\begin{equation}\label{eq:Bun-is-pushout-of-L}
			\bfBun_{n,v}^\circ\simeq \cH_{\bX}^\circ\times^{\bfP^\circ[p^\infty]}\bfHig_n^\circ\simeq \cH_{\bX}^\circ\times^{\bfP^\circ[p^\infty]}\bfP^\circ,
		\end{equation}
		along  $\bfP^\circ[p^\infty]\to\bfP^\circ$. This proves (2).
		It remains to see (1). We deduce this from the following:
		
		\begin{lemma}\label{l:representability-Bun-axiomatised}
			Let $S$ be any rigid space and let $G\to S$ be a commutative smooth relative group. Let $F$  be an abelian sheaf on $S_v$ that fits into a short exact sequence of abelian sheaves on $S_v$
			\[ 0\to G[p^\infty]\to F\to \G_a^d\to 0 \]
			for some $d\in \N$. Then the pushout $F\times^{ G[p^\infty]}G\to S$ is represented by a smooth morphism of rigid spaces.
		\end{lemma}
		\begin{proof}
			This is an axiomatisation of the argument in \cite[Corollary 2.9.5]{diamantinePicard}:
			The statement is local on $S$, so we may assume that $S$ is quasi-compact.
			For any $r\in \Z$, write $\mathbb B_r:=p^r \G_a^{+d}$ for the closed disc of radius $|p^r|$. Let $F_r$ be the pullback of $F$ to $\mathbb B_r$. By compactness, we then have 
			\[ \Ext_{S}^1(\mathbb B_r,G[p^\infty])=\textstyle\varinjlim_j \Ext_{S}^1(\mathbb B_r,G[p^j]),\]
			hence there is $j$ such that $F_r$ admits a reduction of structure group to $G[p^j]$.
			
			Consider now the morphism $[p^j]:F\to F$, which induces a morphism of short exact sequences
			\[\begin{tikzcd}
				0 \arrow[r] & {G[p^\infty]} \arrow[d, "p^j"] \arrow[r] & F_r \arrow[d, "p^j"] \arrow[r] & \mathbb B_r \arrow[d, "p^j","\sim"'] \arrow[r] & 0 \\
				0 \arrow[r] & {G[p^\infty]} \arrow[r]                  & F_{r+j} \arrow[r]              & \mathbb B_{r+j} \arrow[r]              & 0.
			\end{tikzcd}\]
			This can be interpreted as the pushout along $[p^j]:G[p^\infty]\to G[p^\infty]$.
			Since the class of $F_r$ in $\Ext_{S}^1(\mathbb B_r,G[p^\infty])$ is $p^j$-torsion, this shows that the bottom sequence is split.
			
			Let now $P_r$ be the pushout of $F_r$ along $G[p^\infty]\to G$. Then since the bottom sequence is split, we have
			$P_{r+j}=G\times_S \mathbb B_{r+j}$,
			which is consequently represented by a rigid space that is smooth over $S$. 
			
			Considering the pushout of the diagram over the morphism $[p^j]:G\to G$, which is \'etale by \Cref{c:G[p]-etale}, we see that the map
			$P_r\to P_{r+j}$
			is an \'etale morphism of v-sheaves. Since $P_{r+j,\et}^\dia=P_{r+j,\et}$, this shows that also $P_r$ is representable and smooth over $S$. Since $P$ is glued from the $P_r$, this shows the result.
		\end{proof}
		
		Recall now that by construction in \Cref{d:H_X} and \Cref{def:cHX}, there is a Cartesian diagram
		\begin{equation}\label{eq:Cartesian-descr-HX}
		\begin{tikzcd}
			\cH_{\bX} \arrow[d] \arrow[r] &  \A \arrow[d,"s_{\bX}\circ \tau" ] \\
			\rR^1f_{v\ast}\widehat{\mathcal B^\times} \arrow[r]    & \rR^1f_{v\ast}\mathcal B.
		\end{tikzcd}
		\end{equation}
		By \Cref{l:RnfvastLieUJ}, the term on the bottom right is isomorphic to $\mathbb G_a^d\times  \A$ for some $d$.
		By \Cref{p:log-seq-R^1fv}, the bottom morphism is surjective and has kernel $\rR^1f_{v\ast}\mathcal B^\times[p^\infty]$. When we now consider the restriction of this diagram to $ \A^\circ$, then $(\rR^1f_{v\ast}\mathcal B^\times[p^\infty])_{| \A^\circ} =\bfP^\circ[p^\infty]$. Hence we may apply \Cref{l:representability-Bun-axiomatised} to see that 
		\[ 	\mathcal M:=\rR^1f_{v\ast}\widehat{\mathcal B^\times}_{| \A^\circ}\times^{\bfP^\circ[p^\infty]}\bfP^\circ\]
		is represented by a smooth rigid space over $\A^\circ$. On the other hand, using \eqref{eq:Bun-is-pushout-of-L}, we now deduce from applying $-\times^{\bfP^\circ[p^\infty]}\bfP^\circ$ to the left hand side of \eqref{eq:Cartesian-descr-HX} that we have a Cartesian diagram
		\[
		\begin{tikzcd}
			\bfBun_{n,v}^\circ \arrow[d] \arrow[r,"\wt H"] & \A^\circ \arrow[d,"s_{\bX}\circ \tau" ] \\
			\mathcal M \arrow[r]    & (\rR^1f_{v\ast}\mathcal B)_{|\A^\circ}.
		\end{tikzcd}\]
		This exhibits $\bfBun_{n,v}^\circ$ as the fibre product of rigid spaces, hence it is itself represented by a rigid space. Moreover, since the bottom morphism is smooth, and $\A^\circ$ is smooth, it follows that $\bfBun_{n,v}^\circ$ is smooth.
	\end{proof}
	\begin{rem}
		Conceptually, the object $\mathcal M$ in the proof has the following meaning: We expect there to be a variant of  \Cref{p:leray-seq-for-UJ} saying that there is a Leray exact sequence  	(cf \cite[Theorem~2.4]{Heu23})
		\begin{equation}\label{eq:rem-M-seq-1}
			1\to \rR^1f_{\Et\ast}\B^\times \to \rR^1f_{v\ast}\B^\times \to f_{v\ast}(\Lie  J\otimes \widetilde{\Omega}_X)\to 0.
		\end{equation}
		By a comparison of Leray sequences, this would receive a natural morphism from the short exact sequence 
		\begin{equation}\label{eq:rem-M-seq-2}
			1\to \rR^1f_{\Et\ast}(\B^\times[p^\infty]) \to \rR^1f_{v\ast}\widehat{\B^\times} \to \rR^1f_{v\ast}\Lie  J\to 0.
		\end{equation}
		
		Over $\A^\circ$, the first term of \eqref{eq:rem-M-seq-1} is equal to $\bfP$, whereas the first term of \eqref{eq:rem-M-seq-2} is $\bfP[p^\infty]$.  Hence $\mathcal M$ sits in between these two sequences, and we can think of it as a replacement for $\rR^1f_{v\ast}\B^\times$ that is technically easier to work with. This also explains the relation to twists by invertible $\B$-modules as used in \cite{Heu23}.
	\end{rem}

\section{The stack of representations of $ \pi_1^{\et}(X,x)$} \label{s:representations}
	
	\subsection{Pro-finite-\'etale vector bundles on curves}
	Throughout this section, let $G$ be a linear algebraic group over $K$ and let $X$ be a connected smooth projective curve over $K$ of genus $g$, considered as an adic space. We fix a base-point $x\in X(K)$, then we have the \'etale fundamental group $\pi^\et_1(X,x)$, a profinite group. Since $X$ and $x$ will be fixed, we simply denote $\pi^\et_1(X,x)$ by $\pi$ throughout the following subsections.
	\begin{definition}[{\cite[Definition~4.6]{HeuSigma}}]
		The \textit{universal pro-finite-\'etale cover} of $X$ is the diamond \[\wt X=\varprojlim_{X'\to X}X'\] where the index category is given by the connected finite \'etale covers $X'\to X$ together with a lift $x'\in X(K)$ of $x$. Then the projection
		$q:\wt X\to X$
		is a pro-\'etale $\underline{\pi}$-torsor. 
		We recall two key technical properties:
	\end{definition}
	
	\begin{prop}[{\cite[Corollary~5.6]{perfectoid-covers-Arizona}, \cite[Proposition~3.9]{Heu22b}}]\label{p:wtXperf}
		Assume that $g\geq 1$.
		\begin{enumerate}
			\item
			$\wt X$ is represented by a quasi-compact perfectoid space. 
			\item For any affinoid perfectoid space $S$ and any $\epsilon>0$, we have \[\rH^0(\wt X\times S,\O)=\O(S),\quad \rH^1_v(\wt X\times S,\O^+/p^\epsilon)\,\aeq\, 0.\]
			\item For any linear algebraic group $G$, we have $\rH^0(\wt X\times S,G)=\rH^0(S,G)$.
		\end{enumerate}
	\end{prop}
	\begin{proof}
		In \cite[Proposition~3.9]{Heu22b}, it is shown that $\rH^i_v(\wt X\times S,\O^+/p^k)\aeq \rH^i_v(S,\O^+/p^k)$ for any $k\in \N$ and $i=0,1$. The statement for general $\epsilon$ follows by considering for any $k\in \N$ the long exact sequence of
		\[ 0\to \O^+/p^{\epsilon}\to \O^+/p^k\to \O^+/p^{k-\epsilon}\to 0.\]
		Since $\wt X\times S\to S$ admits a splitting, it is clear that $\rH^i_v(S,\O^+/p^\epsilon)\to \rH^i_v(\wt X\times S,\O^+/p^\epsilon)$ is injective. This is enough to show the statement for $i=0$. For $i=1$, the statement now follows from the 5-Lemma.
		
		To deduce (3), we note that (2) implies the case of $G=M_n(\O)$, which implies the case of $G=\GL_n(\O)$. This implies the general case by choosing a faithful representation of $G$.
	\end{proof}
	
	The relevance of the pro-finite-\'etale cover in this context stems from the following construction from \cite{HeuSigma}: Let $G$ be a linear algebraic group over $K$. Let $S$ be an affinoid perfectoid space over $K$ and let 
	\[\rho:\pi\to G(S)\]
	be a continuous homomorphism. 
	Then we can regard $\rho$ as a 1-cocycle in v-sheaves $\rho:\underline{\pi}\times S\to G$ (see \cite[\S2.3]{HWZ} for a detailed discussion of this fact, and the natural topology on $G(S)$). After base-changing to $G_S=G\times S$, we can regard $\rho$ as a homomorphism of relative adic group over $S$. We can therefore associate to $\rho$ a $G$-torsor $V_\rho$ on $(X\times S)_v$, defined as the pushout of $q:\wt X\times S\to X\times S$ along $\rho$:
	\[ V_{\rho}:=(\wt X\times S)\times^{\underline\pi}G_{S}.\]
	For any two representations $\rho_{1,2}:\underline{\pi}\times S\to G$, we define the set of morphisms $\rho_1\to\rho_2$ as the set of  $g\in G(S)$ such that $\rho_1=g^{-1}\rho_2g$. We will see below that this defines a fully faithful functor
	\[
	{\Big\{\begin{array}{@{}c@{}l}\text{ continuous representations}\\\pi\to G(S)\end{array}\Big\}} \hookrightarrow {\Big\{\begin{array}{@{}c@{}l}\text{$G$-torsors on $(X\times S)_v$}\end{array}\Big\}},\quad \rho\mapsto V_\rho.
	\]
	
	Next, we wish to pass to moduli stacks. But for varying $S$, the left hand side does not yet satisfy v-descent. To rectify this, we need to slightly generalise the construction: Let $E$ be any $G$-torsor on $S_v$. By \cite[Theorem~1.1]{heuer-G-torsors-perfectoid-spaces}, this is already \'etale-locally trivial. Thus $\uAut_G(E)\to S$ is a smooth relative group over $S$.
	\begin{definition}\label{def:cts-rep-on-torsor}
		By a \textit{representation of $\pi$ on $E$}, we mean a homomorphism of sheaves on $S_v$
		\[\rho:\underline{\pi}\to \uAut_G(E).\]
		A morphism between two such representations is a morphism of $G$-torsors compatible with $\pi$-actions.  By the same construction as above, we can associate to $\rho$ a v-topological $G$-torsor $V_\rho:=(\wt X\times S)\times^{\underline\pi}E$ on $X\times S$.
	\end{definition}
	\begin{definition}
	We denote by $\CRep_G(\pi)$ the prestack on $\Perf_{K,v}$ defined by sending $S$ to the groupoid of pairs $(E,\rho)$ consisting of a $G$-torsors $E$ on $S_v$ and a representation of $\pi$ on $E$.
	\end{definition}
	\begin{definition}
		\begin{enumerate}
			\item Let $S\in \Perf_K$. Let $q:\wt X\times S\to X\times S$ and $\pr_S:\wt X\times S\to S$ be the projection maps. We call a $G$-torsor $V$ on $(X\times S)_v$ \textit{pro-finite-\'etale} if $q^{\ast}V\simeq \pr_S^\ast E$ for some $G$-torsor $E$ on $S_v$.
			\item Let $\CBun_{G,v}^{\profet}\subseteq \CBun_{G,v}$ be the sub-pre-stack on $\Perf_{K}$ of pro-finite-\'etale $G$-torsors.
			\end{enumerate}
	\end{definition}
	\begin{lemma}\label{l:pushforward-of-profet-G-torsor}
		$\CBun_{G,v}^{\profet}$ is a v-stack. Indeed,
		for any v-$G$-torsor $W$ on $\wt X\times S$, the following are equivalent:
		\begin{enumerate}
			\item $W=\pr_S^\ast E$ for some $G$-torsor $E$ on $S_v$,
			\item the v-sheaf $\pr_{S,\ast} W$ on $S_v$ is a $G$-torsor.
		\end{enumerate}
		In this case, the adjunction maps $\pr_{S}^\ast\pr_{S,\ast} W\to W$ and $E\to \pr_{S,\ast}\pr_S^\ast E$ are isomorphisms.
	\end{lemma}
	\begin{proof}
		Assume that $W=\pr_S^{\ast}E$. To see (2), it suffices to see that the adjunction map $E\to  \pr_{S,\ast}\pr_S^{\ast}E=\pr_{S,\ast}W$ is an isomorphism. As we can check this locally on $S_v$, we can reduce to the case $E=G$. Then the statement follows from \Cref{p:wtXperf}, which says that $G=\pr_{S,\ast}\pr_S^{\ast}G$.

		Conversely, assume that $\pr_{S,\ast} W$ is a $G$-torsor. As the adjunction map $\pr_{S}^\ast\pr_{S,\ast}W\to W$ is clearly $G$-equivariant, it is then a morphism of $G$-torsors, hence an isomorphism. Setting $E:=\pr_{S,\ast}W$ gives (1).
		
		Finally, the v-stack property follows from the equivalence because (2) can be checked locally on $S_v$.
	\end{proof}
	We now have the following generalisation of \cite[Theorem 5.2]{HeuSigma}:
	\begin{prop}\label{p:Thm5.2-in-families}
		For $S\in \Perf_K$,
		the construction of \Cref{def:cts-rep-on-torsor} defines an equivalence of groupoids:
		\[ 	{\Bigg\{\begin{array}{@{}c@{}l}\text{ representations}\\\underline{\pi}\to \uAut_G(E) \\\text{on $G$-torsors $E$ on $S_{v}$}\end{array}\Bigg\}} \isomarrow {\Bigg\{\begin{array}{@{}c@{}l}\text{pro-finite-\'etale}\\\text{$G$-torsors on $(X\times S)_v$}\end{array}\Bigg\}},\quad
		\begin{aligned}
				(E,\rho:\underline{\pi}\to \uAut_G(E))\,\mapsto \,&V_{\rho}:=(\wt X\times S)\times^{\underline\pi}E\\
			\pr_{S,\ast}q^{\ast}V\,\mapsfrom\, &V
		\end{aligned}
		\]
		In particular, this defines a natural equivalence of v-stacks
		\[\CRep_G(\pi)\isomarrow \CBun_{G,v}^{\profet}.\]
	\end{prop}
	\begin{proof}
		It is clear from the construction that $q^\ast V_{\rho}=\pr_S^{\ast}E$, so $V_\rho$ is indeed pro-finite-\'etale. 
		
		In order to construct the inverse functor, let $V$ be a pro-finite-\'etale $G$-torsor on $(X\times S)_v$. Then by  \Cref{l:pushforward-of-profet-G-torsor}, the v-sheaf $E:=\pr_{S,\ast}q^{\ast}V$ on $S_v$ is a $G$-torsor. 
		Since $\pr_{S}$ is $\pi$-equivariant, the natural $\pi$-action on $q^{\ast}V$ thus induces a $\pi$-action on $E$.
		 This defines an object in the left hand side.
		 Using the last sentence of \Cref{l:pushforward-of-profet-G-torsor}, one easily checks that these constructions are quasi-inverse to each other.
	\end{proof}
	\begin{rem}
		\Cref{p:Thm5.2-in-families} is no longer true in general when $G$ is allowed to be a more general algebraic or rigid group: See \cite[\S5.6]{HWZ} for a counterexample when $G$ is an abelian variety.
	\end{rem}

	\subsection{The pro-finite-\'etale locus of $\CBun_{G,v}$}
	The goal of this subsection is to prove:
	\begin{theorem}\label{t:profet-locus-open}
		The sub-v-stack
		$\CBun_{G,v}^{\profet}\subseteq \CBun_{G,v}$ is  open. In particular, \Cref{p:Thm5.2-in-families} induces a natural open immersion of v-stacks
		\[ \CRep_{G}(\pi)\hookrightarrow \CBun_{G,v}.\]
	\end{theorem}
	\begin{rem}
		Let $\mathbf{Rep}_G(\pi)$ be the v-sheaf that sends $S\in \Perf_K$ to the set of conjugacy classes of continuous representations $\rho:\pi\to G(S)$. This is the quotient by conjugation of the representation variety $\FHom(\pi,G)$ of \cite[\S8.6]{Heu}. \Cref{t:profet-locus-open} induces an open immersion of v-sheaves 
		$\mathbf{Rep}_G(\pi)\hookrightarrow \mathbf{Bun}_{G,v}$. Together with \Cref{t:moduli-spaces-circ}, this implies that the fibre of $\mathbf{Rep}_G(\pi)\to \A$ over $\A^\circ\subseteq \A$ is represented by a smooth rigid space. This gives a (regular) $p$-adic analogue of the character variety from complex geometry.
 	\end{rem}
		
	\begin{proof}[{Proof of \Cref{t:profet-locus-open}}]
	For the proof, we use the following series of lemmas.
	\begin{lemma}\label{l:U-torsors-over-wtX}
		There is a rigid open subgroup $U\subseteq G$ such that for any affinoid perfectoid space $S$, we have
		\[ \rH^1_v(\wt X\times S,U)=1.\]
	\end{lemma}
	\begin{proof}By \cite[Corollary~3.8]{heuer-G-torsors-perfectoid-spaces}, we may replace $G$ by an open subgroup to assume that $G$ has good reduction. 
		From \Cref{p:wtXperf}.(2), we deduce by tensoring with $\mathfrak m$ that $\rH^1_v(\wt X\times S,\mathfrak m\O^+/p^\epsilon\mathfrak m)= 0$. As $\wt X$ is perfectoid by \Cref{p:wtXperf}.(1), we may thus apply \cite[Lemma~4.27]{heuer-G-torsors-perfectoid-spaces}, which gives the result.
	\end{proof}
	
	\begin{lemma}\label{l:spreading-out-profet-property}
		Let $S\in \Perf_K$ and let $V$ be a $G$-torsor on $(X\times S)_v$. If $\eta:\Spa(C,C^+)\to S$ is any geometric point of $S$ such that the pullback of $V$ to $X\times \eta$ is pro-finite-\'etale, then there is an open subspace $W\subseteq S$ containing $\im(\eta)$ such that the restriction of $V$ to $X\times W$ is still pro-finite-\'etale.
	\end{lemma}
	\begin{proof}
		Let $U\subseteq G$ be as in \Cref{l:U-torsors-over-wtX}. As any v-$G$-torsor on $\Spa(C,C^+)$ is trivial by \cite[Theorem~1.1]{heuer-G-torsors-perfectoid-spaces}, the condition on means that $V$ becomes trivial on $\wt X\times \Spa(C,C^+)$. In particular, it admits a reduction of structure group to $U$ on $\wt X\times \Spa(C,C^+)$.
		We now invoke \cite[Corollary~4.12]{heuer-G-torsors-perfectoid-spaces}: This asserts that there is an \'etale map $S'\to S$ whose image contains $\im(\eta)$ such that $V$ admits a reduction of structure group to $U$ on $\wt X\times S'$.  By \Cref{l:U-torsors-over-wtX}, this means that $V$ becomes trivial on $\wt X\times S'$. Let now $W\subseteq S$ be the image of $S'\to S$. This is open since it is the image of an \'etale map, and it contains $\eta$ by construction. Then $S'\to W$ is an \'etale cover. Hence by \Cref{l:pushforward-of-profet-G-torsor}, $V$ is pro-finite-\'etale already over $X\times W$.
	\end{proof}
	At this point, we have proved the following ``fibrewise criterion for pro-finite-\'etale bundles'':
	\begin{prop}\label{p:fibre-wise-criterion}
		Let $S$ be a perfectoid space.
		Let $V$ be a $G$-torsor on $X\times S$. Suppose that for every geometric point $\eta:\Spa(C,C^+)\to S$, the pullback of $V$ to $X\times \eta$ is pro-finite-\'etale. Then $V$ is pro-finite-\'etale.
	\end{prop}
	\begin{proof}
		By \Cref{l:spreading-out-profet-property}, we can find an open cover $\mathfrak U=(S_i\to S)_{i\in I}$ such that the restriction of $V$ to $X\times S_i$ is pro-finite-\'etale for each $i\in I$. By the first sentence of \Cref{l:pushforward-of-profet-G-torsor}, this implies that $V$ is pro-finite-\'etale.
	\end{proof}
	We can now complete the proof of \Cref{t:profet-locus-open}: Recall that $\CBun_{G,v}$ is a small v-stack, so it has an associated topological space $|\CBun_{G,v}|$, see \cite[Definition 12.8.]{Sch18}. This is given by the equivalence classes of maps $\Spa(C,C^+)\to \CBun_{G,v}$ where $(C,C^+)$ is any affinoid field over $K$.
		It now follows from \Cref{p:fibre-wise-criterion} that $\CBun_{G,v}^{\profet}$ is the sub-v-stack associated to the subspace $U\subseteq |\CBun_{G,v}|$ of points represented by maps $\Spa(C,C^+)\to \CBun_{G,v}$ for which the corresponding v-$G$-torsor on $X\times \Spa(C,C^+)$ is pro-finite-\'etale. 
		
		To see that $U$ is open, it now suffices to see that for any map $Y\to \CBun_{G,v}$ from an affinoid perfectoid space, the preimage of $U$ under the composition $|Y|\to |\CBun_{G,v}|$ is open. This is precisely  \Cref{l:spreading-out-profet-property}.
	\end{proof}

	\subsection{Relation to semi-stability and degree}
	Of particular interest to non-abelian $p$-adic Hodge theory is the question raised by Faltings which Higgs bundles correspond to continuous representations of $\pi$ under the $p$-adic Simpson correspondence \eqref{eq:paS}. 
	More geometrically, for $n\in \N$, one can ask which locus of $\CHig_{\GL_n}$ corresponds to $\CRep_{\GL_n}(\pi)\hookrightarrow \CBun_{\GL_n,v} $ under the twisted isomorphism \Cref{t:fCiso}. But we  can now reduce this question  to the case of $C$-points by \Cref{p:fibre-wise-criterion}: Indeed, it follows from this that the correct condition on Higgs bundles needs to satisfy the analogous ``fibrewise criterion'' as in \Cref{p:fibre-wise-criterion}. 
	
	Using the results of this article, the above question can thus be formulated in a more precise way by asking which Higgs bundles are associated to pro-finite-\'etale v-vector bundles in the following sense:
	\begin{definition}\label{d:assoc}
		Let $(E,\theta)$ be a Higgs bundle on $X$ and let $V$ be a v-vector bundle on $X$, both of rank $n$. We say that $(E,\theta)$ and $V$ are \textit{associated} if $h(E,\theta)=b=\wt h(V)$ and there is a $\wh J_b$-bundle $L\in \CH(b)$ such that \[\CS(L,(E,\theta))\simeq V,\] where $\CS$ is the $p$-adic Simpson correspondence of \Cref{t:fCiso} for $G=\GL_n$.
	\end{definition}
	Our last goal is to prove the following, which proves and generalises an assertion by Faltings \cite[\S5]{Fal05}:
	\begin{prop}\label{p:image-sst-deg-0}
		Let $(E,\theta)$ be a Higgs bundle on $X$ that is associated to a pro-finite-\'etale v-vector bundle. Then $(E,\theta)$ is semi-stable of degree $0$.
	\end{prop}
	\begin{proof}
		We first show  that $\deg E=0$.
		For this we use that by \cite[Theorem 1.3]{diamantinePicard}, the \mbox{v-Picard} functor $\mathbf{Bun}_{\G_m,v}=\mathbf{Pic}_{X,v}$ is a rigid group whose group of connected components can be canonically identified with
		\[ \pi_0(\mathbf{Pic}_{X,v})=\pi_0(\mathbf{Pic}_{X,\et})=\Z.\]
		\begin{definition}
			 The \textit{degree of a v-line bundle} on $X$ is the connected component $\in \pi_0(\mathbf{Pic}_{X,v})= \Z$ of the associated $K$-point of $\mathbf{Pic}_{X,v}$.
			The \textit{degree of a v-vector bundle} $V$ on $X$ is $\deg V=\deg(\det V)$.
		\end{definition}
		 		\begin{lemma}\label{p:deg-profet-vvb}
		 	Let $V$ be a pro-finite-\'etale v-vector bundle on $X$. Then $\deg V=0$.
		 \end{lemma}
		 \begin{proof}
		 	Since the formation of $\det$ commutes with pullback, $\det V$ is itself a pro-finite-\'etale line bundle. Hence we may assume that $G=\G_m$. In this case, any pro-finite-\'etale v-line bundle $L$ on $X$ is topologically torsion by \cite[Theorem 3.6]{Heu22b}: This means that $L^{n!}\to 0$ in $\mathbf{Pic}_{X,v}(K)$ for $n\to \infty$. It follows that its image in the discrete group $\pi_0(\mathbf{Pic}_{X,v})$ is torsion. Since $\pi_0(\mathbf{Pic}_{X,v})\simeq \Z$, this means that $\deg(L)=0$.
		 \end{proof}
		\begin{lemma}\label{l:deg-of-associated-bundle}
			Let $(E,\theta)$ be a Higgs bundle associated to a v-vector bundle $V$. Then $\deg(E)=\deg(V)$.
		\end{lemma}
		\begin{proof}
			By functoriality of $\CS$ applied to $\det$ (\Cref{p:functoriality-of-CS} and \Cref{r:det-Kostant}),  $\det (E,\theta)$ is associated to $\det V$. Hence we may assume $G=\G_m$. Then by \Cref{l:CS-for-GLn}, there is a v-line bundle $L$ on $X$ with a reduction of structure group to $\wh{\G}_m$ such that $V=E\otimes L$. As $\deg$ is additive, it therefore suffices to prove that $\deg L=0$. But $L$ is pro-finite-\'etale by \cite[Theorem~3.6]{Heu22b}, so this follows from \Cref{p:deg-profet-vvb}.
		\end{proof}
Together, \Cref{p:deg-profet-vvb} and \Cref{l:deg-of-associated-bundle} combine to show that $(E,\theta)$ in \Cref{p:image-sst-deg-0} has degree $0$.

To see the semi-stability, we first need to understand what $\CS$ does to sub-Higgs bundles: 
	\begin{lemma}\label{l:sub-bundle-associated}
		Let $(E,\theta)$ be a Higgs bundle associated to a v-vector bundle $V$. Assume that $(N,\theta_{|N})\subseteq (E,\theta)$ is a sub-Higgs bundle. Then there is a sub-v-vector bundle $W\subseteq V$ that is associated to $(N,\theta_{|N})$.
	\end{lemma}
	\begin{proof}
		With notation as in \S\ref{s:CS-for-GL_n}, that $(E,\theta)$ is associated to $V$ means that there is an invertible $\B$-module $L$ on $X_v$ such that $V=\nu^\ast E\otimes_{\B}L$ where the $\B$-action on $\nu^\ast E$ comes from the homomorphism of $\O_{X_v}$-algebras
		$\theta:\B\to \uEnd(\nu^\ast E)$
		obtained from applying $\nu^\ast$ to \eqref{eq:G=GL_n-action-B-on-E-varphi}. 
		 That $\theta$ preserves $N$ means that any endomorphism in the image sends $N$ into $N$. It thus induces by restriction a homomorphism $\B\to \uEnd(\nu^\ast N)$. By functoriality, this factors through the sheaf $\B_{N}$ obtained by applying the constructions of \S\ref{s:CS-for-GL_n} to $(N,\theta_{|N})$. Hence \[\nu^\ast N\otimes_{\B_N}(\B_N\otimes_\B L)=\nu^\ast N\otimes_{\B}L\subseteq \nu^\ast E\otimes_{\B}L.\] By \Cref{l:CS-for-GLn}, this means that
		 $\CS(\B_N\otimes_\B L,(N,\theta_{|N}))\subseteq \CS(L,(E,\theta))=V$
		 is associated to $(N,\theta_{|N})$.
	\end{proof}
	With these preparations, our argument to show  semi-stability in \Cref{p:image-sst-deg-0} is now a generalisation of that of Würthen, who previously considered the case of vanishing Higgs field $\theta$ \cite[Proposition~4.16]{wuerthen_vb_on_rigid_var} in the context of the functor of parallel transport of Deninger--Werner \cite{DeningerWerner_vb_p-adic_curves}.
	When $K=\mathbb{C}_p$, \Cref{p:image-sst-deg-0} can be also shown by extending Deninger--Werner's construction to a larger class of Higgs bundles \cite[Proposition~5.3.1]{Xu22}, namely those that have vanishing Higgs field after ``twisted pullback'' in Faltings' sense.
	
	\begin{lemma}\label{l:sections-LB-on-wtX}
		Assume $g\ge 1$. Let $L$ be a v-line bundle on $X$ with $\deg(L)>g-1$, then $\dim\rH^0(\wt X,L)=\infty$.
	\end{lemma}
	\begin{proof}
		By \cite[Theorems~1.3 and~5.7]{HeuSigma}, there is an analytic line bundle $L'$ on $X$ of degree $=\deg L$ such that the pullbacks of $L$ and $L'$ to $\wt X$ are isomorphic. We may therefore assume without loss of generality that $L$ is an analytic (hence, algebraic) line bundle.
		Let $f:X'\to X$ be any connected finite \'etale cover of degree $n$, then $f^{\ast}L$ has degree $n\deg L$. By Riemann--Hurwitz, $g(X')-1=n(g(X)-1)$. By Riemann--Roch,
		\[ \dim_K\rH^0(X',L)\geq \deg(f^{\ast}L)+1-g(X')=n(\deg L-g+1)\geq n.\]
		On the other hand, since $\wt X\to X'$ is a v-cover, we have
		$\rH^0(X',L)\subseteq \rH^0(\wt X,L)$.
		As the degree of connected finite \'etale covers of $X$ is unbounded, this combines to show that $\dim_K \rH^0(\wt X,L)=\infty$.
	\end{proof}
	\begin{lemma}\label{l:subbundle-profet}
		Let $V$ be a pro-finite-\'etale v-vector bundle on $X$. Let $W\subseteq V$ be any sub-v-vector bundle. Then $\deg(W)\leq 0$.
	\end{lemma}
	\begin{proof}
		If $X=\P^1$, then $\pi^\et_1(\P^1)=1$, so \Cref{p:Thm5.2-in-families} implies that $V$ is a  trivial vector bundle. By \Cref{t:canonicalHiggs}, it follows that $W$ is \'etale. As trivial bundles are semi-stable, this settles the case of $\P^1$.
		
		Hence we may assume $g\geq 1$. We can then make further reductions as in \cite[Proposition~10.4]{NarasimhanSeshadri}:
		Let $d=\rk W$, then $\wedge^dV$ is still pro-finite-\'etale. We may therefore replace $W\subseteq V$ by $\wedge^dW\subseteq \wedge^dV$ to assume that $W$ is a line bundle.
		Second, suppose that $\deg W>0$. As $V^{\otimes n}$ is pro-finite-\'etale for any $n\in \N$, we can then further replace $W\subseteq V$ by $W^{\otimes n}\subseteq V^{\otimes n}$ to assume that $\deg W>g-1$. Consider now the inclusion
		\[\rH^0(\wt X,W)\subseteq \rH^0(\wt X,V).\]
		Since $V$ is pro-finite-\'etale, the right hand side is a finite dimensional $K$-vector space by \Cref{p:wtXperf}.(2). But by \Cref{l:sections-LB-on-wtX}, the left hand side is infinite dimensional, a contradiction.
	\end{proof}
	
	To finish the proof of \Cref{p:image-sst-deg-0}, assume that there is a sub-Higgs-bundle $(N,\theta_{|N})\subseteq (E,\theta)$ of degree $>0$. Then by \Cref{l:sub-bundle-associated}, this is associated to a sub-v-vector bundle $W\subseteq V$. By \Cref{l:deg-of-associated-bundle}, we have $\deg(W)=\deg(N)>0$. By \Cref{l:subbundle-profet}, this contradicts the assumption that $V$ is pro-finite-\'etale.
		\end{proof}

	\newcommand{\etalchar}[1]{$^{#1}$}
	
\end{document}